\documentclass[10pt,reqno]{amsart}

\usepackage{verbatim}
\usepackage{graphicx}
\usepackage{amssymb}
\usepackage{enumerate}
\usepackage{multicol}
\usepackage{stmaryrd}

\usepackage{mathrsfs}

\newcommand{\gNorm}[1]{{\left\|\kern-0.24ex\left|{ #1 } \right|\kern-0.24ex\right\|}}

\newcommand{\DifferenceTensor}[1]{  {({}^{#1\!}D)} }

\newcommand{\dVol}[1]{ {dV_{#1} } }
\newcommand{\Nabla}[1]{ {({}^{#1}{\nabla})} }

\newcommand{\KN}{\oast}

\renewcommand{\bar}{\overline}

\renewcommand{\tilde}{\widetilde}

\renewcommand{\epsilon}{\varepsilon}

\newcommand{\D}{d}

\newcommand{\B}{\mathcal{H}}
\newcommand{\phg}{\textup{\normalfont phg}}
\newcommand{\WAH}{\mathscr{M}_\textup{weak}}

\newcommand{\bbR}{\mathbb{R}}

\newcommand{\bbC}{\mathbb{C}}
\newcommand{\bbH}{\mathbb{H}}

\renewcommand{\Re}{\mathrm{Re}}
\renewcommand{\subset}{\subseteq}
\renewcommand{\setminus}{\!\smallsetminus\!}

\renewcommand{\)}{\textup{)}}

\newcommand{\Defn}[1]{{\boldmath\it\bfseries #1}}

\DeclareMathOperator{\Riem}{Riem}
\DeclareMathOperator{\Ric}{Ric}
\DeclareMathOperator{\R}{R}
\DeclareMathOperator{\grad}{grad}
\DeclareMathOperator{\Hess}{Hess}

\DeclareMathOperator{\Div}{div}
\DeclareMathOperator{\tr}{tr}
\DeclareMathOperator{\supp}{supp}

\DeclareMathOperator{\Int}{int}
\DeclareMathOperator{\ran}{ran}
\DeclareMathOperator{\Id}{Id}

\theoremstyle{plain}
\newtheorem{theorem}{Theorem}
\newtheorem{lemma}[theorem]{Lemma}
\newtheorem{proposition}[theorem]{Proposition}
\newtheorem{corollary}[theorem]{Corollary}

\newtheorem{remark}[theorem]{Remark}

\newtheorem{assumption}{Assumption}

\numberwithin{theorem}{section}
\numberwithin{equation}{section}

\title{Weakly asymptotically hyperbolic manifolds}

\author[Allen -- Isenberg -- Lee -- Stavrov Allen]{Paul T.~Allen, James Isenberg, John M.~Lee, Iva Stavrov Allen}

\begin{document}

 \begin{abstract}
 We introduce a class of ``weakly asymptotically hyperbolic'' geometries whose sectional curvatures tend to $-1$ and are $C^0$, but are not necessarily $C^1$, conformally compact.
We subsequently investigate the rate at which curvature invariants decay at infinity, identifying a conformally invariant tensor which serves as an obstruction to ``higher order decay'' of the Riemann curvature operator.
Finally, we establish Fredholm results for geometric elliptic operators, extending the work of Rafe Mazzeo \cite{Mazzeo-Edge} and John M.~Lee \cite{Lee-FredholmOperators} to this setting.
As an application, we show that any weakly asymptotically hyperbolic metric is conformally related to a weakly asymptotically hyperbolic metric of constant negative scalar curvature. 
\end{abstract}

\maketitle


\section*{Introduction}

The mapping properties of elliptic operators on asymptotically hyperbolic manifolds have been studied in  \cite{Andersson-EllipticSystems}, \cite{Lee-FredholmOperators}, \cite{Mazzeo-Edge}, \cite{MazzeoMelrose-Meromorphic}, among others.
These studies have all required that the metric be conformally compact of at least class $C^2$; indeed in these works the notion of asymptotic hyperbolicity is defined in terms of conformal compactification.
However, it is not clear whether a complete manifold with asymptotically negative curvature necessarily admits such a compactification; to our knowledge, the best results available are those of \cite{Gicquaud-Conformal2} (see also \cite{BahuaudGicquaud-Conformal},\cite{HuQingShi-AH}), where it is shown that if the sectional curvatures of a complete manifold approach $-1$ to second order at infinity then the manifold is $C^{1,\beta}$ conformally compact for every $\beta\in (0,1)$.
(In fact, the work \cite{Bahuaud-Intrinsic} presents an example of a manifold for which the curvature  operator approaches 
the negative identity operator to first order, but for which no Lipschitz conformal compactification exists.)
These works are part of a body of evidence suggesting that for problems in geometric analysis in the asymptotically hyperbolic setting, it is desirable to have a theory applicable to metrics with sufficient ``interior'' regularity for PDE theory (such as interior elliptic regularity), but with somewhat limited regularity at the conformal boundary.

Our primary purpose here is to introduce a condition we call ``weakly asymptotically hyperbolic," which does not necessarily imply that the geometry is $C^1$ conformally compact, but under which we are nevertheless able to establish Fredholm results for geometric elliptic operators; see Theorem \ref{WeaklyFredholmTheorem}.
Roughly, a complete Riemannian metric is weakly asymptotically hyperbolic if the curvature operator tends to $-\Id$ at infinity, and if the metric is an element of certain weighted H\"older spaces; see \S\ref{Results} below for a formal definition and for additional details.
We emphasize that the definition is intrinsic in the sense that we do not assume {\it a priori} that the metric is conformally compact, 
but metrics that are weakly asymptotically hyperbolic do indeed admit Lipschitz-continuous conformal compactifications, thus excluding the example in \cite{Bahuaud-Intrinsic}.
We further remark that the class of weakly asymptotically hyperbolic metrics is considerably larger than, for example, the class of  asymptotically hyperbolic metrics with smooth conformal compactifications; this is due to the fact that smooth functions are not dense in the space of H\"older continuous functions.
(The closure of smooth functions with respect to the $C^{0,\alpha}$ norm is a proper subset of $C^{0,\alpha}$ called the ``little H\"older space.'')

In the first part of our work here we show several properties of weakly asymptotically hyperbolic metrics, followed by some results  that highlight the importance of the extrinsic curvature of the boundary, and which are in some sense complementary to those of \cite{AnderssonChruscielFriedrich} and \cite{Gicquaud-Conformal2}.
Under a slightly stronger regularity assumption, which implies that the metric is $C^{1,1}$ conformally compact but not necessarily $C^2$, we introduce a conformally invariant tensor that agrees with the trace-free extrinsic curvature along the boundary.
We show in Theorem \ref{FH} that if the scalar curvature of a weakly asymptotically hyperbolic metric approaches a constant at the ``second order'' rate of \cite{Gicquaud-Conformal2}, then the invariant tensor vanishes along the boundary if and only if the full curvature operator, or its derivative, vanishes along the boundary at the second order rate.

We then prove Fredholm results for geometric elliptic operators arising from weakly asymptotically hyperbolic metrics.
As an application, we prove that the Yamabe problem can be solved in this class of metrics, without loss of regularity; see Theorem \ref{YamabeTheorem}.
This extends the results of \cite{AnderssonChruscielFriedrich}, where  the case of smoothly conformally compact asymptotically hyperbolic metrics is considered.

We conclude this introduction by remarking that the class of weakly asymptotically hyperbolic metrics includes an important class of smooth metrics whose conformal compactifications are not smooth:~the polyhomogeneous metrics, for which the formal expansion along the conformal boundary involves powers of both the distance to the boundary and its logarithm. 
Such boundary regularity  is, in fact, a feature typical of problems involving the much more general class of elliptic edge operators developed in \cite{Mazzeo-Edge}, and such metrics arise naturally in a variety of contexts; see \cite{AnderssonChrusciel-Dissertationes}, \cite{ChruscielDelayLeeSkinner}, \cite{FeffermanGraham}, among others.
For completeness, and to display the manner in which the present work is situated among the existing literature, we include  an appendix containing a self-contained account of the boundary regularity of elliptic problems in the polyhomogeneous setting.
We emphasize that the polyhomogeneity results are not new, but follow from a straightforward adaptation of results in \cite{Mazzeo-Edge}; see also \cite{AnderssonChrusciel-Dissertationes}.
As the results in the appendix don't appear in the literature in the form presented here, however, we take this opportunity to present a self-contained exposition.

\subsection*{Acknowledgements}
We thank Michael Eastwood for bringing the works \cite{BairdEastwood} and \cite{EastwoodGraham} to our attention.
We furthermore thank Eric Bahuaud for helpful conversations.
This work was partially supported by NSF grants PHYS-1306441 and DMS-63431.

\section{Statement of results}
\label{Results}

Let $\bar M$ be a smooth, compact $(n+1)$-dimensional manifold with boundary, with $n\geq 1$; let $M$ be the interior of $\bar M$ and denote by $\partial M$ the boundary of $\bar M$.
Let $\rho\colon \bar M\to [0,\infty)$ be a smooth function with $\rho^{-1}(0) = \partial M$ and $\D\rho \neq 0$ on $\partial M$; such a function is called a \Defn{defining function}.
A Riemannian metric $g$ on $M$ is called \Defn{conformally compact} if the metric $\bar g := \rho^2g$ extends continuously to a (non-degenerate) metric on $\bar M$.
A conformally compact metric $g$ is said to be \Defn{asymptotically hyperbolic of class $C^{l,\beta}$} if $\bar g$ is of class $C^{l,\beta}$ on $\bar M$ and $|\D\rho|_{\bar g} =1$ on $\partial M$.
In view of the notion of weakly asymptotically hyperbolic introduced below, we henceforth refer to asymptotically hyperbolic metrics of class $C^{l,\beta}$ as \Defn{strongly asymptotically hyperbolic}.

The definition of strongly asymptotically hyperbolic metrics  is motivated by the fact that if $\bar g$ extends to a metric of class $C^2$ on $\bar M$, then the sectional curvatures of $(M,g)$ approach $-|\D\rho|^2_{\bar g}$ as $\rho \to 0$.
To see this, consider the ``raised index'' version of the Kulkarni-Nomizu product, defined as follows: 
For $(1,1)$-tensor fields $u$ and $v$, we define  $u \KN v\colon \Lambda^2(TM) \to \Lambda^2(TM)$ by setting
\begin{equation*}
(u\KN v) (x\wedge y) = u(x) \wedge v(y) - u(y) \wedge v(x)
\end{equation*}
for decomposables, and then extending the map to all of $\Lambda^2(TM)$ by linearity.
In coordinates, we have the expressions 
\begin{equation*}
(u \KN v)_{ij}^{kl} = \tfrac{1}{2}\left(u_i^k v_j^l + u_j^l v_i^k - u_i^l v_j^k - u_j^k v_i^l\right), 
\end{equation*}
$\Id_{ij}^{kl} = (\delta\KN\delta)_{ij}^{kl} = \delta_i^k \delta_j^l - \delta_i^l\delta_j^k $,
 and $((\Hess_{\bar g}\rho)^\sharp)_i^j = \bar g^{jk}(\Hess_{\bar g}\rho)_{ik}$.
Here $\Id$ is the identity  of $\Lambda^2(TM)$, considered as a $(2,2)$ tensor, and $\delta$ is the identity of $TM$, viewed as a $(1,1)$ tensor.

The Riemann curvature operator $\Riem[g]\colon\Lambda^2(TM)\to \Lambda^2(TM)$ is related to that of $\bar g$ by
\begin{equation}
\label{FirstRiemID}
\Riem[g] 
= -|\D\rho|^2_{\bar g} \Id
+ 2\rho \, \delta \KN (\Hess_{\bar g}\rho)^\sharp
+ \rho^2 \Riem[\bar g],
\end{equation}
from which we immediately read off the asymptotic behavior of the sectional curvatures.
Contraction of \eqref{FirstRiemID} yields the following expressions for the Ricci operator, viewed as a $(1,1)$ tensor, and the scalar curvature:
\begin{align}
\label{RicciID}
\Ric[g] &= -n\,|\D\rho|^2_{\bar g} \delta 
+ \rho (\Delta_{\bar g}\rho) \delta 
+ (n-1)\rho (\Hess_{\bar g}\rho)^\sharp
+ \rho^2\Ric[\bar g] ,
\\
\label{ScalarID}
\R[g] &= -n(n+1)|\D\rho|_{\bar g}^2 + 2n \rho(\Delta_{\bar g}\rho) + \rho^2 \R[\bar g].
\end{align}

In order to describe the boundary regularity condition in our definition of 
weakly asymptotically hyperbolic metrics, we introduce several notations.
First, $C^{k,\alpha}(M)$ is an intrinsic H\"older space of tensors on $M$,
and similarly $H^{k,p}(M)$ is an intrinsic Sobolev space; see 
\S\ref{RegularityClasses} for definitions.
We also use weighted spaces $C^{k,\alpha}_\delta(M) = \rho^\delta C^{k,\alpha}(M)$
and $H^{k,p}_\delta(M) = \rho^\delta H^{k,p}(M)$.

There is an alternative characterization of these spaces in terms
of Lie derivatives that helps to shed
light on them. 
Let $\mathscr V = \mathfrak X(\bar M)$, the space of smooth
vector fields on $\bar M$, and let $\mathscr V_0$ be the subspace
of $\mathscr V$ consisting of vector fields that vanish on $\partial M$.
If a metric $g\in C^{k,\alpha}(M)$, then $\bar g = \rho^2 g\in C^{k,\alpha}_2(M)$,
which is equivalent to saying that 
$\mathcal L_{X_1}\dots \mathcal L_{X_j} \bar g\in C^{k-j,\alpha}_2(M)$
whenever $0\le j\le k$ and $X_1,\dots,X_j\in \mathscr V_0$.
On the other hand, if $g$ has a $C^{k,\alpha}$ conformal compactification,
then
$\mathcal L_{X_1}\dots \mathcal L_{X_j} \bar g\in C^{k-j,\alpha}_2(M)$
for any vector fields $X_1,\dots,X_j\in \mathscr V$, not just ones
that vanish at the boundary.

The purpose of this paper is to show that 
much of the theory of elliptic operators on conformally compact 
manifolds can be extended to metrics satisfyng the following boundary regularity condition, which is much weaker than being $C^{2,\alpha}$ conformally compact:
\begin{equation}
\label{FirstTriple}
\bar g\in C^{k,\alpha}_2(M)
\quad\text{ and }\quad
\mathcal L_X\bar g\in C^{k-1,\alpha}_2(M) \quad \text{for all $X\in \mathscr V$.}
\end{equation}

We remark that these regularity conditions imply that $\bar g$ extends to a Lipschitz continuous metric on $\bar M$; see Lemma \ref{lemma:properties-of-scriptC}\eqref{part:intermediate} below.
But even if \eqref{FirstTriple} holds for all $k$, it need not be the case that $\bar g$ extend to a $C^1$ metric on $\bar M$; see Remark \ref{Remark-TripleExample}.

Our first theorem shows that, just as for $C^2$ conformally compact metrics,
the asymptotic behavior of the curvature of a metric satisfying \eqref{FirstTriple}
is determined by the value of $|d\rho|_{\bar g}$ along $\partial M$.

\begin{theorem}
\label{thm:WAH-TFAE}
Let $k\geq 2$ and $\alpha\in [0,1)$, and let $g = \rho^{-2}\bar g$ be a Riemannian metric on $M$ satisfying 
\eqref{FirstTriple}.
The following are equivalent:
\begin{enumerate}
\item\label{NEWWAH1} $\Riem[g] \to -\Id$ as $\rho \to 0$.

\item\label{NEWWAH2} $\Ric[g]\to -n \delta$ as $\rho \to 0$.

\item\label{NEWWAH3} $\R[g]\to -n(n+1)$ as $\rho \to 0$.

\item\label{NEWWAH4} $|d\rho|_{\bar g} = 1$ on $\partial M$.

\end{enumerate}

\end{theorem}

For $k\geq 2$ we define a metric  $g$ on $M$ to be \Defn{weakly $C^{k,\alpha}$ asymptotically hyperbolic} if
$g$ is conformally compact and $\bar g = \rho^2 g$ satisfies the regularity conditions \eqref{FirstTriple}
and one (and hence all) of the conditions \eqref{NEWWAH1}--\eqref{NEWWAH4} in the above theorem.
We denote by $\WAH^{k,\alpha;1}$ the collection of all weakly $C^{k,\alpha}$ asymptotically hyperbolic metrics on $M$; here the superscript $1$ indicates that we have imposed the improved regularity condition on one derivative of the metric.

\begin{theorem}\label{thm:properties-of-WAH}
Suppose $g\in \WAH^{k,\alpha;1}$  for $k\geq 2$ and $\alpha \in [0,1)$.
Then we have the following:
\begin{enumerate}
\item\label{PropWAH2}
 $\Riem[g]+\Id\in  C^{k-2,\alpha}_1(M)$.

\item\label{PropWAH3}
 $\Ric[g]+n \delta \in  C^{k-2,\alpha}_1(M)$.

\item\label{PropWAH4}
 $\R[g]+n(n+1)\in  C^{k-2,\alpha}_1(M)$.

\item\label{PropWAH1}
  $|\D\rho|_{\bar g}^2 - 1\in  C^{k,\alpha}_1(M)$.
  
\item\label{PropWAH5}
For all $1\le j\le k-2$, the tensor $\rho^j ({}^g\nabla)^j\Riem[g]$ extends continuously
to $\overline M$ and is $\mathcal O(\rho)$ as $\rho\to 0$.
\end{enumerate}
\end{theorem}

Before introducing further results concerning the decay of curvature at infinity, let us recall the results of Andersson, Chru\'sciel, and Friedrich 
\cite{AnderssonChruscielFriedrich}, which show that if $\bar g \in C^\infty(\bar M)$ there exist smooth functions $\psi, r\in C^\infty(\bar M)$, with $\psi >0$ and $\psi =1$ along $\partial M$, such that 
\begin{equation*}
\R[\psi^{4/(n-1)} g] = -n(n+1) + \rho^{n+1} r.
\end{equation*}
They further show that unless $r=0$ on $\partial M$, it is not possible to make the scalar curvature approach $-n(n+1)$ to higher order
with a conformal factor in
$C^\infty(\bar M)$. 
In particular, the metric $g$ is conformally related to a smoothly conformally compact metric of constant scalar curvature, and thus the Yamabe problem admits a smoothly conformally compact solution, if and only if $r=0$ on $\partial M$.
As well, they show that if the dimension of $M$ is three, then $r=0$ if and only if the trace-free part of the second fundamental form induced on $\partial M$ by $\bar g$ vanishes.

Our next results are somewhat complementary to the results in \cite{AnderssonChruscielFriedrich} in that they highlight the importance of the traceless part of the extrinsic curvature of the conformal boundary.
First, using \eqref{RicciID} we  write \eqref{FirstRiemID} as
\begin{multline}
\label{SecondRiemID}
\Riem[g] 
+\Id
=  \big(\R[g]+n(n+1)\big)\frac{1}{n(n+1)} \Id
\\
+ 2\rho \,\delta \KN \left(\Hess_{\bar g}\rho - \frac{1}{n+1}(\Delta_{\bar g}\rho)\bar g\right)^\sharp
\\
+ \rho^2 \left( \Riem[\bar g] - \frac{1}{n(n+1)}\R[\bar g]\Id \right);
\end{multline}
we emphasize that the full contraction of the identity operator is $n(n+1)$ and that, as before, the musical isomorphism is with respect to $\bar g$.

From \eqref{SecondRiemID} we see that the rate at which the curvature operator $\Riem[g]$ approaches $-\Id$ is governed by the rate at which the scalar curvature $\R[g]$ approaches $-n(n+1)$, and by the extent to which the trace-free Hessian of $\rho$, with respect to $\bar g$, vanishes as $\rho \to 0$.

We are able to obtain more refined results concerning the asymptotic behavior of the curvature provided we assume slightly more regularity than is provided by the weakly asymptotically hyperbolic condition.
The reason is that for $g\in \WAH^{k,\alpha;1}$  one can only conclude that $\big\lvert\Riem[\bar g]\big\rvert_{\bar h}=\mathcal O(\rho^{-1})$ as $\rho \to 0$, but  under a stronger regularity hypothesis we can conclude that the norm of the curvature operator $\Riem[\bar g]$ is bounded; see Lemma \ref{BarRiemLemma}.
Consequently, we introduce the class $\WAH^{k,\alpha;2}$ of metrics $g\in \WAH^{k,\alpha;1}$ such that
\begin{equation}
\begin{gathered}
\label{FirstHAHNorm}
\bar g\in  C^{k,\alpha}_2(M),
\quad
\mathcal L_{X_1} \bar g \in  C^{k-1,\alpha}_2(M),
\quad
\mathcal L_{X_1}\mathcal L_{X_2} \bar g \in C^{k-2,\alpha}_2(M)
\\
\text{ for all } X_1,X_2\in \mathscr V.
\end{gathered}
\end{equation}
We note that if $g\in \WAH^{k,\alpha;2}$, then $\bar g$ extends to a metric of class $C^{1,1}$ on $\bar M$, but not necessarily to a metric of class $C^2$. 
The next theorem gives additional properties of metrics in $\WAH^{k,\alpha;2}$.

\begin{theorem}
\label{thm:HAH-TFAE}
Let $k\geq 2$ and $\alpha \in [0,1)$, and suppose that $g\in \WAH^{k,\alpha;2}$.
Then the following are equivalent:
\begin{enumerate}
\item\label{HAH(2)}
$|d\rho|^2_{\bar g}-1-\dfrac{2}{n+1}\rho\Delta_{\bar g}\rho\in  C^{k-1,\alpha}_2(M)$.\vspace{.5ex}

\item\label{HAH(3)}
$|d\rho|^2_{\bar g}-1-\dfrac{2}{n+1}\rho\Delta_{\bar g}\rho=O(\rho^2)$ as $\rho\to 0$.

\item\label{HAH(4)} $\R[g]+n(n+1)\in  C^{k-2,\alpha}_2(M)$.

\item\label{HAH(0)}
$\R[g]+n(n+1)=O(\rho^2)$ as $\rho\to 0$.

\end{enumerate}
\end{theorem}
\noindent
The proofs of Theorems \ref{thm:WAH-TFAE}--\ref{thm:HAH-TFAE} appear in \S\ref{Section-WAH}.

In \S\ref{IvaTensorSection} below, we define a tensor $\B_{\bar g}(\rho)$ that is a conformally invariant version of the trace-free Hessian of $\rho$.
It follows from Theorem \ref{thm:properties-of-WAH} that if $g$ is in  $\WAH^{k,\alpha;1}$, then the scalar curvature $\R[g]$ satisfies $\R[g]+n(n+1)=\mathcal O(\rho)$ as $\rho \to 0$.
If we assume that 
$g\in \WAH^{k,\alpha;2}$ and in addition that 
$R[g]+n(n+1) \in  C^{k-2,\alpha}_2(M)$, we have
\begin{equation}
\label{BatB}
\left.\B_{\bar g}(\rho)\right|_{\partial M} 
=
 \left[\Hess_{\bar g}\rho - \frac{1}{n+1}(\Delta_{\bar g}\rho)\bar g\right]_{\partial M};
\end{equation}
see Proposition \ref{B-DefiningFunctionProperties}.
We remark that while we have independently constructed the tensor $\B_{\bar g}(\rho)$, it has since come to our attention that a general procedure exists for constructing such invariants; see \cite{BairdEastwood}, \cite{EastwoodGraham}.

The following theorem shows that if the scalar curvature of a metric in $\WAH^{k,\alpha;2}$ has faster decay, then  the tensor $\B_{\bar g}(\rho)$ serves as an obstruction to faster decay of the full curvature operator to $-\Id$.
\begin{theorem}
\label{FH}
Suppose that $g\in \WAH^{k,\alpha;2}$ for $k\geq 2$ and $\alpha \in [0,1)$.
If $\R[g] + {n(n+1)}\in  C^{k-2,\alpha}_2(M)$, then the following are equivalent:
\begin{enumerate}
\item\label{FH1}
 $\Riem[g]+\Id\in  C^{k-2,\alpha}_2(M)$,

\item\label{FH2}
 $\Ric[g]+n\,\delta\in  C^{k-2,\alpha}_2(M)$,

\item\label{FH3}
 $\B_{\bar g}(\rho)=0$ along $\partial M$, and

\item\label{FH4}
 $\B_{\bar g}(\rho)\in C^{k-1, \alpha}_3(M)$.

\end{enumerate}

\end{theorem}

We emphasize that conditions \eqref{FH3} and \eqref{FH4} in Theorem \ref{FH} are manifestly conformally invariant.
We furthermore note that it is an immediate consequence of Theorem \ref{FH} that if the metric $g$ is Einstein, then the tensor $\B_{\bar g}(\rho)$ vanishes at $\partial M$.
The proof of Theorem \ref{FH} can be found in \S\ref{IvaTensorSection}.

We also note that the tensor  $\B_{\bar g}(\rho)$  has further applications in general relativity, where it gives rise to a conformally invariant description of the ``shear-free condition'' for asymptotically hyperbolic solutions to the Einstein constraint equations.
In this context, the conformal invariance of $\B_{\bar g}(\rho)$ is particularly useful for constructing solutions to the constraint equations via conformal deformation; construction of shear-free solutions using the tensor $\B_{\bar g}(\rho)$ is carried out in \cite{AHEM-Preliminary}.

\begin{center}
---------------------------
\end{center}

One motivation for defining the weakly asymptotically hyperbolic condition is to establish Fredholm results for geometric elliptic operators arising from a metric $g$ that is sufficiently regular on the interior $M$ for establishing interior elliptic regularity results, but whose conformal compactification $\bar g$ is less regular at $\partial M$ than is typically assumed in the literature.
Such metrics include the polyhomogeneous metrics; see Appendix \ref{appendix} for a detailed discussion of polyhomogeneity.

We now consider a linear elliptic operator $\mathcal P$ acting on sections of a tensor bundle $E$ having weight $r$.
(The \Defn{weight} of a tensor bundle is the covariant rank less the contravariant rank.)
Following \cite{Lee-FredholmOperators}, we make the following assumptions on $\mathcal P$.

\setcounter{assumption}{15}
\begin{assumption}
\label{Assume-P}
We assume $\mathcal P = \mathcal P[g]$ is a second-order linear elliptic operator acting on sections of a tensor bundle $E$.
Furthermore
\begin{enumerate}
\item 
\label{Assume-P-Basic}
We assume that $\mathcal P$ is \Defn{geometric} in the sense of \cite{Lee-FredholmOperators}: In any coordinate frame the components of $\mathcal P u$ are linear functions of $u$ and its derivatives, whose coefficients are universal polynomials in the components of $g$, their partial derivatives, and $\sqrt{\det g_{ij}}$, such that the coefficient of the $j$th derivative of $u$ involves no more than $2-j$ derivatives of the metric.

\item We assume that $\mathcal P$ is formally self-adjoint, and that there is a compact set $K\subset M$ and a constant $C>0$ such that
\begin{equation}
\label{BasicL2Estimate}
\|u\|_{L^2(M)}\leq C\|\mathcal P u\|_{L^2(M)} 
\quad\text{ for all }\quad
 u\in C^\infty_c(M\setminus K).
\end{equation} 

\end{enumerate}
\end{assumption}

\begin{remark}
\label{Generalize}

It is possible to weaken the hypothesis that $\mathcal P$ be geometric in the sense described above.
For example, Theorem \ref{WeaklyFredholmTheorem} below easily generalizes to operators $\mathcal P = \mathcal P[g,\rho]$ whose coefficients, in any smooth chart, are universal polynomials in both $\rho$ and components of $g$, and their derivatives.
\end{remark}

If $(M,g)$ is strongly asymptotically hyperbolic of class $C^{k,\alpha}$ for $k\geq 2$, then Lemma 4.1 of \cite{Lee-FredholmOperators} shows that operators $\mathcal P$ satisfying Assumption \ref{Assume-P} are uniformly degenerate at $\partial M$, meaning that in background coordinates (see \S\ref{RegularityClasses}) we may write
\begin{equation}
\label{OperatorInCoordinates}
\mathcal P = a^{ij}(\rho\partial_i)(\rho\partial_j) + b^i\rho\partial_i + c,
\end{equation}
where the matrix-valued functions $a^{ij}$, $b^i$, $c$ extend continuously to $\bar M$.
If $g\in \WAH^{k,\alpha;1}$ this remains true; see Lemma \ref{I-Lemma} below.

In the strongly asymptotically hyperbolic setting, it is known that the mapping properties of operators $\mathcal P$ satisfying Assumption \ref{Assume-P} can, to a great extent, be understood via the mapping properties of the \Defn{indicial map} 
$$I_s(\mathcal P)\colon \left.(E\otimes \mathbb C)\right|_{\partial M}\to \left.(E\otimes \mathbb C)\right|_{\partial M},
$$ defined for each $s\in \mathbb C$  by
\begin{equation}
\label{DefineIndicialMap}
I_s(\mathcal P)\bar u = \left.\rho^{-s}\mathcal P (\rho^s \bar u)\right|_{\rho=0}.
\end{equation}
In Lemma \ref{I-Lemma} we show that the indicial map is still well-defined in the case that $\mathcal P$ arises from a weakly asymptotically hyperbolic metric, and that in this weaker setting $I_s(\mathcal P)$ is a $C^0$ bundle map.

In \cite{Lee-FredholmOperators} it is shown that the  \Defn{characteristic exponents} of $\mathcal P$, defined as the set of $s\in \bbC$ for which $I_s(\mathcal P)$ has nontrivial kernel at some point on $\partial M$, are located symmetrically around the line $\Re(s)=n/2-r$, where $r$ is the weight of the tensor bundle $E$. 
Of particular relevance here is the distance between this line and the closest characteristic exponent, called the \Defn{indicial radius} and denoted by $R$.

The following theorem shows that the affirmative Fredholm results of \cite{Lee-FredholmOperators} hold in the weakly asymptotically hyperbolic setting.

\begin{theorem}
\label{WeaklyFredholmTheorem}
Suppose $g\in \WAH^{l,\beta;1}$ for some $l\geq 2$ and $\mathcal P$ satisfies Assumption \ref{Assume-P}. 
Then the indicial radius $R$ of $\mathcal P$ is positive. 
Furthermore,
\begin{enumerate}
\item if $\beta\in [0,1)$, then
$$\mathcal P\colon H^{k,p}_\delta(M)\to H^{k-2,p}_\delta(M)$$ 
is Fredholm for $1<p<\infty$, $2\le k\le l$, and $|\delta+\tfrac{n}{p}-\tfrac{n}{2}|<R$; and

\item if $\beta\in (0,1)$, then
$$\mathcal P\colon C^{k,\alpha}_\delta(M)\to C^{k-2,\alpha}_\delta(M)$$ is Fredholm for $0<\alpha<1$, $2< k+\alpha\leq l+\beta$, and $|\delta-\tfrac{n}{2}|<R$.
\end{enumerate}
In both cases the operators are of index zero, and the kernel is equal to the $L^2$ kernel of $\mathcal P $. 
\end{theorem}
\noindent
The proof of Theorem \ref{WeaklyFredholmTheorem} consists of adapting results of \cite{Lee-FredholmOperators} to the weakly  asymptotically hyperbolic setting, and is the content of \S\ref{Fredholm} below.

To further illustrate the utility of the weakly asymptotically hyperbolic condition, we now consider the Yamabe problem, which is the question of whether an asymptotically hyperbolic metric can be conformally deformed to another such metric of constant scalar curvature.
In the case that $\bar g\in C^\infty(\bar M)$, it is known that there exists a smooth, positive function $\phi \in C^\infty(M)$ such that the scalar curvature of $\phi^{4/(n-1)} g$ is identically $-n(n+1)$; see \cite[Theorem 1.2]{AnderssonChruscielFriedrich}, as well as \cite{AnderssonChrusciel-Dissertationes}.
In the weakly asymptotically hyperbolic setting, we prove the following.

\begin{theorem}
\label{YamabeTheorem}
Suppose $g\in \WAH^{k,\alpha;1}$ for $k\geq 2$ and $\alpha \in (0,1)$.
Then there exists a unique positive function $\phi$ with $\phi -1\in C^{k,\alpha}_1(M)$ such that $\hat g = \phi^{4/(n-1)}g\in \WAH^{k,\alpha;1}$ and $\R[\hat g] = -n(n+1)$.
Furthermore, if $g\in \WAH^{k,\alpha;2}$, then $\hat g\in \WAH^{k,\alpha;2}$.
If $g$ is also polyhomogeneous, then $\hat g$ is polyhomogeneous as well.
\end{theorem}

\noindent
The proof of Theorem \ref{YamabeTheorem}, which appears in \S\ref{Yamabe},  
 relies on the identity
\begin{equation}
\label{ChangeSC}
\R[\phi^{4/(n-1)}g] =\left( -\frac{4n}{n-1}\Delta_g\phi + \R[g]\phi\right)\phi^{-(n+3)/(n-1)},
\end{equation}
where our sign convention for the Laplacian is $\Delta_g\phi = \tr_g \Hess_g\phi$.
Thus $\hat g = \phi^{4/(n-1)}g$ has constant scalar curvature $-n(n+1)$ if $\phi$ satisfies
\begin{equation}
\label{YP}
\Delta_g\phi - \frac{n-1}{4n}\R[g]\phi = \frac{n^2-1}{4}\,\phi^{(n+3)/(n-1)},
\qquad \left.\phi\right|_{\partial M} =1,
\qquad \phi>0.
\end{equation}
We show the existence of a function $\phi$ satisfying \eqref{YP} in \S\ref{Yamabe}.

Combining Theorem \ref{YamabeTheorem} with Theorem \ref{FH}, we observe the following: If $g\in \WAH^{k,\alpha;2}$, then the tensor $\B_{\bar g}(\rho)$ determines whether $g$ is conformally related to a metric in $\WAH^{k,\alpha;2}$ whose curvature operator tends towards $-\Id$ to higher order.

\section{Regularity classes}
\label{RegularityClasses}

In this section we define weighted H\"older and Sobolev spaces of geometric tensor fields on $M$, and relate them to the construction given in \cite{Lee-FredholmOperators}.
While the definitions of these spaces are independent of any Riemannian structure, it is often convenient to work with equivalent norms defined using a background metric $h$ introduced below.
Some of our results also concern polyhomogeneous tensor fields.
We furthermore refer the reader to \S\ref{S-DefinePolyhomogeneity} for a careful definition of $C^k_\phg(\bar M)$, the class of polyhomogeneous tensor fields on $M$ which extend to fields of class $C^k$ on $\bar M$.

In order to construct H\"older and Sobolev spaces on $M$, we introduce a collection of coordinate charts covering a neighborhood of $\partial M$ in $\bar M$ as follows.
Choose a collar neighborhood $\underline{\mathcal C}$ of $\partial M$ in $\bar M$ and a diffeomorphism $\underline{\mathcal C}\to \partial M \times [0,\rho_*)$ whose last coordinate function is $\rho$; for convenience we hereafter implicitly identify $\underline{\mathcal C}$ with $\partial M\times [0,\rho_*)$.
For any $a\in (0,\rho_*]$, denote by $\underline{\mathcal C}_a $ the subset $\partial M \times [0,a)$, and define $\mathcal C = \Int(\underline{\mathcal C})\approx \partial M \times (0,\rho_*)$ and $\mathcal C_a = \Int(\underline{\mathcal C}_a)\approx \partial M \times (0,a)$.

Fix a finite collection of coordinate charts for $\partial M$ such that for each $(U,\theta)$ in the collection, $\theta$ extends smoothly to  a coordinate chart containing $\bar U$.
For each $(U,\theta)$ we extend $\theta$ to $\underline{\mathcal U} := U\times[0,\rho_*)$ by declaring it to be independent of $\rho$ and define coordinates $\Theta = (\theta,\rho)$ on $\underline{\mathcal U}$.
Following the nomenclature of \cite{Lee-FredholmOperators}, we refer to $\Theta$ as \Defn{background coordinates}.
For any $k\in \mathbb N_0$ and $\alpha\in [0,1)$ we define the H\"older spaces $C^{k,\alpha}(\bar M)$ using these background coordinate charts together with a finite number of charts covering the complement of $\underline{\mathcal C}$.

We furthermore  use the coordinates $\Theta$ to identify $\underline{\mathcal U}$ and $\mathcal U:=\Int \underline{\mathcal U}$ with subsets of the half space $\mathbb R^{n}\times[0,\infty)$.
These identifications allow one to compare the geometry of $(M,g)$ near $\partial M$ to that of hyperbolic space; to make this precise we use the following construction from \cite{Lee-FredholmOperators}.

Let $(\mathbb H,\breve g)$ be the upper half-space model of $(n+1)$-dimensional hyperbolic space, with coordinates $(x,y) = (x^1,\dots, x^n,y)\in \mathbb R^n\times (0,\infty)$ and with the hyperbolic metric $\breve g=y^{-2}((dx^1)^2+\dots+(dx^n)^2+dy^2)$.
For any $r>0$, define $B_r\subset\bbH$ to be the ball of radius $r$, with respect to $\breve g$, centered at $(0,\dots, 0,1)$.
Using background coordinates to identify subsets of $\mathcal U$ with subsets of $\mathbb R^{n+1}$, we may for each point $p_0 = (\theta_0,\rho_0)\in \mathcal C$ construct a  \Defn{M\"obius parametrization} $\Phi \colon B_2\to M$ centered at $p_0$ by $\Phi(x,y) = (\theta_0 + \rho_0x, \rho_0y)$.
(The complement of $\mathcal C$ in $M$, which is compact, we also cover by finitely many parametrizations $B_2\to M$, which we include in the collection of M\"obius parametrizations.)
We fix countably many M\"obius parametrizations $\Phi_i$ such that $\{\Phi_i(B_1)\}$ covers $M$ and $\{\Phi_i(B_2)\}$ is uniformly locally finite.

We define the H\"older norm $\| u\|_{C^{k,\alpha}(M;E)}$ of a section $u$ of a tensor bundle $E$ by
\begin{equation}
\label{DefineHolderNorms}
\| u\|_{C^{k,\alpha}(M;E)} := \sup_i \| \Phi_i^*u\|_{C^{k,\alpha}(B_2)};
\end{equation}
 the H\"older space $C^{k,\alpha}(M;E)$ is the space of sections for which this norm is finite.
For $\delta\in \mathbb R$, we define the weighted H\"older spaces by $C^{k,\alpha}_\delta(M;E) = \rho^\delta C^{k,\alpha}(M;E)$ using the norms
\begin{equation}
\label{DefineWeightedHolderNorms}
\| u\|_{C^{k,\alpha}_\delta(M;E)} = \| \rho^{-\delta} u\|_{C^{k,\alpha}(M;E)}.
\end{equation}
The Sobolev spaces ${H^{k,p}(M;E)}$ are defined analogously; for $k\in \mathbb N_0$  and $p\in (1,\infty)$ we have
\begin{equation*}
\|u\|_{H^{k,p}(M;E)}^p 
= \sum_i \|\Phi_i^* u\|^p_{H^{k,p}(B_2)}.
\end{equation*}

As defined here, the H\"older and Sobolev norms are independent of any Riemannian structure on $M$.
To simplify the analysis below, we fix a smooth ($C^\infty$) background metric $\bar h$ on $\bar M$ such that $|\D\rho|_{\bar h} =1$ along $\partial M$, and let $h = \rho^{-2}\bar h$ be the corresponding asymptotically hyperbolic metric on $M$.
Throughout the remainder of this paper we adopt the following convention:
\begin{equation*}
\text{
\emph{
$\bar\nabla$ and $\nabla$ denote the Levi-Civita connections of $\bar h$ and $h$, respectively.
}
}
\end{equation*}

A detailed account of H\"older and Sobolev spaces, including various embeddings and equivalent norms that make use of a sufficiently regular asymptotically hyperbolic metric and its Levi-Civita connection, is given in Chapter 3 of \cite{Lee-FredholmOperators}.
In particular, the background metric $h$ gives rise to the following norm equivalences:
\begin{equation}
\label{BasicHolderNormEquivalences}
\frac{1}{C}\|u\|_{C^{k,\alpha}(M;E)}
\leq
\sum_{0\leq j\leq k} \sup_M |\nabla^j u|_h
+ \|\nabla^ku\|_{C^{0,\alpha}(M;E)}
\leq 
C\|u\|_{C^{k,\alpha}(M;E)}
\end{equation}
and
\begin{equation}
\label{BasicSobolevNormEquivalences}
\begin{gathered}
\frac{1}{C}\|u\|_{H^{k,p}(M;E)}^p
\leq
\sum_{0\leq j \leq k} \int_M |\nabla^j u|_h^p\, \D V_h
\leq
C \|u\|_{H^{k,p}(M;E)}^p.
\end{gathered}
\end{equation}
Note that \cite{Lee-FredholmOperators} contains a small error; see Appendix \ref{CorrectFred} for a description of the error and necessary corrections.

We record the following elementary facts about H\"older spaces on $M$; recall that the {weight $r$ of a tensor bundle} is its covariant rank less its contravariant rank.

\begin{lemma}[Lemmas 3.3, 3.6, and 3.7 of \cite{Lee-FredholmOperators}]
\label{DLemma-ListOfHolderFacts}
Suppose $\bar h$ is a smooth metric on $\bar M$ as described above.

\begin{enumerate}
\item 
\label{DLemma-Include}
If $E$ is a geometric tensor bundle of weight $r$ over $(\bar M,\bar h)$, and if $\alpha\in (0,1)$ and $k\in \mathbb N_0$, then the following inclusions are continuous
\begin{gather*}
C^{k,\alpha}(\bar M;E) \hookrightarrow C^{k,\alpha}_r(M;E)
\\
C^{k,\alpha}_{k+\alpha+r}(M;E)\hookrightarrow C^{k,\alpha}(\bar M;E).
\end{gather*}
Note that the first inclusion holds for $\alpha\in [0,1)$.

\item 
\label{HolderMultiplication}
Let $E_1,E_2$ be geometric tensor bundles over $(M,h)$.
For all $\alpha\in [0,1)$, $k\in \mathbb N_0$, and $\delta_1,\delta_2\in \mathbb R$, the pointwise tensor product is a continuous map
\begin{equation*}
C^{k,\alpha}_{\delta_1}(M;E_1)
\times
C^{k,\alpha}_{\delta_2}(M;E_2)
\to
C^{k,\alpha}_{\delta_1+\delta_2}(M;E_1\otimes E_2).
\end{equation*}

\item We have $\D\rho\in C^{k,\alpha}_1(M;TM^*)$ for all $k\in \mathbb N_0$ and $\alpha\in [0,1)$.

\item
\label{DLemma-DifferenceTensorH}
The difference tensor $D = \nabla - \bar\nabla$ is in $C^{k,\alpha}_0(M;T^2M^*\otimes TM)$ for all $k\in \mathbb N_0$ and $\alpha\in [0,1)$, and therefore $\bar\nabla\colon C^{k+1,\alpha}_\delta(M;E)\to C^{k,\alpha}_\delta(M;E\otimes TM^*)$.

\end{enumerate}

\end{lemma}

The weight of a tensor bundle  is important for understanding the behavior of sections near $\partial M$: If $u$ is a section of a tensor bundle $E$ with weight $r$, then $|u|_h = \rho^r |u|_{\bar h}$.
For notational convenience, however, we frequently omit explicit reference to the relevant tensor bundle, writing $\|u\|_{C^{k,\alpha}_\delta(M)}$ for $\|u\|_{C^{k,\alpha}_\delta(M;E)}$, etc.
We nevertheless encourage the reader to be mindful of the weight of the relevant bundle.

In preparation for a discussion of the properties of weakly asymptotically hyperbolic metrics, we introduce spaces of tensor fields with additional regularity near the boundary.
Let $k\geq 0$, $\alpha\in [0,1)$, and $0\leq m \leq k$.
By definition, a tensor field $u$ of weight $r$ is in $\mathscr C^{k,\alpha;m}(M)$ if for all $0\leq j\leq m$  we have
\begin{equation*}
\mathcal L_{X_1}\dots\mathcal L_{X_j} u\in C^{k-j,\alpha}_r(M)
\quad\text{ for all }X_1,\dots,X_j\in \mathscr V.
\end{equation*}
(Closely related spaces, in which the additional derivatives are taken
only with respect to vector fields tangent to the boundary,
have been considered by many authors, and we use such spaces in 
Appendix \ref{appendix} for proving polyhomogeneity results. But 
the spaces we introduce here are novel in that we require additional
regularity in all directions. To the best of our knowledge,
this is the first time that a
detailed analysis of elliptic operators has been carried out under
the assumption that the metric has boundary regularity as weak as we require here.)

\begin{lemma}
Let $k\geq0$, $\alpha \in [0,1)$, and $0\leq m \leq k$.
\begin{enumerate}
\item A  tensor field $u$ of weight $r$ is an element of $\mathscr C^{k,\alpha;m}(M)$ if and only if
\begin{equation*}
\bar\nabla{}^j u\in C^{k-j,\alpha}_{r +j}(M)
\quad\text{ for all }0\leq j\leq m.
\end{equation*}

\item Endowed with the norm
\begin{equation}
\label{DefineTripleNorm}
\gNorm{u}_{k,\alpha;m} = 
\sum_{l=0}^m \| \bar\nabla\,{}^l u\|_{C^{k-l,\alpha}_{r+l}(M)},
\end{equation}
the collection $\mathscr C^{k,\alpha;m}(M)$ is a Banach space.
\end{enumerate}
\end{lemma}

\begin{proof}
The first claim relies on the formula 
\begin{equation}
\label{I-love-Lie}
\mathcal L_Xu = \bar\nabla _X u + u\ast \bar\nabla X,
\end{equation}
where $\ast$ represents a contraction of the tensor product.
For $m=0$ there is nothing to show.
Consider the case $m=1$ and suppose that $u\in \mathscr C^{k,\alpha;1}(M)$.
For any $X\in \mathscr V$, the tensor $\bar\nabla X$ has weight zero and is smooth on $\bar M$.
Hence $\bar\nabla X\in C^{k,\alpha}(M)$.
Therefore, \eqref{I-love-Lie} implies that $\bar\nabla_Xu \in C^{k-1,\alpha}_r(M)$ for all $X\in \mathscr V$.
Because every vector field in $\mathscr V_0$ can be written $Y=\rho X$ for some $X\in \mathscr V$,
this implies $\bar\nabla_Yu \in C^{k-1,\alpha}_{r+1}(M)$ for all $Y\in \mathscr V_0$.
Using the finite collection of background coordinate charts, we can choose a finite set of vector fields in $\mathscr V_0$ that contains an orthonormal basis (with respect to $h$) in a neighborhood of each point.
Therefore $\bar\nabla u\in C^{k-1,\alpha}_{r+1}(M)$.
Conversely, formula \eqref{I-love-Lie} implies that if $u\in C^{k,\alpha}_r(M)$ and $\bar\nabla u\in C^{k-1,\alpha}_{r+1}(M)$, then for any $X\in \mathscr V$ we have $\mathcal L_Xu\in C^{k,\alpha}_r(M)$.

Repeated application of \eqref{I-love-Lie} shows that
\begin{equation*}
\mathcal L_{X_1}\dots\mathcal L_{X_m}u
= X_1\ast\dots\ast X_m \ast \bar\nabla{}^m u
+ \sum_{l<m}B_l \ast \bar\nabla{}^l u
\end{equation*}
for some tensors $B_l$, which are in $ C^\infty(\bar M)$ if $X_i\in \mathscr V$.
The first claim then follows by induction.

That $\mathscr C^{k,\alpha;m}(M)$ is complete, and thus a Banach space, follows from the completeness of the spaces $C^{k,\alpha}_\delta(M)$.
\end{proof}

The following lemma describes some important properties of the spaces $ \mathscr C^{k,\alpha;m}(M)$; in particular, parts \eqref{part:script-inclusions} and \eqref{part:intermediate} show that $\mathscr C^{k,\alpha;1}(M)$ is intermediate between $C^{k,\alpha}(\bar M)$ and $C^{k,\alpha}_r(M)$.

\begin{lemma}
\label{lemma:properties-of-scriptC}
Suppose $\alpha\in [0,1)$ and $0\le m\le k$.

\begin{enumerate}

\item\label{part:algebras}
$\mathscr C^{k,\alpha;m}(M)$ is an algebra under the tensor product, and is invariant
under contraction.

\item \label{part:script-inclusions}
If $u\in C^{k,\alpha}_{r+m}(M)$ is a tensor field of weight $r$, then $u\in \mathscr C^{k,\alpha;m}(M)$.
All tensor fields of weight $r$ in $\mathscr C^{k,\alpha;m}(M)$ are in $C^{k,\alpha}_r(M)$.

\item\label{part:intermediate}
The following inclusions are continuous:
\begin{align}
&C^{k,\alpha}(\bar M) \subseteq \mathscr C^{k,\alpha;m}(M), && 0\le m\le k,\label{eq:smooth-to-bdry}
\\
&\mathscr C^{k,\alpha;m}(M) \subseteq C^{m-1,1}(\overline M), &&1\le m \le k, \label{eq:scriptC-Lipschitz}
\end{align}
where $C^{m-1,1}(\overline M)$ denotes 
the space of tensor fields  on $\bar M$ with Lipschitz continuous derivatives up to order $m-1$.

\item\label{part:BoundaryLemma}
If $u\in \mathscr C^{k,\alpha;m}(M)$ is a tensor field of weight $r$ and 
\begin{equation*}
|\bar\nabla{}^j u|_{\bar h} \to 0 \text{ as }\rho \to 0, \quad 0\leq j \leq m-1,
\end{equation*}
then $u\in C^{k,\alpha}_{r+m}(M)$, with 
$\| u \|_{C^{k,\alpha}_{r+m}(M)} \leq 
C \gNorm{u}_{k,\alpha;m}$ for some constant $C$ depending only on universal parameters.

\item\label{part:TripleComponents}
If $u\in\mathscr C^{k,\alpha;m}(M)$, then the functions $u_{i_1\dots i_p}^{j_1\dots j_q}$ describing the components of $u$ in background coordinates $(\mathcal U,\Theta)$ satisfy 
\begin{equation*}
(\partial_\Theta)^\beta(\rho\partial_\Theta)^\gamma u_{i_1\dots i_p}^{j_1\dots j_q} \in L^\infty(\mathcal U),
\qquad |\beta|\le m, \ |\beta|+|\gamma|\le  k.
\end{equation*}
Furthermore, if $\Phi\colon B_2\to M$ is a M\"obius parametrization centered at $(\theta_0, \rho_0)$ then  
\begin{equation*}
\begin{gathered}
\| \partial^\beta(u_{i_1\dots i_p}^{j_1\dots j_q}\circ \Phi)\|_{C^{k-|\beta|,\alpha}(B_2)}\leq  \rho_0^{|\beta|}\gNorm{u}_{k,\alpha;m},
\qquad 0\le|\beta|\le m.
\end{gathered}
\end{equation*}

\item\label{part:nablabar-mapping}
The following maps are continuous:
\begin{align}
\label{eq:nablabar}
\overline\nabla&\colon \mathscr C^{k,\alpha;m}(M)\to \mathscr C^{k-1,\alpha;m-1}(M), \qquad 1\le m\le k,
\\
\label{eq:rho-nablabar}
\rho\overline\nabla&\colon \mathscr C^{k,\alpha;m}(M)\to \mathscr C^{k-1,\alpha;m}(M), \qquad 0\le m\le k-1.
\end{align}
Furthermore, multiplication by $\rho$ is a continuous map from $\mathscr C^{k,\alpha;m}(M)$ to $\mathscr C^{k,\alpha;m+1}(M)$.

\end{enumerate}
\end{lemma}

\begin{proof}
The first claim follows from the product rule, and the fact that contraction preserves the weight of a tensor field.
For the second claim, \eqref{eq:smooth-to-bdry} follows from 
Lemma \ref{DLemma-ListOfHolderFacts}\eqref{DLemma-Include}
and the fact that if $u\in C^{k,\alpha}(\bar M)$ is a tensor of weight $r$, then
$\overline\nabla{}^l u$ is a tensor of weight $r+l$ in $C^{k-l,\alpha}(\overline M)$.
To prove \eqref{eq:scriptC-Lipschitz},
it suffices to consider the case where $m=1$.
We have $|u|_{\bar h}$ and  $|\bar\nabla\, u|_{\bar h}$ bounded on $M$.
Thus $u$ is uniformly continuous on $M$ and extends uniquely to a Lipschitz continuous tensor field on $\bar M$.

For \eqref{part:BoundaryLemma}, consider first the case $m=1$.
In the case, we have that $\bar\nabla u \in C^{k-1,\alpha}_{r+1}(M)$ and that $|u|_{\bar h}$ vanishes along $\partial M$.
Integrating $\bar\nabla_{\grad \rho}u$ from $\rho =0$, where $u$ vanishes, we see that $u\in C^{0}_{r+1}$.
The desired estimate follows from \eqref{BasicHolderNormEquivalences} and Lemma \ref{DLemma-ListOfHolderFacts}\eqref{DLemma-DifferenceTensorH}.
Iteratively applying this same argument to $\bar\nabla{}^l u$, $1\leq l \leq m-1$ yields the desired result.

The remaining claims follow directly from the definition.
\end{proof}

\begin{remark}
\label{Remark-TripleExample}
Lemma \ref{lemma:properties-of-scriptC}\eqref{part:intermediate} is essentially sharp in view of the following example: Let $ u = \rho\sin{(\log{\rho})}$.
It is easy to see that $u \in \mathscr C^{k,\alpha;1}(M)$
for all $k\geq 0$ and $\alpha\in [0,1)$.
However, $\bar\nabla u$ does not extend continuously to $\bar M$.
\end{remark}

\begin{remark}
\label{WeaklyPolyhomogeneous}
If $u\in \mathscr C^{k,\alpha;m}(M)$ with $1\leq m \leq k$ and $u$ is polyhomogeneous, then $u\in C^m_\phg(\bar M)$; see Lemma \ref{InclusionLemma}.
\end{remark}

We now establish the following regularization theorem.

\begin{theorem}
\label{thm:Regularization}
Suppose
$\tau$ is a tensor field of weight $r$ in $\mathscr{C}^{l,\beta;m}(M)$
for some $0\le m\le l$ and $\beta\in[0,1)$. 
Then there exists a tensor $\widetilde \tau$, de\-pending linearly on $\tau$, such that $\widetilde\tau\in\mathscr C^{k,\alpha;m}(M)$ for all $k\geq0$ and $0\leq \alpha <1$ 
and such that 
$\widetilde \tau-\tau\in C^{l,\beta}_{r+m}(M)$.
For each $k$ and $\alpha$ there is a constant $C$ such that
$\gNorm{\widetilde\tau}_{k,\alpha;m} 
\leq C \gNorm{\tau}_{l,\beta;m}$.
\end{theorem}

The construction of $\tilde\tau$ makes use of the group-theoretic convolution operation on hyperbolic space, which we now describe.

Let $\mathbb H$ be the $(n+1)$-dimensional upper half-space with coordinates $\Theta = (\theta,\rho)$. 
Note that $\mathbb H$ is a group under the multiplication $(\theta,\rho)\cdot(\theta',\rho') =
(\theta +\rho\theta',\rho\rho')$, with identity $(0,1)$ and  inverses given by
$(\theta,\rho)^{-1} = (-\theta/\rho,1/\rho)$. The hyperbolic metric $\breve g$ is left-invariant
under this group structure. (Geometrically, the group structure arises from identifying
$\mathbb H$ with the set of isometries of hyperbolic space generated by dilations
and horizontal translations.)

For any bounded integrable functions $\tau$ 
and $\psi$, at least one of which is compactly supported, we define the group-theoretic convolution $\tau\ast\psi$ by  $(\tau*\psi)(q) = \int_{\mathbb H}\tau(p)\psi(p^{-1}q)\,dV_{\breve g}(p)$.
More explicitly, this is 
\begin{equation}\label{eq:def-regularization}
(\tau*\psi)(\theta,\rho) = \int_{\mathbb H} \tau(u,v) \psi\left(\frac{\theta-u}{v},\frac{\rho}{v}\right) v^{-(n+1)}\,du^1\dots du^n\,dv.
\end{equation}
The change of variables $u^i=\theta^i+\rho x^i$, $v = \rho y$ converts this to the 
alternative form
\begin{equation}\label{eq:def-regularization2}
(\tau*\psi)(\theta,\rho) = \int_{\mathbb H} \tau(\theta+\rho  x,\rho  y) \psi\left(-\frac{ x}{ y},\frac{1}{ y}\right) { y}^{-(n+1)}\,d x^1\dots d x^n\,d y.
\end{equation}

\begin{lemma}[Properties of Group Convolution]
\label{lemma:prop-conv}
Let $\mathcal U$ and $\mathcal V$ be open subsets of $\mathbb H$.
Suppose $\psi\in C^\infty_c(\mathcal V)$ and $\tau$ is a bounded
integrable function supported in $\mathcal U$.
\begin{enumerate}
\item\label{part:conv-supp}
$\supp\tau*\psi\subseteq \mathcal U \mathcal V = \{pq: p\in \mathcal U\text{ and }q\in\mathcal V\}$. 

\item\label{part:conv-regularity}
If $\tau\in \mathscr C^{m,0;m}(\mathbb H)$, then $\tau*\psi\in  \mathscr C^{k,\alpha;m}(\mathbb H)$ for all $k\geq 0$ and $0\leq \alpha<1$,
with 
$$
\gNorm{\tau*\psi}_{k,\alpha;m}
\leq C\gNorm{\tau}_{m,0;m} \|\psi\|_{C^{k+1}(\mathbb H)}
$$
for some constant $C$ depending only on $k,\alpha,m$.
\item\label{part:conv-boundarydecay}
If $\tau\in \mathscr C^{1,0;1}(\mathbb H)$ and $\int_{\mathbb H}\psi(q^{-1})dV_{\breve g}(q) =1$, then $\tau-\tau*\psi=\mathcal O(\rho)$.
\end{enumerate}
\end{lemma}

\begin{proof}
Claim \eqref{part:conv-supp}
follows from \eqref{eq:def-regularization2},
as does the fact that $\tau*\psi$ is bounded by a constant multiple of $\|\tau\|_{L^\infty(\mathbb H)}\|\psi\|_{L^\infty(\mathbb H)}$.

A direct computation using \eqref{eq:def-regularization} shows that $X(\tau*\psi) = \tau*(X\psi)$ if $X$ is one
of the vector fields $\rho\partial/\partial\rho$, $\rho\partial/\partial \theta^\alpha$. 
Note that these are orthonormal vector fields that form a basis for the Lie algebra of $\mathbb H$.
Therefore the $C^{k,0}$ norm of a function $u$ is equivalent
to the supremum of $|X_{i_1}\cdots X_{i_j}u|$ over all $j$-tuples of these vector fields, $0\le j\le k$.
Since $X_{i_1}\cdots X_{i_j}\psi$ is also smooth and compactly supported in $\mathcal V$,
it follows that $\tau*\psi$ remains bounded
after any number of applications of these vector fields, so $ \tau\ast\psi\in C^{k,\alpha}(\mathbb H)$
for all $k$ and all $\alpha$, with $\|\tau*\psi\|_{C^{k,\alpha}(\mathbb H)}\leq C \|\tau\|_{L^\infty(\mathbb H)}\|\psi\|_{C^{k+1}(\mathbb H)}$.

Next assume that $\tau\in \mathscr C^{m,0;m}(\mathbb H)$ for some $m\ge 0$.
If $m=0$, there is nothing more to prove, so assume $m\ge 1$.
A simple computation using \eqref{eq:def-regularization2} shows that 
\begin{equation}\label{eq:partial-ftwiddle-theta}
\frac{\partial(\tau*\psi)}{\partial \theta^\alpha} = 
\frac{\partial \tau}{\partial \theta^\alpha}*\psi
\end{equation}
for $\alpha=1,\dots, n$. 
A slightly more involved computation shows 
\begin{equation}\label{eq:partial-ftwiddle-rho}
\frac{\partial(\tau*\psi)}{\partial \rho} =
\sum_{\alpha=1}^n \frac{\partial \tau}{\partial\theta^\alpha}*\psi^\alpha + \frac{\partial \tau}{\partial\rho}*\psi^0,
\end{equation}
where $\psi^\alpha,\psi$ are the compactly supported functions defined by
\begin{equation*}
\psi^\alpha(u,v) = -\frac{u^\alpha}{v}\psi(u,v),\qquad
\psi^0(u,v) = \frac{1}{v}\psi(u,v).
\end{equation*}
Iterating these computations shows that for any multi-index $I$ with $|I|\le m$,
we have
\begin{equation*}
\frac{\partial^I(\tau*\psi)}{\partial \Theta^I} =
\sum_{J: |J|=|I|} \frac{\partial^J \tau}{\partial\Theta^J}*\psi^J
\end{equation*}
for some $\psi^J\in C^\infty_c(\mathcal V)$.
The fact that 
$\tau\in \mathscr C^{m,0;m}(\mathbb H) $ with $m\ge |J|$ implies that 
each derivative $\partial^J \tau/\partial\Theta^J$
is bounded, so the argument above shows that ${\partial^I(\tau*\psi)}/{\partial \Theta^I}\in 
C^{k,\alpha}(\mathbb H)$ for all $k$ and all $\alpha$,
and thus $\tau*\psi \in \mathscr C^{k,\alpha;m}(\mathbb H)$, with norm 
bounded by a constant multiple
of $\gNorm{\tau}_{m,0;m}\|\psi\|_{C^{k+1}(\mathbb H)}$;
this proves
\eqref{part:conv-regularity}.

Finally, assume the hypotheses of 
\eqref{part:conv-boundarydecay} and let $\tilde\tau = \tau\ast\psi$.
The fact that the first derivatives of $\tau$ with respect to
$(\theta,\rho)$ are bounded implies that $\tau$ is Lipschitz continuous in
these coordinates, so $|\tau(\theta+\rho x,\rho y) - \tau(\theta,\rho)|\le C\rho(|x|+|y-1|)$. 
Since $|x|$ and $|y-1|$ are bounded on the support of $\psi(-x/y,1/y)$, we have
\begin{align*}
|\widetilde \tau(\theta,\rho) &- \tau(\theta,\rho)|
\\
&\le  \int_{\mathbb H} \left|
\tau(\theta+\rho  x,\rho  y) - \tau(\theta,\rho)\right|
\psi\left(-\frac{ x}{ y},\frac{1}{ y}\right) { y}^{-(n+1)}\,d x^1\dots d x^n\,d y\\
&\le  C\rho\int_{\mathbb H} \left(|x| + |y-1|\right)
\psi\left(-\frac{ x}{ y},\frac{1}{ y}\right) { y}^{-(n+1)}\,d x^1\dots d x^n\,d y\\
&= \mathcal O(\rho).
\tag*{\qedsymbol}
\end{align*}
\let\qed\relax
\end{proof}

\begin{proof}[Proof of Theorem \ref{thm:Regularization}]
By means of a partition of unity, we may restrict attention to a tensor field
supported in a single
background chart $(\mathcal U,\Theta)$, and we may assume that the background coordinates
extend to a larger open set $\mathcal U'\supseteq \overline{\mathcal U}$. 
To further simplify, we prove the theorem in the case that $\tau$ is a function;
applying the same argument to the components of an arbitrary tensor
field in background coordinates easily yields the analogous result in the higher-rank tensor case. 
We denote the background coordinates by $\Theta = (\theta,\rho)$,
and use them to identify $\mathcal U'$ with an open subset of the upper half-space $\mathbb H$.

We prove by induction on $q$ that for each $q=0, \dots, m$ there exists $\tilde\tau_q\in \bigcap_{k,\alpha}\mathscr C^{k,\alpha;m}(M)$ such that $\tau - \tilde\tau_q\in C^{l,\beta}_{q}(M)$ and such that
\begin{equation*}
\gNorm{\tilde\tau_q}_{k,\alpha;m} \leq C \gNorm{\tau}_{l,\beta;m}.
\end{equation*}
When $q=0$, we just set $\tilde\tau_q = 0$.
Then assume, for some $0\leq q \leq m-1$, the existence of $\tilde\tau_q$ satisfying the above conditions and set $u = \tau - \tilde\tau_q$.
Thus $u\in \mathscr C^{l,\beta;m}(M) \cap C^{l,\beta}_q(M)$ and 
$$w := \frac{1}{q!}\frac{\partial^q u}{\partial\rho^q} \in \mathscr C^{l-q,\beta;m-q}(M).$$

Let $\phi$ be a smooth function on $\mathbb H$ that 
satisfies $\int_{\mathbb H} \phi(p^{-1})\,dV_{\breve g}(p)=1$,
and that is compactly supported in a neighborhood 
$\mathcal V$ 
of $(0,1)$ small enough that
$\mathcal U\mathcal V\subset\mathcal U'$.
Let $\tilde w = w\ast \phi$.
By Lemma \ref{lemma:prop-conv} we have $\tilde w \in \bigcap_{k,\alpha} \mathscr C^{k,\alpha;m-q}(M)$
and 
$\gNorm{\tilde w}_{k,\alpha;m-q} \leq C \gNorm{w}_{m-q,0,m-q}$.
Since $m-q \geq 1$, 
Lemma \ref{lemma:prop-conv}\eqref{part:conv-boundarydecay} implies that $w-\widetilde w = \mathcal O(\rho)$.

We now seek to apply Lemma \ref{lemma:properties-of-scriptC}\eqref{part:BoundaryLemma} to 
show that 
$u - \rho^q\tilde w\in C^{l,\beta}_{q+1}(M)$.
By Lemma \ref{lemma:properties-of-scriptC}\eqref{part:nablabar-mapping} applied to $\tilde w$, 
we have $u - \rho^q\tilde w\in \mathscr C^{l,\beta;m}(M)\cap C^{l,\beta}_q(M)$.
Thus it remains to show that derivatives of $u-\rho^q \tilde w$ having order $q$ vanish at $\rho =0$.
When $|J| \leq q-1\leq m-2$, 
\begin{equation}\label{eq:J-derivs}
\frac{\partial^J}{\partial\Theta^J}\left(u - \rho^q\tilde w\right)\in \mathscr C^{2,\beta;2}(M)\cap C^{2,\beta}_{1}(M),
\end{equation}
and therefore all such derivatives vanish at $\rho=0$.
To handle the derivatives of order $q$, note that each such derivative can be expressed in
one of the following forms:
\begin{equation}\label{qth-derivs}
\frac{\partial}{\partial\theta^j}\frac{\partial^J}{\partial\Theta^J}\left(u - \rho^q\tilde w\right)
\quad \text{ or }\quad 
\frac{\partial^q}{\partial \rho^q}\left(u - \rho^q\tilde w \right)
\end{equation}
for some multi-index $J$ of length $q-1$.
It follows from 
Lemma \ref{lemma:properties-of-scriptC}\eqref{part:intermediate}
that the expression in \eqref{eq:J-derivs} is in $C^{1,1}(\overline M)$ and vanishes
on $\partial M$, so the first expression in \eqref{qth-derivs} vanishes on $\partial M$
as well. 
Since $q\leq m-1$ and $\tilde w \in \mathscr C^{m,0;m-q}(M)$, we have
$\rho^{j-1}{\partial^j \tilde w}/{\partial \rho^{j}}$ bounded for any $1\leq j \leq m$.
Thus
\begin{equation*}
\begin{aligned}
\frac{\partial^q}{\partial \rho^q}\left(u - \rho^q\tilde w \right)
&=
\frac{\partial^q u}{\partial \rho^q} - q! \tilde w + \mathcal O(\rho)
\\
&=
\frac{\partial^q u}{\partial \rho^q} - q! w + \mathcal O(\rho)\\
&=
\mathcal O(\rho),
\end{aligned}
\end{equation*}
where the second equality comes from $\tilde w - w=\mathcal O(\rho)$
and the third from the definition of $w$.
Thus Lemma \ref{lemma:properties-of-scriptC}\eqref{part:BoundaryLemma} implies that $u - \rho^q\tilde w\in C^{l,\beta}_{q+1}(M)$.

We now set $\tilde\tau_{q+1} = \tilde\tau_q + \rho^q \tilde w$.
By Lemma \ref{lemma:properties-of-scriptC}\eqref{part:nablabar-mapping} and the estimates recorded above we have
\begin{equation*}
\gNorm{\rho^q \tilde w}_{k,\alpha;m}
\leq C\gNorm{\tilde w}_{k,\alpha;m-q}
\leq C\gNorm{ w}_{m-q,0;m-q}
\leq C \gNorm{u}_{m,0,m},
\end{equation*}
from which we obtain 
$\gNorm{\tilde\tau_{q+1}}_{k,\alpha;m} \leq C \gNorm{\tau}_{l,\beta;m}.$
\end{proof}

\section{Properties of weakly asymptotically hyperbolic metrics}
\label{Section-WAH}

Recall that a metric $g$ on $M$ is said to be \Defn{conformally compact}
if $\bar g=\rho^2 g$ extends continuously to a nondegenerate metric on $\overline M$.
The next lemma describes the behavior of the curvature operator 
$\Riem[\bar g]$ (viewed as a $(2,2)$ tensor) of the
conformal compactification in case $\bar g$ is in one of the spaces $\mathscr C^{k,\alpha;m}(M)$.

\begin{lemma}
\label{BarRiemLemma}
Let $k\geq 2$ and $\alpha \in [0,1)$, and suppose $\bar g$ is a Riemannian metric on $\overline M$.
\begin{enumerate}
\item If $\bar g \in  \mathscr C^{k,\alpha;1}(M)$, then $\Riem[\bar g]\in C^{k-2,\alpha}_{-1}(M)$.
\item If $\bar g \in \mathscr C^{k,\alpha;2}(M)$, then $\Riem[\bar g]\in C^{k-2,\alpha}_{0}(M)$.
\end{enumerate}
\end{lemma}

\begin{proof}
Let $\bar D[\bar g] = \bar\nabla - \Nabla{\bar g}$ be the difference tensor between the Levi-Civita connections of the compactified background metric $\bar h$ and of $\bar g$; we easily see that $\bar D[\bar g]$ is the sum of (contractions of) terms of the form $\rho^{-1} (\bar g)^{-1}\otimes \bar g \otimes \bar\nabla{}\bar g$.
Thus $\Riem[\bar g]$ is the sum of (contractions of) terms of the form 
\begin{equation}
\label{BargTerms}
 A^{-4}(\bar g)\otimes \bar\nabla^2\bar g,
\quad
A^{-6}(\bar g)\otimes \bar\nabla\,\bar g\otimes \bar\nabla\,\bar g,
\quad
 A^{-3}(\bar g) \otimes \bar\nabla\,\bar g,
\quad
 A^0(\bar g);
\end{equation}
here $A^r(\bar g)$ represents a tensor of weight $r$ which, in any coordinate system, is a smooth polynomial in $\bar g$ and $(\bar g)^{-1}$ with coefficients in $C^\infty(\bar M)$.
If $\bar g\in \mathscr C^{k,\alpha;m}(M)$, then $(\bar g)^{-1}\in \mathscr C^{k,\alpha;m}(M)$ and thus the fact that $\mathscr C^{k,\alpha;m}(M)$ is an algebra implies that $A^r(\bar g) \in \mathscr C^{k,\alpha;m}(M)\subset C^{k,\alpha}_r(M)$.
The desired estimates for the final three terms of \eqref{BargTerms} follow immediately from Lemma \ref{lemma:properties-of-scriptC}.

We now estimate the first term in \eqref{BargTerms}.
If  $\bar g\in  \mathscr C^{k,\alpha;1}(M)$, then Lemma \ref{lemma:properties-of-scriptC}\eqref{part:nablabar-mapping} implies 
$
\rho\bar\nabla{}^2\bar g \in \mathscr C^{k-2,\alpha;0}(M) = C^{k,\alpha}_4(M)
$
and thus $\bar\nabla{}^2\bar g\in C^{k-2,\alpha}_3(M)$.
If $\bar g\in  \mathscr C^{k,\alpha;2}(M)$ then $ \bar\nabla{}^2\bar g \in C^{k-2,\alpha}_4(M)$ and the desired result follows immediately.
\end{proof}

For $k\geq 2$, $1\le m\le k$, and $\alpha\in [0,1)$ we define   $\WAH^{k,\alpha;m}$ to be the 
set of Riemannian metrics $g$ on $M$ such that $\bar g = \rho^2 g\in \mathscr C^{k,\alpha;m}(M)$ extends to a non-degenerate metric on $\bar M$, and such that
$\Riem[g]\to -\Id$ as $\rho \to 0$.
(Recall that $\rho$ is a fixed defining function in $C^\infty(\bar M)$.)
As in \S\ref{Results}, metrics in  $\WAH^{k,\alpha;1}$ are called \Defn{weakly $\boldsymbol C^{k,\alpha}$ asymptotically hyperbolic}.

The following version of Taylor's theorem is used below.

\begin{lemma}\label{lemma:taylor}
Suppose $g\in \WAH^{k,\alpha;2}$ for some $k\geq 2$ and $\alpha\in [0,1)$.
Then for any function $u\in \mathscr C^{k,\alpha;2}(M)\cap C^{k,\alpha}_1(M)$ we have $ u-\rho \langle d\rho, d u\rangle_{\bar g}
 \in C^{k-1,\alpha}_2(M)$ with 
\begin{equation*}
\| u-\rho \langle d\rho, d u\rangle_{\bar g}\|_{ C^{k-1,\alpha}_2(M)}\leq C \gNorm{ u }_{k,\alpha;2}
\end{equation*}
for some constant $C$ depending only on $\gNorm{\bar g}_{k,\alpha;2}$. 
\end{lemma}

\begin{proof}
The assumptions on $u$ imply that $\eta:=\langle\D\rho, \D u\rangle_{\bar g}$ is in $\mathscr C^{k-1,\alpha;1}(M)$.
By Theorem \ref{thm:Regularization} there exists $\tilde \eta\in \mathscr C^{k,\alpha;1}(M)$ such that $\eta-\tilde \eta\in C^{k-1,\alpha}_1(M)$.
By Lemma \ref{lemma:properties-of-scriptC}\eqref{part:BoundaryLemma} and by the estimate in Theorem \ref{thm:Regularization} we have 
$$\|\eta - \tilde\eta\|_{C^{k-1,\alpha}_1(M)}\leq C \gNorm{\eta - \tilde\eta}_{k-1,\alpha;1}\leq C \gNorm{u}_{k,\alpha;2},$$ where here and throughout the proof $C$ represents any constant depending on $\gNorm{\bar g}_{k,\alpha;2}$.

We now seek to apply Lemma \ref{lemma:properties-of-scriptC}\eqref{part:BoundaryLemma} to the function $u^\prime:= u - \rho\tilde \eta$, which is an element of $\mathscr C^{k,\alpha;2}(M)$
by Lemma
\ref{lemma:properties-of-scriptC}\eqref{part:nablabar-mapping}.
Consequently $\D u^\prime$ extends continuously to $\bar M$; note also that $u^\prime \in C^{k,\alpha}_1(M)$.
Thus at $\rho=0$ both $u^\prime$ and the restriction of $\D u^\prime$ to $T\partial M$ vanish.
Direct computation, using the definitions of $u^\prime$ and $\tilde v$, shows that
\begin{equation*}
\langle \D u^\prime, \D\rho\rangle_{\bar g}\big|_{\partial M}
= \langle \D u, \D\rho\rangle_{\bar g}\big|_{\partial M}
-\tilde \eta\big|_{\partial M}
=
\langle \D u, \D\rho\rangle_{\bar g}\big|_{\partial M}
- \eta\big|_{\partial M}
=0.
\end{equation*}
Thus we may invoke Lemma \ref{lemma:properties-of-scriptC}\eqref{part:BoundaryLemma} to conclude that $u^\prime \in C^{k,\alpha}_2(M)$ and that $\|u^\prime \|_{C^{k,\alpha}_2(M)}\leq C \gNorm{u^\prime}_{k,\alpha;2}$.
The proof now follows from the identity $u-\rho \langle \D\rho, \D u\rangle_{\bar g} = u^\prime + \rho (\tilde \eta - \eta)$.
\end{proof}

We are now ready to prove Theorems \ref{thm:WAH-TFAE}--\ref{thm:HAH-TFAE}.

\begin{proof}[Proof of Theorem \ref{thm:WAH-TFAE}]
Since  $\bar g\in  \mathscr C^{k,\alpha;1}(M)$, we have $\Hess_{\bar g}\rho \in C^{k-1,\alpha}_2(M)$.
Thus $(\Hess_{\bar g}\rho)^\sharp \in C^{k-1,\alpha}(M)$ (where the sharp operator is with
respect to $\bar g=\rho^2 g$).  
Because a $(1,1)$ tensor has weight $0$, this implies that $|(\Hess_{\bar g}\rho)^\sharp|_{\bar h}$
is bounded by a constant multiple of $|(\Hess_{\bar g}\rho)^\sharp|_{\bar g}= |(\Hess_{\bar g}\rho)^\sharp|_{ g}$,
which is bounded.

Lemma \ref{BarRiemLemma}
shows that $\Riem[\bar g]$,  $\Ric[\bar g]$, and $\R[\bar g]$ are all $\mathcal O(\rho^{-1})$.
Thus the equivalence of parts \eqref{NEWWAH1}--\eqref{NEWWAH4} of Theorem  \ref{thm:WAH-TFAE}
follows immediately from \eqref{FirstRiemID}, \eqref{RicciID}, and \eqref{ScalarID}.
\end{proof}

\begin{proof}[Proof of Theorem \ref{thm:properties-of-WAH}]
Let $f:=1-|\D\rho|^2_{\bar g}$.
Since $\bar g$ and $\D\rho$ are in $\mathscr C^{k,\alpha;1}(M)$, we have $f\in \mathscr C^{k,\alpha;1}(M)$ as well.
The assumption that $g$ is weakly asymptotically hyperbolic means that $f\to 0$ as $\rho\to 0$, and thus
Lemma \ref{lemma:properties-of-scriptC}\eqref{part:BoundaryLemma}
shows that $f\in C^{k,\alpha}_1(M)$; this is \eqref{PropWAH1}.
Properties \eqref{PropWAH2}, \eqref{PropWAH3}, and \eqref{PropWAH4} then follow
from \eqref{FirstRiemID}, \eqref{RicciID}, \eqref{ScalarID}, respectively,
together with 
Lemma \ref{BarRiemLemma}.

To prove \eqref{PropWAH5}, note that \eqref{PropWAH2} implies 
$({}^g\nabla)^j\Riem[g]$ is a tensor of weight $j$ in $C^{k-j,\alpha}_1(M)$, and the $\bar h$-norm
of such a tensor is $\mathcal O(\rho^{-j+1})$.
\end{proof}

\begin{proof}[Proof of Theorem \ref{thm:HAH-TFAE}]
Equation \eqref{ScalarID} 
can be written 
\begin{equation}
\label{eq:new-1.11}
\R[g]+n(n+1) = -n(n+1) f + \rho^2 \R[\bar g],
\end{equation}
where 
\begin{equation}
\label{little-f}
f:=|d\rho|^2_{\bar g}-1-\frac{2}{n+1}\rho\,\Delta_{\bar g}\rho.
\end{equation}
Lemma \ref{BarRiemLemma} shows 
that
$\rho^2 \R[\bar g]\in C^{k-2,\alpha}_2(M)$, and  it 
follows immediately that \eqref{HAH(2)} $\Rightarrow$ \eqref{HAH(4)} $\Rightarrow$ \eqref{HAH(0)} 
$\Rightarrow$
\eqref{HAH(3)}.

We complete the proof by showing that \eqref{HAH(3)} implies \eqref{HAH(2)}.
Assume therefore that $f=\mathcal O(\rho^2)$.
Since $\bar g\in \mathscr C^{k,\alpha;2}(M)$, we have 
$(\bar g)^{-1}\in \mathscr C^{k,\alpha;2}(M)$. 
Thus the function $u: = |d\rho|_{\bar g}^2 - 1$ is in $\mathscr C^{k,\alpha;2}(M)$ and, due to  
Theorem \ref{thm:properties-of-WAH}\eqref{PropWAH1}, $u\in C^{k,\alpha}_1(M)$.  
Therefore, by Lemma \ref{lemma:taylor}, we can write 
$u= \rho\langle d\rho,du\rangle _{\bar g} + v$,
for some $v\in C^{k-1,\alpha}_2(M)$.
Consequently,
\begin{equation*}
f  = \rho\langle d\rho,du\rangle _{\bar g} -\frac{2}{n+1}\rho\,\Delta_{\bar g}\rho + v.
\end{equation*}
On the other hand, the fact that $g\in \mathscr M^{k,\alpha;2}_{\text{weak}}$ also implies
\begin{equation*}
w:= \langle d\rho,du\rangle _{\bar g} -\frac{2}{n+1}\,\Delta_{\bar g}\rho \in  \mathscr C^{k-1,\alpha;1}(M),
\end{equation*}
and the assumption that $f=\mathcal O(\rho^2)$ implies $w=\mathcal O(\rho)$. Therefore, 
Lemma \ref{lemma:properties-of-scriptC}\eqref{part:BoundaryLemma} implies  
$w\in C^{k-1,\alpha}_1(M)$, from which it follows that $f = \rho w + v\in C^{k-1,\alpha}_2(M)$. 
\end{proof}

\begin{remark}
The proof of Theorem \ref{thm:HAH-TFAE} above invokes both Lemma \ref{lemma:properties-of-scriptC}\eqref{part:BoundaryLemma} and Lemma \ref{lemma:taylor} in order to establish that $f\in C^{k-1,\alpha}_2(M)$ under the hypothesis that $\R[g]+n(n+1) = \mathcal O(\rho^2)$.
The estimates in those lemmas imply that $g \mapsto \R[g]+n(n+1)$ is locally Lipschitz continuous, viewed as a map taking metrics in $\WAH^{k,\alpha;2}$ satisfying $\R[g]+n(n+1) = \mathcal O(\rho^2)$ to functions in   $C^{k-2,\alpha}_2(M)$.
\end{remark}

\section{The tensor $\B_{\bar g}(\omega)$}
\label{IvaTensorSection}

Let $(\bar M, \bar g)$ be a $(n+1)$-dimensional Riemannian manifold and let $\omega\colon\bar M\rightarrow \mathbb{R}$ be any $C^2$ function.
Then the vector field $|\D\omega|^{-2}_{\bar g} \grad_{\bar g}\omega$ is conformally invariant in the sense that for any positive function $\theta$ we have 
\begin{equation*}
|\D\omega|^{-2}_{\theta \bar g} \grad_{\theta \bar g}\omega
=
|\D\omega|^{-2}_{\bar g} \grad_{\bar g}\omega.
\end{equation*}
Let $\mathcal D_{\bar g}$ be the \Defn{conformal Killing (or Alhfors)} operator, taking vector fields to symmetric tracefree covariant $2$-tensor fields, defined by
\begin{equation*}
\mathcal D_{\bar g}X 
= \frac12 \mathcal L_X{\bar g} - \frac{1}{n+1} (\Div_{\bar g}X)\bar g.
\end{equation*}
The operator $\mathcal D_{\bar g}$ transforms under conformal changes of $\bar g$ as follows: For any positive $C^1$ function $\theta$ we have 
\begin{equation*}
\mathcal D_{\theta \bar g} X = \theta\, \mathcal D_{\bar g} X.
\end{equation*}
Thus the map
\begin{equation*}
\omega \mapsto |\D\omega|^{}_{\bar g}\mathcal D_{\bar g}(|\D\omega|^{-2}_{\bar g} \grad_{\bar g}\omega)
\end{equation*}
is a conformally invariant operator taking the function $\omega$ to a symmetric tracefree covariant $2$-tensor field.

Another such operator can be constructed as follows. Observe that the $p$-Laplacian
\begin{equation*}
\Div_{\bar g}\left[ |\D\omega|_{\bar g}^p \grad_{\bar g}\omega\right]
\end{equation*}
is conformally invariant for $p= n-1$, in the sense that 
\begin{equation*}
|\D\omega|^{-(n+1)}_{\theta \bar g} \Div_{\theta \bar g}\left[ |\D\omega|_{\theta \bar g}^{n-1} \grad_{\theta \bar g}\omega\right]
=
|\D\omega|^{-(n+1)}_{\bar g} \Div_{\bar g}\left[ |\D\omega|_{\bar g}^{n-1} \grad_{\bar g}\omega\right].
\end{equation*}
Multiplying by 
\begin{equation*}
\D\omega \otimes \D\omega - \frac{1}{n+1}|\D\omega|^2_{\bar g}\, \bar g
\end{equation*}
yields a conformally invariant operator taking 
a function $\omega$ to a symmetric tracefree covariant $2$-tensor field.

We now combine the two conformally invariant operators above, first multiplying by powers of  $|\D\omega|_{\bar g}$ in order to avoid negative powers and in order to achieve homogeneity in $\omega$, and define the tensor $\B_{\bar g}(\omega)$ by
\begin{equation}
\label{DefineB1}
\B_{\bar g}(\omega)
:=
|\D\omega|_{\bar g}^6\,\mathcal D_{\bar g}(|\D\omega|^{-2}_{\bar g} \grad_{\bar g}\omega)
+ A_{\bar g}(\omega) \left( \D\omega \otimes \D\omega - \frac{1}{n+1}|\D\omega|^2_{\bar g} \bar g \right),
\end{equation}
where
\begin{equation*}
A_{\bar g}(\omega) 
:= \frac{1}{n} |\D\omega|^{3-n}_{\bar g} \Div_{\bar g}\left[ |\D\omega|_{\bar g}^{n-1}\grad_{\bar g}\omega\right].
\end{equation*}
We remark that this definition of the tensor field $\B_{\bar g}(\omega)$ makes sense for manifolds with or without boundary. 

One may readily verify by direct computation that
\begin{multline}
\label{DefineB2}
\B_{\bar g}(\omega) 
= |\D\omega|_{\bar g}^4\left(\Hess_{\bar g}\omega-\frac{1}{n+1}(\Delta_{\bar g}\omega)\bar g\right)
\\
-|\D\omega|_{\bar g}^2\,\, \Nabla{\bar g}_{\grad_{\bar g}\omega} \left[\D\omega\otimes \D\omega-\frac{1}{n+1}|\D\omega|^2_{\bar g}\,\bar g\right]
\\
+A_{\bar g}(\omega)\left(\D\omega\otimes \D\omega -\frac{1}{n+1}|\D\omega|^2_{\bar g}\,\bar g\right),
\end{multline}
where $\Nabla{\bar g}$ is the Levi-Civita connection associated to $\bar g$.

The following basic properties of $\B_{\bar g}(\omega)$, which are immediate from the definition, show that it is a conformally invariant version of the trace-free Hessian.
\begin{proposition}\hfill
\label{B-BasicProperties}
\begin{enumerate}
\item $\B_{\bar g}(\omega)$ is symmetric and trace-free.

\item\label{B-TransverseProperty} $\B_{\bar g}(\omega)(\grad_{\bar g}\omega, \cdot)=0$.

\item\label{B-Scaling} $\B_{\bar g}(c \omega)=c^5\B_{\bar g}(\omega)$ for all constants $c$.

\item\label{B-ConformalScaling} 
If $\tilde g = \theta^{}\bar g$ for a strictly positive function $\theta$, then $\B_{\tilde g}(\omega)=\theta^{-2}\B_{\bar g}(\omega)$ and $A_{\tilde g}(\omega) = \theta^{-2}A_{\bar g}(\omega)$.
\end{enumerate}
\end{proposition}

In the asymptotically hyperbolic setting, we make use of $\B_{\bar g}(\omega)$  with $\omega$ replaced by the defining function $\rho$. 
We first note the following regularity properties.
\begin{lemma}\label{B-in-WAH}
Let $g\in \WAH^{k,\alpha;1}$ be a weakly asymptotically hyperbolic metric on $M$ for $k\geq 1$ and $\alpha\in [0,1)$, and let $\bar g = \rho^2 g$.
Then $\B_{\bar g}(\rho)\in C^{k-1,\alpha}_2(M)$. 

If furthermore $g\in \WAH^{k,\alpha;2}$ and $k\geq 2$, then $\bar\nabla \B_{\bar g}(\rho)\in C^{k-2,\alpha}_3(M)$ and thus $\Div_{\bar g}\B_{\bar g}(\rho) \in C^{k-2,\alpha}_1(M)$.
\end{lemma}

\begin{proof}
Observe that $\B_{\bar g}(\rho)$ consists of terms which are contractions of
\begin{equation}
\label{B-Schema}
(\bar g)^{-1}\otimes (\bar g)^{-1}\otimes (\bar g)^{-1}\otimes \bar g\otimes 
\D\rho \otimes \D\rho \otimes \D\rho \otimes \D\rho \otimes \Nabla{{\bar g}}(\D\rho).
\end{equation} 
Noting that $\D\rho\in C^{k,\alpha}_1(M)$, $\bar\nabla\D\rho\in C^{k,\alpha}_2(M)$,  $\bar\nabla{}^2\D\rho\in C^{k,\alpha}_3(M)$, and  observing that the difference tensor $\Nabla{\bar g} - \bar\nabla$ consists of contractions of $(\bar g)^{-1}\otimes \bar\nabla\,\bar g$, the lemma follows from direct computation.
\end{proof}

We now show that $\B_{\bar g}(\rho)$ agrees with the trace-free Hessian of $\rho$ along $\partial M$ if the scalar curvature decays to $-n(n+1)$ as $\mathcal O(\rho^2)$.
\begin{proposition}
\label{B-DefiningFunctionProperties}
Suppose  $g\in \WAH^{k,\alpha;2}$ for $k\geq 2$ and $\alpha\in [0,1)$.
If $\R[g]+n(n+1)\in C^{k-2,\alpha}_2(M)$, then $\B_{\bar g}(\rho)$ extends continuously to $\bar M$ and satisfies
\begin{equation}
\label{BisHessian}
\B_{\bar g}(\rho) - \left(\Hess_{\bar g}\rho - \frac{1}{n+1}(\Delta_{\bar g}\rho)\bar g\right)\in C^{k-1,\alpha}_3(M).
\end{equation}
In particular \eqref{BatB} holds.
\end{proposition}

\begin{proof}
From Theorem \ref{thm:HAH-TFAE} we have
\begin{equation}\label{d-rho-to-B}
|\D\rho|_{\bar g}^2 - 1 - \frac{2}{n+1}\rho(\Delta_{\bar g}\rho) \in C^{k-1,\alpha}_2(M).
\end{equation}
Note that, as in the proof of Lemma \ref{B-in-WAH}, we have $\Nabla{\bar g} - \bar\nabla\in C^{k-1,\alpha}_1(M)$ and $\bar\nabla\left(\Nabla{\bar g} - \bar\nabla\right)\in C^{k-2,\alpha}_2(M)$; consequently, $d(\Delta_{\bar g}\rho)\in C^{k-2,\alpha}_1(M)$. Taking the differential of \eqref{d-rho-to-B} we  find
\begin{equation*}
\Hess_{\bar g}\rho (\grad_{\bar g}\rho, \cdot) 
- \frac{1}{n+1}(\Delta_{\bar g}\rho) \D\rho 
\in  C^{k-2,\alpha}_2(M). 
\end{equation*}
Since $\Nabla{\bar g}_{\grad_{\bar g}\rho}\D\rho = \Hess_{\bar g}\rho (\grad_{\bar g}\rho, \cdot)$ and $\D\rho\in C^{k,\alpha}_1(M)$, we may by direct computation verify that
\begin{multline}
\Nabla{\bar g}_{\grad_{\bar g}\rho} \left[\D\rho\otimes \D\rho-\frac{1}{n+1}|\D\rho|^2_{\bar g}\,\bar g\right]
\\
= \frac{2}{n+1}(\Delta_{\bar g}\rho) \left[\D\rho\otimes \D\rho-\frac{1}{n+1}|\D\rho|^2_{\bar g}\,\bar g\right] + C^{k-2,\alpha}_3(M)
\end{multline}
and
\begin{equation*}
A_{\bar g}(\rho) = \frac{2}{n+1}(\Delta_{\bar g}\rho)+ C^{k-2,\alpha}_1(M).
\end{equation*}
Inserting this information into the expression for $\B_{\bar g}(\rho)$ we obtain 
\begin{equation}\label{BisHess1}
\B_{\bar g}(\rho) - \left(\Hess_{\bar g}\rho - \frac{1}{n+1}(\Delta_{\bar g}\rho)\bar g\right)\in C^{k-2,\alpha}_3(M).
\end{equation}
On the other hand, the facts that $\B_{\bar g}(\rho)$ consists of terms of the form \eqref{B-Schema} and that $\bar g\in  \mathscr C^{k,\alpha;2}(M)$ imply that 
\begin{equation}\label{BisHess2}
\bar\nabla \left(\B_{\bar g}(\rho) - \left(\Hess_{\bar g}\rho - \frac{1}{n+1}(\Delta_{\bar g}\rho)\bar g\right)\right)\in C^{k-2,\alpha}_3(M).
\end{equation}
Claim \eqref{BisHessian} is immediate from \eqref{BisHess1} and \eqref{BisHess2}, together with Lemma \ref{lemma:properties-of-scriptC}\eqref{part:BoundaryLemma}.
\end{proof}

We now present the proof of Theorem \ref{FH}.

\begin{proof}[Proof of Theorem \ref{FH}]
We first recall Lemma \ref{BarRiemLemma}, which implies that the ultimate term in \eqref{SecondRiemID} is a $(2,2)$ tensor field of class $C^{k-2,\alpha}_2(M)$.
We proceed by showing \eqref{FH1} $\Rightarrow$ \eqref{FH2} $\Rightarrow$ \eqref{FH3} $\Rightarrow$ \eqref{FH4} $\Rightarrow$ \eqref{FH1}.

The condition \eqref{FH1} immediately implies \eqref{FH2}. 
Note that if $h$ is a $(1,1)$ tensor field then the contraction of first upper and first lower indices of $\delta\KN h$ is  $\frac{n-1}{2} h + \frac12 (\tr h)\delta$.
Thus supposing that \eqref{FH2} holds, we may take a contraction of \eqref{SecondRiemID}, and then contract with $\bar g$, to conclude that 
\begin{equation*}
\left(\Hess_{\bar g}\rho - \frac{1}{n+1}(\Delta_{\bar g}\rho)\bar g\right) \in C^{k-2,\alpha}_3(M).
\end{equation*}
In view of Proposition \ref{B-DefiningFunctionProperties}, this implies  \eqref{FH3}.

To see that \eqref{FH3}  implies \eqref{FH4} we note that $g\in  \WAH^{k,\alpha;2}$ implies $ \B_{\bar g}(\rho) \in C^{k-1,\alpha}_2(M)$ and $\bar\nabla \B_{\bar g}(\rho) \in C^{k-2,\alpha}_3(M)$; see Lemma \ref{B-in-WAH}.
Thus applying Lemma \ref{lemma:properties-of-scriptC} \eqref{part:BoundaryLemma} with $u =\B_{\bar g}(\rho)$ gives the desired implication.

Finally, assuming \eqref{FH4} we may use Proposition \ref{B-DefiningFunctionProperties}, together with \eqref{SecondRiemID}, to deduce \eqref{FH1}.
\end{proof}

\section{Fredholm results}
\label{Fredholm}

%

The proof of Theorem \ref{WeaklyFredholmTheorem} consists of adapting the arguments in \cite{Lee-FredholmOperators} to the weakly asymptotically hyperbolic setting.
The arguments in \cite{Lee-FredholmOperators} rely on the fact that a strongly asymptotically hyperbolic metric $g$ of class $C^{l,\beta}$ satisfies 
\begin{equation}
\label{MetricBounds}
\sup_i\| \Phi_i^*g - \breve g\|_{C^{l,\beta}(B_2)}\leq C
\quad\text{ and }\quad 
\sup_i\|(\Phi_i^*g)^{-1}\breve{g}\|_{C^{0}(B_2)}\leq C.
\end{equation} 
An important observation is that \eqref{MetricBounds} holds under the hypothesis that $g\in \WAH^{l,\beta;1}$; the first estimate is a consequence of $\bar g\in C^{l,\beta}_2(M)$, while the second follows from $(\bar g)^{-1}\in C^0(\bar M)$.
The estimates \eqref{MetricBounds} are a key ingredient in the proof of the following elliptic regularity estimates for geometric operators.
\begin{lemma}[Lemma 4.8 of \cite{Lee-FredholmOperators}]
\label{EllipticRegularity}
Suppose that $g$ satisfies \eqref{MetricBounds}, and let $\mathcal P$ satisfy part \eqref{Assume-P-Basic} of Assumption \ref{Assume-P}.
\begin{enumerate}
\item 
Suppose that $\beta\in [0,1)$, $\delta\in \bbR$, $1<p<\infty$, and $2\leq k \leq l$. 
For each $u\in H^{0,p}_\delta(M)$ with $\mathcal P u \in H^{k-2,p}_\delta(M)$, we have $u\in H^{k,p}_\delta(M)$ with
\begin{equation*}
\|u\|_{H^{k,p}_\delta(M)}\leq C\left(\|\mathcal P u\|_{H^{k-2,p}_\delta(M)} + \|u\|_{H^{0,p}_\delta(M)} \right).
\end{equation*}

\item Suppose that $\beta\in (0,1)$,  $\delta\in \bbR$, $0<\alpha<1$, and $2 < k+\alpha  \leq l+\beta$. 
For each $u\in C^{0}_\delta(M)$ with $\mathcal P u \in C^{k-2,\alpha}_\delta(M)$, we have $u\in C^{k,\alpha}_\delta(M)$ with
\begin{equation*}
\|u\|_{C^{k,\alpha}_\delta(M)}\leq C\left(\|\mathcal P u\|_{C^{k-2,\alpha}_\delta(M)} 
+ \|u\|_{C^{0}_\delta(M)} \right).
\end{equation*}
\end{enumerate}
\end{lemma}

The regularity estimates above can be improved if $\mathcal P$ is semi-Fredholm, meaning that the kernel of $\mathcal P$ is finite-dimensional and the image of $\mathcal P$ is closed.
\begin{proposition}
\label{StrongRegularity}
Suppose that $g$ satisfies \eqref{MetricBounds}, and let $\mathcal P$ satisfy part \eqref{Assume-P-Basic} of Assumption \ref{Assume-P}.

\begin{enumerate}
\item Suppose $\beta\in [0,1)$, $\delta\in \mathbb R$,  $1<p<\infty$, and $2\leq k \leq l$.
If $\mathcal P\colon H^{k,p}_\delta(M)\to H^{k-2,p}_\delta(M)$ is semi-Fredholm, then there exist a compact set $K\subset M$ and a constant $C$ such that for each $u\in H^{k,p}_\delta(M)$ we have
\begin{equation}
\label{SobolevSemiFredholm}
\|u\|_{H^{k,p}_\delta(M)} \leq C\left( \|\mathcal P u\|_{H^{k-2,p}_\delta(M)} + \|u\|_{H^{k,p}(K)}\right).
\end{equation}

\item Suppose that $\beta\in (0,1)$, $\delta\in \mathbb R$,  $0<\alpha<1$, and $2< k+\alpha \leq l+\beta$. 
If $\mathcal P\colon C^{k,\alpha}_\delta(M)\to C^{k-2,\alpha}_\delta(M)$ is semi-Fredholm, then there exist a compact set $K\subset M$ and a constant $C$ such that for each $u\in C^{k,\alpha}_\delta(M)$ we have
\begin{equation}
\label{HolderSemiFredholm}
\|u \|_{C^{k,\alpha}_\delta(M)}
\leq C \left(
\|\mathcal P u \|_{C^{k-2,\alpha}_\delta(M)} 
+ \|u\|_{C^{k,\alpha}_\delta(K)}
\right).
\end{equation}

\end{enumerate}
\end{proposition}

\begin{remark}\hfill
\label{RemarkSemiFredholm}
\begin{enumerate}
\item In the Sobolev case it follows from \eqref{SobolevSemiFredholm} that 
\begin{equation*}
\|u\|_{H^{0,p}_\delta(M)} \leq C\|\mathcal P u\|_{H^{0,p}_\delta(M)} 
\end{equation*}
for all $u\in C^\infty_c(M\setminus K)$, which is equivalent to $\mathcal P$ being semi-Fredholm. If the estimate also holds with $p$ replaced by $p^* = p/(1-p)$ and $\delta$ replaced by $-\delta$, then $\mathcal P$ is in fact Fredholm; see \cite[Lemma 4.10]{Lee-FredholmOperators}.

\item The estimates \eqref{SobolevSemiFredholm} and \eqref{HolderSemiFredholm} are related, but not equivalent, to the ``strong regularity intervals'' of \cite{AnderssonChrusciel-Dissertationes}.

\item 
The only properties of $\mathcal P$ used in the proof of Proposition \ref{StrongRegularity} are the semi-Fredholm property and boundedness in the appropriate spaces.
Thus for any compact operator $\mathcal K\colon H^{k,p}_\delta(M)\to H^{k-2,p}_\delta(M)$, the estimate \eqref{SobolevSemiFredholm} holds with $\mathcal P$ replaced by $\mathcal P+\mathcal K$.
Similarly, for any compact operator $\mathcal K\colon C^{k,\alpha}_\delta(M) \to C^{k-2,\alpha}_\delta(M)$, the estimate \eqref{HolderSemiFredholm} holds with $\mathcal P$ replaced by $\mathcal P+\mathcal K$.

\end{enumerate}
\end{remark}

\begin{proof}[Proof of Proposition \ref{StrongRegularity}]
We prove only the H\"older norm estimate \eqref{HolderSemiFredholm}.
The Sobolev estimate follows from analogous reasoning; see also \cite[Lemma 4.10]{Lee-FredholmOperators}.

We first show that sections of $E$ supported near the boundary can be estimated by their distance to the kernel of $\mathcal P$.
Since $\mathcal P$ is semi-Fredholm there exists $\epsilon>0$ such that no non-trivial element of $\ker(\mathcal P)\cap C^{k,\alpha}_\delta(M)$ vanishes identically on the compact set $K = {M\setminus \mathcal C_{\epsilon}}$.
As all norms on a finite-dimensional vector space are equivalent, we see that there exists $c>0$ such that
\begin{equation}
\label{SemiFredholmNormEquivalence}
c^{-1}\|v\|_{C^{k,\alpha}_\delta(K)}
\leq \|v\|_{C^{k,\alpha}_\delta(M)}
\leq c \|v\|_{C^{k,\alpha}_\delta(K)}
\end{equation}
for all $v\in \ker(\mathcal P)\cap C^{k,\alpha}_\delta(M)$.

Let $Y$ be a topological complement of $\ker(\mathcal P)$ in $C^{k,\alpha}_\delta(M)$ so that each $u\in C^{k,\alpha}_\delta(M)$ may be uniquely written as $u = u_0 + u_Y$ with $u_0\in \ker(\mathcal P)$ and $u_Y\in Y$.
The open mapping theorem implies that $\mathcal P\colon Y\to \ran(\mathcal P)\subset C^{k-2,\alpha}_\delta(M)$ is a bijection with bounded inverse.
In particular there exists $C^\prime>0$ such that for all $u = u_0+u_Y\in C^{k,\alpha}_\delta(M)$ we have
\begin{equation}
\label{SemiFredholm-EstimateComplement}
\|u_Y\|_{C^{k,\alpha}_\delta(M)}
\leq C^\prime \|\mathcal P u \|_{C^{k-2,\alpha}_\delta(M)}.
\end{equation}

Now suppose that \eqref{HolderSemiFredholm} fails.
Then, setting $K_m = M\setminus \mathcal C_{1/m}$, there exists a sequence $u_m\in C^{k,\alpha}_\delta(M)$ having unit norm and such that
\begin{equation}
\label{SemiFredholm-AssumeNot}
1=\|u_m \|_{C^{k,\alpha}_\delta(M)}
\geq m \left(
\|\mathcal P u_m \|_{C^{k-2,\alpha}_\delta(M)} 
+ \|u_m\|_{C^{k,\alpha}_\delta(K_m)}
\right).
\end{equation}
Writing $u_m = v_m + u_{m,Y}$, with $u_{m,Y}\in Y$ and $v_m$ in $\ker(\mathcal P)$, we conclude from \eqref{SemiFredholm-EstimateComplement} and \eqref{SemiFredholm-AssumeNot} that 
\begin{equation}
\label{SemiFredholm-BetterComplementEstimate}
\|u_{m,Y}\|_{C^{k,\alpha}_\delta(M)}\leq C^\prime \|\mathcal P u_m \|_{C^{k-2,\alpha}_\delta(M)} \leq C^\prime/m.
\end{equation}
Thus from the reverse triangle inequality we have, for sufficiently large $m$, that 
\begin{equation}
\label{SemiFredholm-LowerBound}
\|v_m\|_{C^{k,\alpha}_\delta(M)}
\geq \left| 1-\|u_{m,Y}\|_{C^{k,\alpha}_\delta(M)}\right|
\geq\frac12.
\end{equation}
For sufficiently large $m$, $K \subset K_m$, and hence the $C^{k,\alpha}_\delta(K_m)$ norm dominates the $C^{k,\alpha}_\delta(K)$ norm.
For such $m$ the norm equivalence \eqref{SemiFredholmNormEquivalence}, together with \eqref{SemiFredholm-AssumeNot} and \eqref{SemiFredholm-BetterComplementEstimate}, imply that
\begin{equation*}
\begin{aligned}
\|v_m\|_{C^{k,\alpha}_\delta(M)}
&\leq c \|v_m\|_{C^{k,\alpha}_\delta(K)}
\\
&\leq c \|v_m\|_{C^{k,\alpha}_\delta(K_m)}
\\
&= c \|u_m - u_{m,Y}\|_{C^{k,\alpha}_\delta(K_m)}
\\
&\leq
c\|u_m \|_{C^{k,\alpha}_\delta(K_m)}
+c\| u_{m,Y}\|_{C^{k,\alpha}_\delta(M)}
\\
&\leq \frac{c}{m}(1+C^\prime).
\end{aligned}
\end{equation*}
However, this contradicts \eqref{SemiFredholm-LowerBound}.
\end{proof}

We now turn to the proof of Theorem \ref{WeaklyFredholmTheorem}, and assume that $g\in \WAH^{l,\beta;1}$ for some $l\geq 2$ and $\beta\in [0,1)$.
We first verify that $\mathcal P$ is indeed a uniformly degenerate operator, and that the indicial map $I_s(\mathcal P)$, defined in \eqref{DefineIndicialMap},  is a $C^0$ bundle map.
\begin{lemma}[Lemmas 4.1 and 4.2 of \cite{Lee-FredholmOperators}]
\label{I-Lemma}
Suppose $g\in \WAH^{l,\beta;1}$ for $l\geq 2$ and $\beta \in [0,1)$, and let $\mathcal P$ satisfy part \eqref{Assume-P-Basic} of Assumption \ref{Assume-P}.
Then in background coordinates we may write
\begin{equation}
\label{BackgroundCoordinateP}
\mathcal P = a^{ij}(\rho\partial_i)(\rho\partial_j) + b^i(\rho\partial_i) + c,
\end{equation}
where the matrix-valued functions $a^{ij}$, $b^i$, $c$ extend continuously to $\bar M$.

Furthermore,  the indicial map $I_s(\mathcal P)\colon (E_\mathbb{C})\big|_{\partial M}\to (E_\mathbb{C})\big|_{\partial M}$ is a $C^0$ bundle map for each $s\in \mathbb C$.
\end{lemma}
\begin{proof}
The proof in the strongly asymptotically hyperbolic setting, as presented in \cite{Lee-FredholmOperators}, relies on the fact that $\rho^2 g$ extends to a $C^{l,\beta}$ metric on $\bar M$.
Here we present those modifications necessary to adapt the arguments in \cite{Lee-FredholmOperators} to the weakly asymptotically hyperbolic setting.

As $\mathcal P$ is geometric, the operator $\mathcal Pu$ is obtained from contractions of tensors formed from $\rho^j\Nabla{g}^j u$, $(\rho \Nabla{g})^j\Riem[g]$, $\bar g$, $(\bar g)^{-1}$,  and $\rho^{n+1}\dVol{g}$; see Chapter 4 of \cite{Lee-FredholmOperators}.
It follows from the definition of $\WAH^{k,\alpha;1}$ that $\bar g$, $(\bar g)^{-1}$, $\rho^{n+1}\dVol{g}$, and
$(\rho \Nabla{g})^j \Riem[g]$ extends continuously to $\bar M$, and that
\begin{equation*}
\left| (\rho \Nabla{g})^j \Riem[g]\right|_{\bar h}
= \mathcal O(\rho)
\qquad \text{ as }\rho \to 0.
\end{equation*}

Thus we focus our attention on $\rho^j\Nabla{g}^j u$, and let $\DifferenceTensor{g} = \Nabla{g} - \bar\nabla$ be the tensor describing the difference between the Levi-Civita connections of $g$ and $\bar h$.
Note that $\rho\DifferenceTensor{g}$ is a tensor field of weight $1$ which is a sum of contractions of 
\begin{equation}
\label{I-D-Parts}
\rho (\bar g)^{-1}\otimes \bar\nabla \bar g
\quad\text{ and }\quad
(\bar g)^{-1} \otimes \bar g \otimes \D\rho;
\end{equation}
the first term is in $C^{l-1,\beta}_2(M) \subset C^0(\bar M)$ and is $\mathcal O(\rho)$ as $\rho\to 0$, and the second term is continuous on $\bar M$.

We claim for $1\leq j \leq l-1$ that the tensor $(\rho \Nabla{g})^j [\rho\DifferenceTensor{g}]$ extends continuously to $\bar M$ and satisfies
\begin{equation*}
\left| (\rho \Nabla{g})^j [\rho\DifferenceTensor{g}] \right|_{\bar h} = \mathcal O(\rho)\quad \text{ as }\rho\to 0.
\end{equation*}
To see this, note that applying $\rho\Nabla{g}$ to the first term in \eqref{I-D-Parts} yields a tensor field in $C^{l-2,\beta}_3(M)\subset C^0(\bar M)$ that is $\mathcal O(\rho)$ as $\rho\to 0$.
Applying $\rho\Nabla{g}$ to the second term in \eqref{I-D-Parts} yields contractions of 
\begin{equation*}
\rho (\bar g)^{-1}\otimes (\bar g)^{-1}\otimes \bar g \otimes (\bar\nabla{}\bar g)\otimes \D\rho
\quad\text{ and }\quad
\rho (\bar g)^{-1} \otimes \bar g \otimes \bar\nabla \D\rho,
\end{equation*}
both of which are in $C^{l-1,\beta}_3(M)$.
The claim regarding higher derivatives follows by induction.

The proof of the lemma now follows exactly as in the proofs of Lemmas 4.1 and 4.2 in \cite{Lee-FredholmOperators}.
\end{proof}

We now extend  the results in Chapter 6 of \cite{Lee-FredholmOperators}, in which a parametrix for $\mathcal P$ is constructed, to the weakly asymptotically hyperbolic setting. 
The construction relies on an estimate for the metric using  boundary M\"obius parametrizations, which we now describe.

Recall from \S\ref{RegularityClasses} that we identify a collar neighborhood $\underline{\mathcal C}_{\rho_*}$ of the boundary with $\partial M \times [0,\rho_*)$.
For each point $\hat p = (\hat\theta,0)\in \partial M$, let $\tilde\Theta = (\tilde\theta,\rho)$ be local coordinates, related to the background coordinates $\Theta$ by an affine transformation of the half space $\mathbb R^n\times [0,\infty)$, such that at $\hat p$ the $\tilde\Theta$ coordinate representation of the metric $\bar g$ is $\delta_{ij}$ and $\hat p$ corresponds to $\tilde\Theta = (0,0)$.
The coordinates $\tilde\Theta$ are uniformly equivalent to the coordinates $\Theta$.
For sufficiently small $r>0$, we define the \Defn{boundary M\"obius parametrization} $\Psi_{ r}\colon Y \to M$ by $(\tilde\theta, \rho)=\Psi_r(x,y) = (rx, ry)$, where $Y$ is the rectangle $Y = \{(x,y)\mid |x|<1, \,0<y<1\}\subset \mathbb H$.
For any choice of  $r>0$, there exists a finite number of boundary M\"obius parametrizations such that the $\{\Psi_r(Y)\}$ cover 
the open set $\mathcal C_r = \partial M\times (0,r)$ 
and are uniformly locally finite; this uniformity is independent of the choice of $r$.

The following estimate of the difference  $\Psi_r^* g - \breve g$, with respect to the intrinsic H\"older norm on $Y\subset \mathbb H$, plays the role of Lemma 6.1 in \cite{Lee-FredholmOperators}.
\begin{lemma}
\label{BoundaryMobiusLemma}
Suppose $g\in \WAH^{l,\beta;1}$ and let $\Psi_r$ be a boundary M\"obius parame\-trization as described above.
Then there is a constant $C>0$, independent of $\hat p$, and a sufficiently small $r$, such that
\begin{equation}
\label{BoundaryMobiusEstimate}
\|\Psi_r^*g - \breve g\|_{C^{l,\beta}(Y)}\leq Cr.
\end{equation}
\end{lemma}

\begin{proof}
It suffices to consider a M\"obius parametrization $\breve\Phi\colon B_2\to \mathbb H$ centered at some $(x_0, y_0)\in \mathbb H$ and to estimate 
\begin{equation*}
(\Psi_r \circ \breve\Phi)^*g - \breve g
\end{equation*}
in $C^{l,\beta}(B_2)$.
Note that  $(\tilde\theta, \rho) = (\Psi_r \circ \breve\Phi)(x,y) = (rx_0 + ry_0 x, r y_0 y)$, and therefore
\begin{equation*}
(\Psi_r \circ \breve\Phi)^*g = (\bar g_{ij} \circ \Psi_r \circ \breve\Phi) \frac{\D\tilde\Theta ^i \otimes \D\tilde\Theta ^j}{y^2}.
\end{equation*}
Note also that $y$ is bounded above and below, and that $\breve\Phi$ is an isometry of $(\mathbb H, \breve g)$.

Let  $f$ be any of the component functions $\bar g_{ij} - \delta_{ij}$  in $\tilde\Theta$ coordinates.
We seek to show
\begin{equation*}
\| f\circ (\Psi_r \circ \breve\Phi)\|_{C^{l,\beta}(B_2)} \leq C r.
\end{equation*}
Since $f$ vanishes at $\tilde\Theta =(0,0)$, the $C^0$ estimate follows from the boundedness of $\partial_{\tilde\Theta} \bar g_{ij}$ and the mean value theorem.
The H\"older estimates of derivatives of $f\circ (\Psi_r \circ \breve\Phi)$ follow from Lemma \ref{lemma:properties-of-scriptC}\eqref{part:TripleComponents}.
\end{proof}

With \eqref{BoundaryMobiusEstimate} established, the parametrix construction of \cite{Lee-FredholmOperators} follows using Lemma \ref{BoundaryMobiusLemma} in place of \cite[Lemma 6.1]{Lee-FredholmOperators}.
In particular, we obtain improved regularity of solutions to $\mathcal P u = f$. 

\begin{lemma}[Lemma 6.4 of \cite{Lee-FredholmOperators}]
\label{ImprovedRegularity}
Suppose $g\in \WAH^{l,\beta;1}$, let $\mathcal P$ satisfy Assumption \ref{Assume-P}, and let $R$ be the indicial radius of $\mathcal P$ as defined in \S\ref{Results}.
\begin{enumerate}
\item 
Suppose that $\beta\in [0,1)$, $1<p<\infty$, $2\leq k \leq l$, $|\delta + n/p - n/2|<R$, and ${|\delta^\prime + n/p - n/2|}<R$. 
Then for each $u\in H^{0,p}_\delta(M;E)$ with $\mathcal P u \in H^{k-2,p}_{\delta^\prime}(M;E)$ we have $u\in H^{k,p}_{\delta^\prime}(M;E)$.

\item 
Suppose $\beta\in (0,1)$, $0<\alpha<1$, $2\leq k+\alpha  \leq l+\beta$, $|\delta - n/2|<R$, and $|\delta^\prime -n/2|<R$. 
Then for each $u\in C^{0}_\delta(M;E)$ with $\mathcal P u \in C^{k-2,\alpha}_{\delta^\prime}(M;E)$ we have $u\in C^{k,\alpha}_{\delta^\prime}(M;E)$.
\end{enumerate}
\end{lemma}

Subsequently, the proofs of Proposition 6.5, Theorem 6.6 and the affirmative portion of Theorem C in \cite{Lee-FredholmOperators}, which corresponds to Theorem \ref{WeaklyFredholmTheorem} above, proceed with no further modifications. 
We have not pursued the possibility of extending the negative portion of Theorem C to the weakly asymptotically hyperbolic setting.

\section{The Yamabe problem}
\label{Yamabe}
We now address the solvability of \eqref{YP}.
In fact, we construct positive solutions to the more general Lichnerowicz-type equation appearing in general relativity (see, for example, \cite{ChoquetBruhat-GRBook}):
\begin{multline}
\label{SimpleLichPhi}
\Delta_g\phi = \frac{n-1}{4n} \R[g]\phi 
\\
- A\phi^{-({3n+1})/({n-1})} - B\phi^{-({n+1})/({n-1})} + \frac{n^2-1}{4} \phi^{({n+3})/({n-1})},
\end{multline}
where $A,B$ are non-negative functions.
Solutions to \eqref{YP} can then be obtained by  taking $A =0$ and $B =0$.

In order to address the solvability of \eqref{SimpleLichPhi} we first use Theorem \ref{WeaklyFredholmTheorem} to establish an existence result for linear scalar equations.
We remind the reader that our sign convention for the Laplace operator is opposite to that of \cite{Lee-FredholmOperators}.

\begin{proposition}
\label{PerturbedPoisson}
Suppose that $g \in \WAH^{l,\beta;1}$ for  $l\geq 2$ and $\beta\in (0,1)$.
Let $k\in \mathbb N$ and $\alpha\in (0,1)$ satisfy $2\leq k+\alpha \leq l+\beta$.
Suppose also that $\kappa\in C^{k-2,\alpha}_\sigma(M)$ for some $\sigma>0$, and that  $c$ is a constant satisfying $c>-n^2/4$ and $c-\kappa\geq 0$.
Then so long as 
\begin{equation}
\label{PP-Radius}
\left|\delta-\frac{n}{2}\right|<\sqrt{\frac{n^2}{4}+c},
\end{equation}
the map
\begin{equation*}
\Delta_g - (c-\kappa)\colon C^{k,\alpha}_\delta(M)\to C^{k-2,\alpha}_\delta(M)
\end{equation*}
is invertible.

Furthermore, if $\rho^2g\in C^2_\phg(\bar M)$, $\rho^{-\nu}f\in C^0_\phg(\bar M)$ for some $\nu >n/2-\sqrt{n^2/4+c}$, and $\kappa$ is a polyhomogeneous function \(which necessarily vanishes on $\partial M$\), then the unique function $u\in C^{2,\alpha}_\delta(M)$ such that
\begin{equation}
\label{PP}
\Delta_g u + (\kappa-c) u = f
\end{equation} 
is polyhomogeneous and satisfies the following boundary regularity conditions:
\begin{itemize}
\item If $\nu > n/2+\sqrt{n^2/4+c}$, then $\rho^{-n/2-\sqrt{n^2/4+c}}\, u \in C^0_\phg(\bar M)$.
\item If $|\nu - n/2|<\sqrt{n^2/4+c}$, then $\rho^{-\nu} u\in C^0_\phg(\bar M)$.
\item If $\nu = n/2+\sqrt{n^2/4+c}$, then $\rho^{-\mu} u \in C^0_\phg(\bar M)$ for all $\mu< \nu$.
\end{itemize}
\end{proposition}

\begin{proof}
Since $\kappa\in C^{k-2,\alpha}_\sigma(M)$, multiplication by $\rho^{-\sigma}\kappa$ is a continuous map 
\begin{gather}
C^{k-2,\alpha}_\delta(M)\to C^{k-2,\alpha}_{\delta}(M).
\end{gather} 
By the Rellich Lemma \cite[Lemma 3.6(d)]{Lee-FredholmOperators}, multiplication by $\rho^\sigma$ is a compact operator 
\begin{gather}
C^{k,\alpha}_\delta(M)\to C^{k-2,\alpha}_{\delta}(M).
\end{gather}
Thus multiplication by $\kappa$, as the composition of a continuous operator and a compact operator, is a compact operator 
\begin{gather}
C^{k,\alpha}_\delta(M)\to C^{k-2,\alpha}_{\delta}(M).
\end{gather}

The Laplacian $\Delta_g$ is well known to be a formally self-adjoint elliptic geometric operator.
From Corollary 7.4 of \cite{Lee-FredholmOperators} we have that the indicial radius of $\Delta_g - c$ is $\sqrt{n^2/4+c}$. 
Hence
\begin{equation}
\label{PerturbedPoissonMap}
\Delta_g  -(c- \kappa)\colon C^{k,\alpha}_\delta(M)\to C^{k-2,\alpha}_{\delta}(M)
\end{equation} 
is Fredholm of index zero so long as \eqref{PP-Radius} holds.
To show that $\Delta_g - (c-\kappa)$ is invertible, it is sufficient to verify that the kernel is trivial.
Suppose, therefore, that $v$ is in the kernel; by  Lemma \ref{ImprovedRegularity} we have $v\in C^{l,\beta}_\delta(M)$ for all $\delta$ satisfying \eqref{PP-Radius}.
In particular, $v$ has sufficient decay that we may integrate by parts to conclude
\begin{equation*}
0 = \int_M \left( |\D v|_{g}^2 + (c-\kappa) v^2 \right) \dVol{g}, 
\end{equation*}
from which we deduce that $v =0$.
Note that in the case $c-\kappa=0$, we must have $c=0$ and thus $\delta>0$ by \eqref{PP-Radius}; since the only constant function in $C^{k,\alpha}_\delta(M)$ is the zero function, we find $v=0$.

Suppose now that $\rho^2g \in C^2_\phg(\bar M)$.
If $f$ is a polyhomogeneous function with $\rho^{-\nu}f \in C^0_\phg(\bar M)$, then $f\in C^{k,\alpha}_\nu(M)$ for all $k\in \mathbb N_0$ and $\alpha\in [0,1)$. 
Thus \eqref{PP} has a unique solution $u\in C^{k,\alpha}_\delta(M)$ for all $0<\delta\leq \nu$ satisfying \eqref{PP-Radius}, and for all $k\geq 2$. 
Theorem \ref{SolutionsArePolyhomogeneous} ensures that the solution $u$ is polyhomogeneous.
 
The boundary regularity follows from inserting the expansion \eqref{GenericTensorExpansion} of $u$ into \eqref{PP} and carrying out a formal asymptotic computation using Lemma \ref{LocalApproximationLemma}: If $\nu$ is in the Fredholm range, then $u$ has the same behavior as $f$, but if $f$ asymptotically decays as $\rho^{n/2+\sqrt{n^2/4+c}}$ there is a resonance, leading to terms with logarithms.
Finally, if $f$ decays faster than $\rho^{n/2+\sqrt{n^2/4+c}}$, then the leading behavior of $u$ is $\rho^{n/2+\sqrt{n^2/4+c}}$, as such terms are annihilated by the indicial operator of $\Delta_g$.
\end{proof}

In order to construct solutions to \eqref{SimpleLichPhi}, it is useful to first make a conformal change of the metric so that it has negative scalar curvature.

\begin{lemma}
\label{WeakYamabe}
Suppose that $g \in \WAH^{l,\beta;1}$ for  $l\geq 2$ and $\beta\in (0,1)$.
Then there exists a positive function $\psi$ with $\psi -1 \in C^{l,\beta}_1(M)$ such that the scalar curvature of $\psi^{4/(n-1)} g$ is strictly negative.

Furthermore, if $\rho^2g\in C^2_\phg(\bar M)$, then $\psi\in C^2_\phg(\bar M)$. 
\end{lemma}

\begin{proof}
Since $g$ is weakly asymptotically hyperbolic, Theorem \ref{thm:properties-of-WAH} implies that the scalar curvature satisfies  $\R[g]+n(n+1) \in C^{l-2,\beta}_1(M)$.
Using a smooth cutoff function, we may construct a  function $\widetilde R$ such that $\widetilde R + n(n+1)\in C^{l-2,\beta}_1$,  such that $\widetilde R\leq \min{(\R[g]-\rho^3,-1)}$ on $\bar M$, and such that  $\R[g] - \widetilde R\in C^{l-2}_2(M)$.
Applying Proposition \ref{PerturbedPoisson} with $\kappa = ((n-1)/4n)(\widetilde R - \R[g])$ and $c=0$, we obtain the existence of a function $u\in C^{l,\beta}_1(M)$ satisfying
\begin{equation*}
\Delta_g u -\frac{n-1}{4n}(\R[g]-\tilde R)u =\frac{n-1}{4n}(\R[g]-\widetilde R).
\end{equation*}

Thus $\psi = 1+u$ satisfies $\Delta_g \psi = ((n-1)/4n)(\R[g]-\widetilde R)\psi$.
As $\R[g]-\widetilde R\geq\rho^3 >0$ on $M$ and $\left.\psi\right|_{\partial M} =1$, the strong  (Hopf) maximum principle implies $\psi >0$.
Thus from \eqref{ChangeSC} we have
\begin{equation}
\label{ChangeScalarID}
\R[\psi^{4/(n-1)}g] 
= \left( -\frac{4n}{n-1}\Delta_g \psi + \R[g]\psi\right)\psi^{-({n+3})/({n-1})}
 = \widetilde R \psi^{-4/(n-1)} <0.
\end{equation}
In the case that $\rho^2g\in C^2_\phg(\bar M)$, the regularity of $\psi$ follows from the latter part of Proposition \ref{PerturbedPoisson}.
\end{proof}

\begin{remark}\hfill
\label{ConformalChangeStillWeak}
 If $g \in \WAH^{l,\beta;1}$ for  $l\geq 2$ and $\beta\in (0,1)$ and $\psi$ is a positive function with $\psi -1\in C^{l,\beta}_1(M)$, then $\psi^{4/(n-1)}g \in \WAH^{l,\beta;1}$ as well.
\end{remark}

We now address the solvability of  \eqref{SimpleLichPhi}, following the standard method of super- and subsolutions \cite{Isenberg-CMCclosed};
see \cite{AnderssonChruscielFriedrich} and \cite{AnderssonChrusciel-Dissertationes} for a related discussion in the asymptotically hyperbolic setting; see \cite{ChoquetBruhat-GRBook}, and the references therein, for analogous treatments in the compact and asymptotically Euclidean settings.

\begin{proposition}
\label{SolveLichPhi}
Suppose that $g \in \WAH^{l,\beta;1}$ for  $l\geq 2$ and $\beta\in (0,1)$.
Suppose furthermore that $A,B\in C^{l-2,\beta}_1( M)$ are nonnegative functions.
Then there exists a unique positive function $\phi$ with $\phi-1\in C^{l,\beta}_1(M)$ satisfying \eqref{SimpleLichPhi}.

Furthermore:
\begin{enumerate}
\item 
\label{part:SLP-High}
 If $g \in \WAH^{l,\beta;2}$ with $\R[g]+n(n+1)\in C^{l-2,\beta}_2(M)$ and $A,B\in C^{l-2,\beta}_2(M)$, then $\phi-1\in C^{l,\beta}_2(M)$ and thus $\phi^{4/(n-1)}g\in \WAH^{l,
\beta;2}$.
\item 
\label{part:SLP-Phg}
If $\rho^2g \in C^2_\phg(\bar M)$ and $\rho^{-2}A, \rho^{-2}B\in C^0_\phg(\bar M)$,  then $\phi\in C^2_\phg(\bar M)$ and thus $\rho^2 \phi^{4/(n-1)}g\in C^2_\phg(\bar M)$.
\end{enumerate}
\end{proposition} 

We remark that if $g$ is smoothly conformally compact, the solution $\phi$ may nevertheless be polyhomogeneous, rather than smooth, on $\bar M$; see \cite{AnderssonChrusciel-Dissertationes}.

\begin{proof}
It follows from  Lemma \ref{WeakYamabe} that there exists a positive function $\psi$ with $\psi -1\in C^{l,\beta}_1(\bar M)$ such that $\R[\psi^{4/(n-1)} g] <0$, and from Remark \ref{ConformalChangeStillWeak} that $\psi^{4/(n-1)}g\in \WAH^{l,\beta;1}$.
Setting ${\gamma} = \psi^{4/(n-1)}g$, $a = \psi^{-4(n+1)/(n-1)}A$, and $b = \psi^{-2(n+2)/(n-1)}B$ we easily verify that a function $\theta$ satisfies
\begin{multline}
\label{ScalarNegativeLich}
\Delta_{\gamma}\theta = F(\theta) 
:= \frac{n-1}{4n} \R[{\gamma}]\theta 
- a\theta^{-(3n+1)/(n-1)} 
\\
- b \theta^{-(n+1)/(n-1)}
+\frac{n^2-1}{4}\theta^{(n+3)/(n-1)}
\end{multline}
if and only if $\phi = \psi\theta$ satisfies \eqref{SimpleLichPhi}; we further require $\left.\theta\right|_{\partial M} =1$ and $\theta>0$.
Note that while $F\colon M\times (0,\infty) \to \mathbb R$, we suppress explicit dependence on $M$.
Note also that any function $v$ with $v = 1 + \mathcal O(\rho)$ satisfies $F(v) = \mathcal O(\rho)$ as $\rho\to 0$.

We show that there exists a solution to \eqref{ScalarNegativeLich} by constructing barriers.
We first note that there exists a constant $C>0$ such that
\begin{equation*}
-C\leq \R[{\gamma}] \leq -\frac{1}{C}
\quad\text{ and }\quad
0\leq a, b \leq {C}.
\end{equation*}
Thus there exists a constant $u_*\in(-1,0)$ with $F(1+u_*)\leq 0$.

Without loss of generality, we may assume that $\Delta_\gamma\rho <0$ on $M$; see the construction in  \cite[Section 4.1]{IsenbergLeeStavrov-AHP}.
Since $\R[\gamma] = -n(n+1) + \mathcal O(\rho)$ and $\R[\gamma]$ is strictly negative, there exists a constant $N>0$ sufficiently large so that
\begin{equation*}
(1-N\rho)^{4/(n-1)}< \frac{-1}{n(n+1)}\R[\gamma]
\quad\text{ when }\rho < N^{-1}.
\end{equation*}
Thus
\begin{equation*}
\Delta_{\gamma}(1-N\rho)\geq 0 \geq F( 1-N\rho)
\qquad\text{ when }\rho < N^{-1}.
\end{equation*}
A similar argument shows that we may furthermore choose $N>0$ such that 
\begin{equation*}
\Delta_{\gamma}(1+N\rho)\leq 0 \leq F(1+N\rho)
\qquad \text{ on }M.
\end{equation*}
Since $(x,u)\mapsto F(1+u)(x)$ and $(x,u)\mapsto \frac{\partial }{\partial u}F(1+u)(x)$ are continuous functions on $\bar M \times [u_*,\max_M(1+N\rho) ]$ we can choose $\Lambda>0$ sufficiently large so that $F(1+u) <\Lambda u$ and $\frac{\partial}{\partial u}F(1+u)<\Lambda$ on that domain.

Define $G(u) = F(1+u) - \Lambda u$; note that $G(u)$ is monotone decreasing in $u$ and that \eqref{ScalarNegativeLich} is satisfied by $\theta=1+u$ if and only if $u$ satisfies
\begin{equation}
\label{uLich}
\Delta_\gamma u - \Lambda u = G(u),
\qquad
\left.u\right|_{\partial M} = 0,
\qquad
u>-1.
\end{equation}

Fix $\delta\in(1/2,1)$ and note that if $v\in C^{2,\beta}_\delta(M)$, then $G(v)\in C^{2,\beta}_\delta(M)$.
Thus by Proposition \ref{PerturbedPoisson} we may define a sequence of functions $\{u_i\}_{i=0}^\infty\subset C^{2,\beta}_\delta(M)$ with $u_0 = N\rho$ and
\begin{equation*}
\Delta_{\gamma} u_{i+1} - \Lambda u_{i+1} = G(u_i),
\quad i\in \mathbb N_0.
\end{equation*}

Since
\begin{equation*}
(\Delta_\gamma-\Lambda)(N\rho) \leq G(N\rho) 
\quad\text{ and }\quad
(\Delta_\gamma-\Lambda)(u_*) \geq G(u_*)
\end{equation*}
and $G$ is monotone decreasing, 
the maximum principle implies that
\begin{equation}
\label{BoundedIterates}
N\rho\geq u_i \geq u_{i+1} \geq u_*\quad i\in \mathbb N_0.
\end{equation}
If $\rho <N^{-1}$ then we have $(\Delta_\gamma-\Lambda)(-N\rho) \geq G(-N\rho)$.
Using the maximum principle, together with the lower bound in \eqref{BoundedIterates}, we conclude that 
\begin{equation}
\label{BetterBoundedIterates}
N\rho \geq u_i \geq \max{(-N\rho,u_*)}
\end{equation}
 for all $i\in \mathbb N_0$.

For any $p>{n}/({1-\delta})$, we have $\rho\in H^{k,p}_\delta(M)$ for all $k\in\mathbb N$; see \cite[Lemma 3.2]{Lee-FredholmOperators}.
Thus the $u_i$ are uniformly bounded in $H^{0,p}_\delta(M)$. 
The monotonicity of $G$ then implies that $G(u_i)$ is uniformly bounded in $H^{0,p}_\delta(M)$ as well.
Thus we may apply elliptic regularity to conclude that the $u_i$ are uniformly bounded in $H^{2,p}_\delta(M)$.

Note that if we also choose $p$ such that $p\geq {(n+1)}/{(1-\beta)} $, then we have $H^{2,p}_\delta(M)\subset C^{1,\beta}_\delta(M)$; see \cite[Lemma 3.6(c)]{Lee-FredholmOperators}.
Thus $\|u_i\|_{C^{1,\beta}_\delta(M)}$ is uniformly bounded in $i$ and the Rellich Lemma \cite[Lemma 3.6(d)]{Lee-FredholmOperators} implies that for any choice of $\delta^\prime\in (1/2,\delta)$ we may pass to a subsequence, which we also denote $\{u_i\}$, that converges to some $u\in C^{0,\beta}_{\delta^\prime}(M)$.
From \eqref{BetterBoundedIterates} we have $u\in C^{0}_1(M)$.
The elliptic estimate of Lemma \ref{EllipticRegularity} implies
\begin{equation*}
\| u_j - u_i\|_{C^{2,\beta}_{\delta^\prime}(M)}
\leq C\left(\| F(u_j) - F(u_i)\|_{C^{0,\beta}_{\delta^\prime}(M)} 
+\| u_j - u_i\|_{C^{0,\beta}_{\delta^\prime}(M)}
\right);
\end{equation*}
thus $\{u_i\}$ is Cauchy in $C^{2,\beta}_{\delta^\prime}(M)$, whence $u\in C^{2,\beta}_{\delta^\prime}(M)$.

For any smooth, compactly supported test function $w$ we have
\begin{equation*}
\begin{aligned}
\int_M &\left(u\Delta_\gamma w - F( 1+u) w\right)\,dV_{\gamma}
\\
&= \lim_{i\to\infty} \int_M \left (u_i\Delta_\gamma w - F( 1+u_i) w\right)\,dV_{\gamma}
\\
&= \lim_{i\to\infty} \int_M \left (\Delta_\gamma u_i - F(1+u_i) \right)w \,dV_{\gamma}
\\
&= \lim_{i\to\infty} \int_M \left (\Lambda (u_i- u_{i-1})+ F( 1+u_{i-1}) - F( 1+u_i) \right)w \,dV_{\gamma}
\\
&=0.
\end{aligned}
\end{equation*}
Thus $u\in C^{2,\beta}_{\delta^\prime}(M)\cap C^0_1(M)$ is a weak, and hence strong, solution to \eqref{uLich}.

To see that $u\in C^{l,\beta}_1(M)$ we note that
\begin{equation*}
\Delta_\gamma u - (n+1)u  = f
\end{equation*}
where
\begin{multline*}
f = \frac{n-1}{4n}\left( \R[\gamma] + n(n+1)\right)(1+u)
- a(1+u)^{-(3n+1)/(n-1)} 
\\
- b(1+u)^{-(n+1)/(n-1)}
+\frac{n^2-1}{4}\left[(1+u)^{(n+3)/(n-1)} -1 -\frac{n+3}{n-1}\,u \right].
\end{multline*}
Since $\delta^\prime>1/2$ we have $f\in C^{0,\beta}_1(M)$. 
Thus Lemma \ref{ImprovedRegularity} implies that $u\in C^{l,\beta}_1(M)$.
Consequently, $\phi = \psi(1+u)$ satisfies \eqref{SimpleLichPhi} and $\phi -1\in C^{l,\beta}_1(M)$.

In the case that $g\in \WAH^{l,\beta;2}$ and $\R[g]+n(n+1)\in C^{l-2,\beta}_2(M)$ we set $w=\phi-1$ and note that $w\in C^{l,\beta}_1(M)$ and that $w$ satisfies
\begin{equation*}
\Delta_g w - (n+1)w  = f^\prime,
\end{equation*}
where
\begin{multline*}
f^\prime = \frac{n-1}{4n}\left( \R[g] + n(n+1)\right)(1+w)
- A(1+w)^{-(3n+1)/(n-1)} 
\\
- B(1+w)^{-(n+1)/(n-1)}
+\frac{n^2-1}{4}\left[(1+w)^{(n+3)/(n-1)} -1 -\frac{n+3}{n-1}\,w \right].
\end{multline*}
If  $A,B\in C^{l-2,\beta}_2(M)$ then $f^\prime\in C^{l-2,\beta}_2(M)$ and hence we conclude that $w\in C^{l,\beta}_2(M)$ as desired.

To show uniqueness we follow the argument in \cite{ChoquetBruhat-GRBook}: Suppose that $\phi,\tilde\phi$ both satisfy \eqref{SimpleLichPhi} and $\phi -1,\tilde\phi -1\in C^{l,\beta}_1(M)$.
Setting $\tilde \gamma = \tilde \phi^{4/(n-1)} g$, $\tilde\theta = \tilde\phi^{-1}\phi$, $\tilde a = \tilde\phi^{-4(n+1)/(n-1)}A$, and $\tilde b = \tilde\phi^{-2(n+2)/(n-1)}B$, 
we have
\begin{equation}
\label{FirstLichRatio}
\begin{aligned}
\Delta_{\tilde \gamma}(\tilde\theta -1)
= \tilde F(\tilde\theta)
&:= \frac{n-1}{4n} \R[\tilde \gamma]\tilde\theta 
- \tilde a \,\tilde\theta^{-(3n+1)/(n-1)} 
\\
&\qquad
- \tilde b \,\tilde\theta^{-(n+1)/(n-1)} 
+\frac{n^2-1}{4} \tilde\theta^{(n+3)/(n-1)}
\\
&= -\tilde a\left( \tilde\theta^{-(3n+1)/(n-1)}-1\right)
-\tilde b\left(\tilde\theta^{-(n+1)/(n-1)}-1\right)
\\
&\qquad
+\frac{n^2-1}{4}\left(\tilde\theta^{(n+3)/(n-1)}-1\right),
\end{aligned}
\end{equation}
where in the second line we have used \eqref{ChangeScalarID} and the fact that $\tilde\phi$ satisfies \eqref{SimpleLichPhi}.
Since $\tilde\theta>0$, for any real number $r$ we have
\begin{equation*}
\tilde\theta^r - 1 = (\tilde\theta -1) f_r
\end{equation*}
for some function $f_r$ with the same sign as $r$.
Therefore, we may express \eqref{FirstLichRatio} in the form
\begin{equation*}
\Delta_{\tilde \gamma}(\tilde\theta-1) - (c-\tilde\kappa )(\tilde\theta-1)=0,
\end{equation*}
where  $\tilde\kappa$ is a $C^{l-1,\beta}_1(M)$ function and $c= (n+1)(n+3)/4$ is a constant with $c\geq \tilde\kappa$.
Thus we may apply Proposition \ref{PerturbedPoisson} to conclude that $\tilde\theta-1=0$.

Suppose now that $\rho^2g \in C^2_\phg(\bar M)$ and $\rho^{-2}A, \rho^{-2}B\in C^0_\phg(\bar M)$.
Then $\psi\in C^2_\phg(\bar M)$ and iteratively applying the elliptic regularity estimates of Lemma \ref{EllipticRegularity} implies that $\phi-1\in C^\infty_{1}(M)$.
From Proposition \ref{Lich-BoundaryRegularity} we have that $\theta$, and thus $\phi$, is polyhomogeneous.
It readily follows from a straightforward asymptotic computation that $\phi\in C^2_\phg(\bar M)$.
\end{proof}

In the case that $g\in \WAH^{k,\alpha;2}$, Proposition \ref{SolveLichPhi} only implies $\phi^{4/(n-1)}g\in \WAH^{k,\alpha;2}$ under the condition that $\R[g]+n(n+1)\in C^{k-2,\alpha}_2(M)$.
To remove this condition, and thus complete the proof of Theorem \ref{YamabeTheorem}, we use the following lemma.

\begin{lemma}
\label{Finally!}
Suppose $g\in \WAH^{k,\alpha;2}$ with $k\ge 2$ and $\alpha\in [0,1)$.
Then the conformal class of $g$ contains a representative $\tilde g\in \WAH^{k,\alpha;2}$ such that $\R[\tilde g] +{n(n+1)}\in C^{k-2,\alpha}_2(M)$.
Furthermore, $\tilde g$ can be chosen so that $g\mapsto \tilde g$ is a locally Lipschitz map $\WAH^{k,\alpha;2}\to \WAH^{k,\alpha;2}$.
\end{lemma}

\begin{proof}
Let $g\in \WAH^{k,\alpha;2}$ for $k\geq 2$ and $\alpha \in [0,1)$. 
We seek a positive function $\theta\in \mathscr C^{k,\alpha;2}(M)$ such that $\left.\theta\right|_{\partial M} =1$ and $\R[\theta^{-2}g] + n(n+1) \in C^{k-2,\alpha}_2(M)$.
Due to Theorem \ref{thm:HAH-TFAE}, it suffices to show that
\begin{equation*}
F(\theta) := |\D(\theta \rho)|^2_{\bar g} -1 - \frac{2}{n+1}(\theta\rho) \Delta_{\bar g}(\theta \rho) \in C^{k-1,\alpha}_2(M).
\end{equation*}

Note that $F(1)=f$, where $f$ is defined by \eqref{little-f}, and that $g\mapsto f$ is locally Lipschitz continuous as a map $\WAH^{k,\alpha;2}\to \mathscr C^{k-1,\alpha;2}(M)$.
By Theorem \ref{thm:Regularization}, there exists $\hat f \in \mathscr C^{k,\alpha;2}$ such that $f - \hat f \in C^{k-1,\alpha}_2(M)$; furthermore, $g\mapsto \hat f$ is locally Lipschitz continuous as a map $\WAH^{k,\alpha;2}\to 
\mathscr C^{k,\alpha;2}$.

Let $\chi\colon \mathbb R \to [0,1]$ be a smooth cutoff function with $\chi=1$ on $(-1/3, \infty)$ and  $\supp{\chi} \subset (-2/3,\infty)$.
Define
\begin{equation*}
\hat w = -\frac{n+1}{4n}\hat f
\quad\text{ and }\quad
w = \chi(\hat w)\hat w,
\end{equation*}
and set $\theta = 1+ w$.
Direct computation using Lemma \ref{lemma:taylor} shows that
\begin{equation*}
\begin{aligned}
F(\theta)
&= f + 2w + \frac{2(n-1)}{n+1}\rho \left\langle d\rho, dw\right\rangle_{\bar g}
+ C^{k-1,\alpha}_2(M)
\\
&= f + \frac{4n}{n+1} w + C^{k-1,\alpha}_2(M)
\\
&\in C^{k-1,\alpha}_2(M)
\end{aligned}
\end{equation*}
as desired.
Finally, as $x\mapsto x^{-2}$ is Lipschitz continuous on $[1/3,\infty)$, the map $g \mapsto \theta^{-2}g$ is locally Lipschitz continuous as claimed.
\end{proof}

\begin{corollary}
\label{cor:ImprovedSLP}
Suppose $g\in \WAH^{k,\alpha;2}$ for some $k\geq 2$ and $\alpha\in (0,1)$.
Then there exists a positive solution $\phi$ to \eqref{SimpleLichPhi} such that $\phi^{4/(n-1)}g\in \WAH^{k,\alpha;2}$.
\end{corollary}

\begin{proof}
From Lemma \ref{Finally!} there exists positive function $\psi\in\mathscr C^{k,\alpha;2}$ such that $\tilde g = \psi^{4/(n-1)}g\in \WAH^{k,\alpha;2}$ and such that $\R[\tilde g]+n(n-1)\in C^{k-2,\alpha}_2(M)$.

Setting $\tilde A = \psi^{-4(n+1)/(n-1)}A$ and $\tilde B = \psi^{-2(n+2)/(n-1)}B$ we easily verify that a function $\tilde\phi$ satisfies
\begin{multline*}
\Delta_{\tilde g}\tilde\phi
= \frac{n-1}{4n} \R[{\tilde g}]\tilde\phi 
\\
- \tilde A\tilde\phi^{-(3n+1)/(n-1)} 
- \tilde B \tilde\phi^{-(n+1)/(n-1)}
+\frac{n^2-1}{4}\tilde\phi^{(n+3)/(n-1)}
\end{multline*}
if and only if $\phi = \psi\tilde\phi$ satisfies \eqref{SimpleLichPhi}.
From Proposition \ref{SolveLichPhi}\eqref{part:SLP-High} there exists $\tilde\phi \in 1+ C^{k,\alpha}_2(M)$ satisfying this equation and such that $\phi^{4/(n-1)} g=\tilde\phi^{4/(n-1)}\tilde g\in \WAH^{k,\alpha;2}$.
\end{proof}

\begin{proof}[Proof of Theorem \ref{YamabeTheorem}]
Suppose that $g\in \WAH^{k,\alpha;1}$ and let $\phi$ be the unique solution to \eqref{YP} provided by  Proposition \ref{SolveLichPhi}.
As $\phi-1\in C^{k,\alpha}_1(M)$, we have $\phi \in\mathscr C^{k,\alpha;1}$ and thus $\hat g = \phi^{4/(n-1)}g \in \WAH^{k,\alpha;1}$ and $\R[\hat g] = -n(n+1)$.

In the case that $g\in \WAH^{k,\alpha;2}$, the result is a consequence of Corollary \ref{cor:ImprovedSLP}.
Finally, in the case that $g$ is polyhomogeneous, the polyhomogeneity of $\phi$, and hence $\hat g$, follows from Proposition \ref{SolveLichPhi}\eqref{part:SLP-Phg}.
\end{proof}

\appendix
\section{Polyhomogeneity and boundary regularity}
\label{appendix}

Our purpose in this appendix is  to give a self-contained account of the boundary regularity of solutions to equations of the form 
\begin{equation}
\label{GenericEllipticEquation}
\mathcal{P} u = f
\end{equation}
in the polyhomogeneous setting; here $\mathcal P$ is a linear geometric operator acting on sections of tensor bundle $E$  arising from a metric $g$ that is polyhomogeneous in the sense defined below.
We further assume that $\mathcal P$ satisfies Assumption \ref{Assume-P}.
Many of the methods employed here have been used elsewhere to obtain related results;  we note in particular \cite{AnderssonChruscielFriedrich}, \cite{LeeMelrose-MongeAmpere},  \cite{Mazzeo-Edge},  \cite{Melrose-TransformationMethods}, and \cite{Melrose-Transformation}.

\subsection{The conormal and polyhomogeneous spaces}
\label{S-DefinePolyhomogeneity}
We first define conormality classes for tensor fields on $M$ using the collection $\mathscr V_b$ of smooth vector fields on $\bar M$ tangent to the boundary $\partial M$.
In background coordinates $\Theta = (\theta^i,\rho)$, a vector $V\in \mathscr V_b$ can be expressed as $V^i\partial_{\theta^i} + V^\rho \rho\partial_\rho$ where $V^i$, $V^\rho$ are smooth functions on $\bar M$.
Define  $\mathcal A(M)$ to be the class of smooth tensor fields $u$ on $M$ satisfying $\mathcal L_{V_1}\dots\mathcal L_{V_k}u \in L^\infty(M)$ for any finite set $\{V_1,\dots, V_k\}\subset \mathscr V_b$.

\begin{remark}
\label{NormalizedFrameRemark}
Direct computation shows that $u\in \mathcal A(M)$ if and only if in any background coordinate chart $(\mathcal U,\Theta)$ the functions expressing $u$ in terms of the `normalized' background coordinate frame $\{\rho\partial_{\Theta^\mu}\}$ and associated dual frame $\{\rho^{-1}\D\Theta^\mu\}$ extend to elements of $\mathcal A(M)$.
\end{remark}

For $\delta\in\mathbb R$ we set $\mathcal A_\delta(M) = \bigcap_{t<\delta} \rho^t\mathcal A(M)$ and $\mathcal A_{-\infty}(M) = \bigcup_{\delta\in \mathbb R}\mathcal A_\delta(M)$.
We emphasize that $\rho^\delta \mathcal A(M)$ is a proper subset of $\mathcal A_\delta(M)$; see Remark \ref{WhatsInAdelta} below.
Sections of class $\mathcal A_{-\infty}$ are called \Defn{conormal};
classes analogous to $\mathcal A$, $\mathcal A_\delta$, and $\mathcal A_{-\infty}$ have been employed elsewhere; see e.g.~\cite{LeeMelrose-MongeAmpere}, \cite{Mazzeo-Edge}, \cite{Melrose-TransformationMethods}, \cite{Melrose-Transformation}.

We now define an important subset of $\mathcal A_{-\infty}(M)$, the polyhomogeneous sections.
First, we consider functions on a background coordinate chart $(\mathcal U,\Theta)$.
We say a complex-valued function $f$ is \Defn{polyhomogeneous} on $\mathcal U$ if
\begin{enumerate}
\item there exist sequences $s_i\in \mathbb C$ and $p_i\in \mathbb N_0$ with $\Re(s_i)$ non-decreasing and diverging to $+\infty$ as $i\to\infty$,
\item there exist smooth functions $\bar f_{ip}(\theta)$, $p=0,\dots,p_i$, defined on an open neighborhood of $\bar U$, and
\item for each $k\in \mathbb N$ there exists $N\in \mathbb N$ such that
\begin{equation*}
f - \sum_{i=0}^N \sum_{p=0}^{p_i} \rho^{s_i} (\log \rho)^p \bar f_{ip} \in \rho^k\mathcal A(\mathcal U),
\end{equation*}
where we extend each $\bar f_{ip}$ to functions on $\mathcal U$ that are independent of $\rho$.
\end{enumerate}
In this case we write
\begin{equation*}
f\sim \sum_{i=0}^\infty \sum_{p=0}^{p_i} \rho^{s_i} (\log \rho)^p \bar f_{ip}.
\end{equation*}
Denote by $\mathcal A_\phg(\mathcal U)$ the collection of polyhomogeneous functions on $\mathcal U$.
We remark that this definition is somewhat more general that those used in  \cite{AnderssonChrusciel-Obstructions}, \cite{AnderssonChrusciel-Dissertationes}, \cite{IsenbergLeeStavrov-AHP}, where $s_i$ are assumed to be real; see \cite{Mazzeo-Edge}.

We call a smooth section $u$ of tensor bundle $E$ on $M$ polyhomogeneous if in each background coordinate chart $(\mathcal U,\Theta)$ the functions that describe the components of $u$ with respect to the normalized background coordinate frame (see Remark \ref{NormalizedFrameRemark}) are in $\mathcal A_\phg(\mathcal U)$ and if the sequences $\{s_i\}$, $\{p_i\}$ are the same in each chart.
Thus in each background coordinate chart, we may write
\begin{equation}
\label{GenericTensorExpansion}
u\sim \sum_{i=0}^\infty \sum_{p=0}^{p_i} \rho^{s_i-r} (\log \rho)^p \bar{u}_{ip}
\end{equation}
for some matrix-valued functions $\bar u_{ip}$; here $r$ is the weight of the bundle $E$.
Note that in fact these matrix-valued functions are the expression in coordinates of smooth sections of $\left.E\right|_{\partial M}$.

Let  $\mathcal A_\phg(M)$ denote the collection of polyhomogeneous tensor fields.
Note that $\mathcal A_\phg(M)\subset \mathcal A_{-\infty}(M)$; see Lemma \ref{InclusionLemma} below. 

It is sometimes convenient to restrict attention to polyhomogeneous fields with exponents $s_i$ in a particular set; thus for $S\subset \mathbb C$ we denote by $\mathcal A_\phg^S(M)$ those elements of $\mathcal A_\phg(M)$ for which the expansion \eqref{GenericTensorExpansion} has $\{s_i\}\subset S$.

We set $C^{k,\alpha}_\phg(\bar M)=C^{k,\alpha}(\bar M)\cap \mathcal A_\phg(M)$.

\begin{remark}
\label{DamnBundleWeights}
 The factor of $\rho^{-r}$ in \eqref{GenericTensorExpansion} is motivated by the fact that if the tensor bundle $E$ has weight $r$ then sections $u$ satisfy $|u|_g = \rho^r|u|_{\bar g}$.
This convention implies that if a tensor $u$ has expansion \eqref{GenericTensorExpansion} then $|u|_g$ behaves as $\rho^{\Re(s_0)}(\log\rho)^{p_0}$ for $\rho$ small; see part \eqref{Inclusion-PhgConormal} of Lemma \ref{InclusionLemma} below.
We further note that $u\in \mathcal A_\phg^{S+r}(M)$ precisely if the functions describing $u$ in any background coordinate chart $(\mathcal U,\Theta)$ are in $\mathcal A_\phg^S(\mathcal U)$.  
\end{remark}

\begin{remark}\hfill
\label{PhgSplitting}

\begin{enumerate} 

\item It follows directly from the definition that if  $u\in \mathcal A_\phg(M)$ then for any $\delta\in \mathbb R$ one may choose a finite set $S\subset \mathbb C$ such that $u = u^\text{fin} + u^\text{rem}$ with $u^\text{fin} \in \mathcal A^S_\phg(M)$ and $u^\text{rem}\in \mathcal A_\delta(M)$.

\item Observe that polyhomogeneous expansions are unique in the sense that if $u = u^\phg + u^\text{rem}$ with $u^\phg\in \mathcal A_\phg(M)$ and $u^\text{rem}\in \mathcal A_\delta(M)$ for some $\delta\in \mathbb R$, then the tensors $\bar u_{ip}$ of the terms $\rho^{s_i}(\log\rho)^p \bar u_{ip}$  
with $\Re (s_i)<\delta$ are uniquely determined.

\end{enumerate}
\end{remark}

\begin{remark}
\label{WhatsInAdelta}
It is helpful to have some examples to distinguish the various regularity classes above.
\begin{enumerate}
\item If $s\in \mathbb C$ then for any $l\in \mathbb N$ we have $\rho^s(\log\rho)^l\in \mathcal A_\delta(M)$ if $\delta = \Re(s)$, but $\rho^s(\log\rho)^l$ is not in $\rho^\delta \mathcal A(M)$.

\item Furthermore, $\rho^s(\log\rho)^l$ is polyhomogeneous, but is in $C^{k,\alpha}_\phg(\bar M)$ only if $\Re(s)>k+\alpha$.

\item Finally, if $\epsilon>0$ and $v\in C^\infty(\bar M)$ is not constant along $\partial M$, then the function $\rho^\epsilon\sin{(v\log\rho)}$ is an element of both $\mathcal A_\epsilon(M)$ and $C^{k,\alpha}_0(M)$ for all $k$ and $\alpha$, but is neither in $\rho^\epsilon \mathcal A(M)$ nor in $\mathcal A_\phg(M)$.

\end{enumerate}
\end{remark}

The following lemma records several important relationships among these regularity classes.

\begin{lemma}
\label{InclusionLemma}\hfill
\begin{enumerate}
\item 
\label{Inclusion-PhgConormal}
If $u\in \mathcal A_\phg(M)$ with leading exponent $s_0$ in expansion \eqref{GenericTensorExpansion}, then $u\in \mathcal A_\delta(M)$ for $\delta = \Re(s_0)$; thus
\begin{equation*}
\mathcal A_\phg(M)\subset\mathcal A_{-\infty}(M).
\end{equation*}

\item 
\label{Inclusion-PhgWeightDelta}
If $\alpha \in [0,1)$ and $\delta\in \mathbb R$ then
for tensor fields of weight $r$ we have
\begin{equation*}
\mathcal A_{\phg}(M) \cap C^{0,\alpha}_{\delta}(M)
\subset \rho^\delta \mathcal A(M)
\cap \rho^{\delta-r} C^0_\phg(\bar M).
\end{equation*}

\item
\label{Inclusion-DeltaIntrinsic}
If $k\in \mathbb N_0$, $\alpha \in [0,1)$, and $\delta^\prime <\delta$, then
\begin{equation*}
\mathcal A_\delta(M)
\subset \rho^{\delta^\prime}\mathcal A(M)
\subset C^{k,\alpha}_{\delta^\prime}(M).
\end{equation*}

\item
\label{Inclusion-Compact}
If $k\in \mathbb N_0$ and $\alpha\in [0,1)$ then for tensor fields of weight $r$ we have
\begin{equation*}
C^{0}_\phg(\bar M) \subset   C^{k,\alpha}_r(M).
\end{equation*}

\item
\label{Inclusion-PhgImprovement}
If $u\in C^{k}_\phg(\bar M)$ then there exists $\gamma\in(0,1)$ such that $u\in C^{k,\gamma}_\phg(\bar M)$.

\end{enumerate}

\end{lemma}

\begin{proof}
The first claim follows from Remarks \ref{WhatsInAdelta} and \ref{PhgSplitting}.

To prove the remaining claims, first observe that in the image of a M\"obius param\-e\-tri\-zation $\Phi_{p_0}$ the defining function $\rho$ is comparable to $\rho_0 = \rho(p_0)$ and $|\Phi_{p_0}^*u(x,y)| = \rho_0^r |u(\theta_0 + \rho_0x,\rho_0y)|$; thus for scalar functions $|\partial_{y} \Phi_{p_0}^*f(x,y)|\approx  \left|\rho\partial_{\rho} f\right|_{\Phi(x,y)}$, etc.
Claim \eqref{Inclusion-PhgWeightDelta} then follows by noting that the weight $\delta$ places restrictions on the leading exponent of the polyhomogeneous expansion; H\"older continuity implies that there can be no ``leading log term'' and thus $\rho^{-\delta+r}u$ extends continuously to $\bar M$.
The third claim follows from the definitions of the spaces involved, while direct computation shows that $C^{0,\alpha}(\bar M) \subset  C^{0,\alpha}_{r}(M)$ and the fourth claim follows from considering M\"obius parametrizations as above.

The final claim is due to the discreteness of the sequence $\{s_i\}$ appearing in the polyhomogeneous expansion of $u$.
\end{proof}

The following is an immediate consequence of Lemma \ref{InclusionLemma}(\ref{Inclusion-PhgWeightDelta},\ref{Inclusion-Compact}) and Proposition \ref{thm:properties-of-WAH}.

\begin{corollary}
\label{PhgMetricsAreWeak}
Suppose that $\bar g\in C^{m}_\phg(\bar M)$ is a Riemannian metric on $\bar M$ for some $m\geq 1$, and that $|\D\rho|_{\bar g} =1$ along $\partial M$.
Then $g = \rho^{-2}\bar g \in \WAH^{k,\alpha;m}$ for all $k\geq m$ and for all $\alpha \in [0,1)$.

Conversely if $g\in \WAH^{k,\alpha;m}$ for some $m\geq 1$ and $g\in \mathcal A_\phg(M)$, then $\bar g = \rho^2 g\in C^{m}_\phg(\bar M)$.
\end{corollary}

\subsection{Analysis of the indicial operator}
We now restrict attention to the case where $g\in \WAH^{k,\alpha;2}\cap\mathcal A_\phg(M)$, and thus $\bar g = \rho^2 g\in C^2_\phg(\bar M)$, and investigate the boundary regularity of solutions $u$ to \eqref{GenericEllipticEquation}.

We first construct from the indicial map $I_s(\mathcal P)$, defined in \eqref{DefineIndicialMap}, a differential operator which, in the polyhomogeneous setting, approximates $\mathcal P$ in the $\rho$ direction; see Lemma \ref{LocalApproximationLemma} below.
Following \cite{Mazzeo-Edge}, we define the \Defn{indicial operator} $I(\mathcal P)$ to be the unique dilation-invariant operator on $\partial M\times (0,\infty)$ satisfying $\rho^{-s}I(\mathcal P)(\rho^s u) = I_s(\mathcal P)u$ for all smooth sections $u$ of $\left.E\right|_{\partial M}$.
 
In background coordinates $\Theta = (\theta,\rho)$, in which  $\mathcal P$ takes the form \eqref{OperatorInCoordinates}, we have by direct computation that
\begin{equation}
\label{ComputeIndicialOperator}
I_s(\mathcal P) = \left( s^2 \bar a + s \bar b + \bar c \right),
\end{equation}
where we have set $\bar a =  \left.a^{\rho\rho}\right|_{\rho=0}$, $\bar b  =\left.b^{\rho}\right|_{\rho=0}$, and $\bar c =\left.c\right|_{\rho=0}$.
Thus the operator $I(\mathcal P)$ is given by
\begin{equation}
\label{LocalIndicialOperator}
I(\mathcal P) = \bar a (\rho \partial_\rho)^2 + \bar b (\rho\partial_\rho) + \bar c.
\end{equation}
We emphasize that the coefficient matrices $\bar a$, $\bar b$, $\bar c$ are the expressions in coordinates of endomorphisms of $\left.E\right|_{\partial M}$ and thus are functions only of $\theta$; we furthermore note that the ellipticity of $\mathcal P$ implies that $\bar a$ is invertible.

Identifying, as above, the collar neighborhood $\mathcal C$ with $\partial M\times (0,\rho_*)$ we extend $I(\mathcal P)$ to an operator $\mathcal I(\mathcal P)$ on $M$ by choosing a smooth cutoff function $\varphi\colon (0,\infty)\to [0,\infty)$ satisfying $\varphi \equiv 1$ on $(0,\frac12 \rho_*]$ and $\varphi \equiv 0$ for $\rho \geq \frac23\rho_*$ and setting $\mathcal I(\mathcal P) = \varphi\, I(\mathcal P)$.
We furthermore define $\mathcal R := \mathcal P - \mathcal I(\mathcal P)$.
The operator $\mathcal I(\mathcal P)$ approximates $\mathcal P$ in the following sense.

\begin{lemma}
\label{LocalApproximationLemma}
Suppose that $g\in \WAH^{k,\alpha;2}\cap\mathcal A_\phg(M)$, and that $\mathcal P$ satisfies Assumption \ref{Assume-P}.
There exists $\gamma \in (0,1]$ such that if $u\in \mathcal A_\delta(M)$ for some $\delta\in \mathbb R$, then $\mathcal R u\in  \mathcal A_{\delta+\gamma}(M)$.
\end{lemma}

\begin{proof}
It suffices to work in that portion of a background coordinate chart $(\mathcal U, \Theta)$ where  $\mathcal I(\mathcal P) = I(\mathcal P)$.
The claim then follows from carefully examining the background coordinate expression \eqref{OperatorInCoordinates} of $\mathcal P$, which is a sum of
\begin{enumerate}
\item terms of the form $\rho \,f\, (\rho\partial_\rho)^k \partial_{\theta^1}^{l_1}\dots \partial_{\theta^m}^{l_m}$, where $1\leq l_1+\dots+l_m$ and $l_1+\dots+l_m+k\leq 2$ and $f\in \mathcal A(\mathcal U)$, and
\item the operator $a^{\rho\rho}(\rho\partial_\rho)^2 + b^\rho(\rho\partial_\rho) + c$.
\end{enumerate}
Operators of the first type clearly map $\mathcal A_\delta(\mathcal U)$ to $\mathcal A_{\delta+1}(\mathcal U)$.
The polyhomogeneity of $g$, and thus of the coefficients of $\mathcal P$, implies that for some $\gamma \in(0,1]$ we can write
\begin{equation*}
a^{\rho\rho} = \bar{a} + \rho^\gamma \tilde a, \qquad
b^{\rho} = \bar{b} + \rho^\gamma \tilde b, \qquad
c = \bar c + \rho^\gamma \tilde c,
\end{equation*}
with $\tilde a, \tilde b, \tilde c \in \mathcal A(\mathcal U)$ and $\bar a, \bar b, \bar c$ as in \eqref{ComputeIndicialOperator}.
Thus 
$
a^{\rho\rho}(\rho\partial_\rho)^2 + b(\rho\partial_\rho) + c
= I(\mathcal P) +  J,
$
where $ J$ takes $\mathcal A_\delta(\mathcal U)$ to $\mathcal A_{\delta+\gamma}(\mathcal U)$ for all $\delta\in \mathbb R$.
\end{proof}

\begin{remark}
We remark that if $\{s_i\}$ is the sequence of exponents appearing in the polyhomogeneous expansion of the coefficients of $\mathcal P$, then the constant $\gamma$ appearing in the lemma is simply a lower bound on the ``first gap'' in the sequence $\{\Re(s_i)\}$.
\end{remark}

The previous lemma suggests that the boundary behavior of solutions to \eqref{GenericEllipticEquation} can be understood by studying $\mathcal I(\mathcal P)$.
We proceed by first showing that on the collar neighborhood $\mathcal C$ of $\partial M$, $\mathcal I(\mathcal P)$ is comparable to the corresponding operator in hyperbolic space.
To this end, denote by $\breve E$ the tensor bundle over $(\mathbb H,\breve g)$ corresponding to the same representation of $O(n+1)$ as $E$, and define $\breve{\mathcal P} = \mathcal P[\breve g]$ to be the geometric operator on $\breve E$ given in coordinates by the same formula as $\mathcal P$.
The operator $\breve{\mathcal P}$ is invariant under isometries of $(\mathbb H,\breve g)$; 
thus the indicial map $I_s(\breve{\mathcal P})$ is translation-invariant along $\{y=0\}$.
Consequently the characteristic exponents of $\breve{\mathcal P}$ and their multiplicities, as well as the coefficients (in Cartesian coordinates) of the indicial operator $I(\breve{\mathcal P})$, are constant as well.

\begin{lemma}
\label{IndicialComparisonLemma}
Suppose $g\in \WAH^{k,\alpha;2}$ and $\mathcal P$ satisfies Assumption \ref{Assume-P}.
\begin{enumerate}
\item The characteristic exponents of $\mathcal P$ and their multiplicities are constant along $\partial M$, and agree with those of $\breve{\mathcal P}$.

\item 
\label{LocalIndicialMapComparison}
Assume that $\left.\bar g\right|_{\partial M} \in C^{l}(\partial M)$.
Then for each of the finitely many coordinate charts $(U,\theta)$ on $\partial M$ used to construct the background coordinate charts there exists an invertible matrix $S\in C^{l}(U)$ such that on $U$ we have
\begin{equation*}
I_s(\mathcal P) = S^{-1} I_s(\breve{\mathcal P}) S.
\end{equation*}

\item Assume that $\left.\bar g\right|_{\partial M} \in C^{l}(\partial M)$, and let $(U,\theta)$ and $S$ be as in point \eqref{LocalIndicialMapComparison} above.
Then the restriction of $I(\mathcal P)$ to $U\times (0,\infty)$ satisfies 
\begin{equation*}
I(\mathcal P) = S^{-1} I(\breve{\mathcal P}) S.
\end{equation*}
\end{enumerate}
\end{lemma}

\begin{proof}
The first claim is the content of Lemma 4.3 of \cite{Lee-FredholmOperators}, the proof of which we summarize here.
First, fix $\hat p \in U\subset \partial M$ and use $\theta$ to identify $U$ with an open subset of $\mathbb R^n = \{y=0\}\subset \bar{\mathbb H}$.
Through an affine change of coordinates $\theta$, we may arrange that $\hat p$ corresponds to the origin and that $\bar g_{ij} = \delta_{ij}$ there.

The proof of the first claim follows by showing that $I_s(\mathcal P) = I_s(\breve{\mathcal P})$ at the origin.
This, in turn, is obtained by carefully examining the various types of terms which may appear in a geometric operator and showing that for each type the difference between a term arising from $g$ and the corresponding term arising from $\breve g$ has vanishing indicial map.
For example, the difference tensor $\nabla - \breve\nabla$ has components $E^i_{jk} = \rho^{-1}\partial_j\rho(\bar g^{il}\bar g_{kl} - \delta^{il}\delta_{kl})+\mathcal O(\rho)$, and thus the fact that 
$\bar g_{ij} = \delta_{ij}$ at the origin implies that the map $u\mapsto \rho\nabla u - \rho\breve{\nabla} u$ vanishes there.

The second claim relies on observing that the aforementioned affine change of coordinates is based on the Gram-Schmidt algorithm and therefore consists of rational functions of the components of $\left.\bar g\right|_{\partial M}$.
Thus at each point the matrix taking the background coordinate frame to the standard Cartesian coordinate frame is as regular as the metric $\left.\bar g\right|_{\partial M}$.

The third claim follows from the coordinate expressions for the indicial map \eqref{ComputeIndicialOperator} and for the corresponding indicial operator \eqref{LocalIndicialOperator}.
\end{proof}

The previous lemma allows us to understand, in the polyhomogeneous setting, solutions to $\mathcal I(\mathcal P)u = f$ if $f$ vanishes near the boundary.
Let $C\subset \mathbb C$ be the (finite) collection of characteristic exponents of $\mathcal P$.

\begin{lemma}
\label{Lemma-AlmostHomogeneous}
Suppose $g\in \WAH^{k,\alpha;2}\cap\mathcal A_\phg(M)$, and suppose that $\mathcal P$ satisfies Assumption \ref{Assume-P}.
If  $w\in C^\infty(M)$ satisfies $\mathcal I(\mathcal P)w = f$ with $f$ vanishing on the collar neighborhood $\mathcal C_{a}$ for some $a\in (0,\rho_*)$, then $w\in \mathcal A_\phg^{C+r}(M)$, where $r$ is the weight of $w$.
\end{lemma}

\begin{proof}
It suffices to work in that part of the background coordinate chart $(\mathcal U, \Theta)$ where  $\mathcal I(\mathcal P) = I(\mathcal P)$ and $f=0$.
Working in coordinates, we view $w$ as a matrix-valued function; note that this involves a shift by $r$ in the set of exponents in polyhomogeneous expansion of $w$ that we construct; see Remark \ref{DamnBundleWeights}.

In view of Lemma \ref{IndicialComparisonLemma}, we have that $I(\mathcal P)w=0$  precisely if $v=Sw$ is a solution to
\begin{equation}
\label{HomogeneousModelODE}
I(\breve{\mathcal P})v:= \breve a(\rho\partial_\rho)^2v + \breve b \rho\partial_\rho v + \breve c v =0.
\end{equation}
Note that the polyhomogeneity of $g$ implies that $\left.\bar g\right|_{\partial M}\in C^\infty(\partial M)$ and thus $S$ is smooth.

We now analyze \eqref{HomogeneousModelODE}, expressing it as the first order system
\begin{equation}
\label{HomogeneousFirstOrderSystem}
\rho\partial_\rho \mathbf v = \mathbf A \mathbf v
\end{equation}
by introducing the auxiliary variable $w = \rho\partial_\rho v$ and setting $\mathbf v = (v,w)^t$; here $\mathbf A$ is the matrix of constants given by 
\begin{equation*}
\mathbf A=\begin{pmatrix} 0 & 1 \\ -\breve a^{-1} \breve c & -\breve a^{-1}\breve b\end{pmatrix}.
\end{equation*}

The eigenvalues of $\mathbf A$ are precisely the characteristic exponents of $\breve{\mathcal P}$ which, in view of Lemma \ref{IndicialComparisonLemma}, agree with those of $\mathcal P$.
All solutions to \eqref{HomogeneousFirstOrderSystem} take the form $\mathbf v = \exp(\mathbf A \log\rho)\mathbf v_0$, where $\mathbf v_0 = \mathbf v_0(\theta)$ is free.
The entries of the matrix exponential $\exp(\mathbf A \log\rho)$ are easily seen to be linear combinations of $\rho^s(\log\rho)^k$ with $s\in C$ and non-negative integers $k$ less than the dimension of $E$; this follows from analyzing the exponential of the Jordan form of $\mathbf A$ (see, for example, Chapter 3 of \cite{Teschl-ODEs}).
Consequently, if the free data $\mathbf v_0$ is smooth in $\theta$ then the corresponding homogeneous solution lies in $\mathcal A_\phg^C(\mathcal U)$.
Finally, note that $v$, the first component of $\mathbf v$, satisfies $I(\breve{\mathcal P})v =0$, and thus $w = S^{-1}v\in \mathcal A_\phg^C(\mathcal U)$ is the corresponding solution to $I(\mathcal P)w =0$.
Adapting the expansion to the normalized background coordinate frame yields the result; see Remark \ref{DamnBundleWeights}.
\end{proof}

We now define an operator $\mathcal G$ which we use below to study solutions to $\mathcal I(\mathcal P)u =f$.

\begin{proposition}
\label{Proposition-DefineG}
Suppose $g\in \WAH^{k,\alpha;2}\cap\mathcal A_\phg(M)$, and suppose that $\mathcal P$ satisfies Assumption \ref{Assume-P} and has characteristic exponents  $C\subset\mathbb C$.
Then there exists an operator $\mathcal G\colon C^\infty(M)\to C^\infty(M)$ such that 
\begin{enumerate}

\item 
\label{VanishNearBdy}
for $a\in (0,\rho_*/2)$ we have that
\begin{equation*}
\left.(\mathcal I(\mathcal P)\circ \mathcal G)(f)\right|_{\mathcal C_a} = \left. f \right|_{\mathcal C_a},
\end{equation*} 

\item for any $\delta\in \mathbb R$ we have that  $f\in \mathcal A_\delta(M)$ implies $\mathcal G(f)\in \mathcal A_\delta(M)$, and

\item for any $S\subset \mathbb C$ we have that  $f\in \mathcal A_\phg^S(M)$ implies $\mathcal G(f)\in \mathcal A_\phg^{S\cup (C+r)}(M)$, where $r$ is the weight of tensor field $u$.

\end{enumerate}

\end{proposition}

\begin{proof}
Let $\varphi$ be the same cutoff function used to define $\mathcal I(\mathcal P)$.
Restrict $f\in C^\infty(M)$ to $\mathcal C$, which we identify with $\partial M\times (0,\rho_*)$, and extend $\varphi f$ to $\tilde f$, smoothly defined on $\partial M\times (0,\infty)$, by $\tilde f =0$ for $\rho\geq \rho_*$; note that $f$ agrees with $\tilde f$ on $\mathcal C_a$ for all $a\leq \rho_*/2$.

We now consider $I(\mathcal P) \tilde u = \tilde f$ as a second-order linear ordinary differential equation in $\rho$.
Existence of a unique, smooth solution $\tilde u$, defined for all $\rho>0$, satisfying $\left.\tilde u\right|_{\rho=\rho_*} =0$ and $\left.\partial_\rho \tilde u\right|_{\rho=\rho_*}=0$ is guaranteed by the classical Cauchy-Lipschitz-Picard-Lindel\"of theorem.

Note that $\tilde u =0$ for all $\rho \geq \frac23\rho_*$.
Thus restricting $\tilde u$ to $\mathcal C$ and then extending trivially we obtain $u\in C^\infty(M)$ such that
\begin{equation*}
\left.\mathcal I(\mathcal P)u\right|_{\mathcal C_a} = \left.f \right|_{\mathcal C_a}
\end{equation*}
for all $0<a<\frac12\rho_*$. Defining $\mathcal G$ by $f\mapsto u = \mathcal G(f)$, the first claim of the proposition holds by construction.

In order to verify the remaining claims, it suffices to study the behavior of $\mathcal G(f)$ in that portion of a background coordinate chart $(\mathcal U,\Theta)$ where $\mathcal I(\mathcal P) = I(\mathcal P)$.
To this end, with $\mathcal U = U\times(0,\rho_*)$, we study $I(\mathcal P)\tilde u = \tilde f$ on $U\times (0,\infty)$.

As in the proof of Lemma \ref{Lemma-AlmostHomogeneous}, it suffices to study the model problem $I(\breve{\mathcal P})v =\breve f$, where $v = S\tilde u$ and $\breve f = S\tilde f$ for smooth $S=S(\theta)$.
We write the model as the first order system
\begin{equation}
\label{ModelFirstOrderSystem}
\rho\partial_\rho \mathbf v = \mathbf A \mathbf v + \mathbf f
\end{equation}
with $\mathbf v$ and $\mathbf A$ as before, and $\mathbf f = (0, \breve a^{-1} \breve f)^t$.
The solution to \eqref{ModelFirstOrderSystem} corresponding to $\tilde u$ must satisfy $\left.\mathbf v\right|_{\rho=\rho_*}=0$ and thus is given by 
\begin{equation}
\label{Duhamel}
\mathbf v(\theta,\rho) = \exp{(\mathbf A\log\rho)}\int_{\rho_*}^\rho \exp{(-\mathbf A\log\sigma)} \mathbf f(\theta,\sigma)\frac{1}{\sigma}\,\D\sigma.
\end{equation}

In order to establish the second claim it suffices to consider the derivatives $\partial_{\theta^i}\mathbf v$ and $\rho\partial_\rho\mathbf v$, as well as higher-order derivatives $(\rho\partial_\rho)^l(\partial_\theta)^m\mathbf v$.
That these are bounded by the corresponding derivatives of $f$ follows from the translation invariance of $\mathbf A$ and $\breve a$, and the identity $\rho\partial_\rho\mathbf v = \mathbf A \mathbf v + \mathbf f$.

In the polyhomogeneous setting it suffices to understand the structure of \eqref{Duhamel} in the case that $\mathbf A$ is a single Jordan block $s \mathbf I + \mathbf N$, where $s$ an eigenvalue of $\mathbf A$ and $\mathbf N$  is nilpotent, and that $S$ is finite.
In this case $\exp{(\mathbf A\tau)}$ is an upper-triangular matrix with $e^{\tau s}$ along the diagonal and entries of the form $e^{\tau s}p(\tau)$, with $p$ some polynomial, above the diagonal.
Taking $\tau = \log\rho$ it is straightforward to verify that if $f\in \mathcal A_\phg^S(M)$, and hence $\mathbf f\in \mathcal A_\phg^S(\mathcal U)$, then $\mathbf v \in \mathcal A_\phg^{S\cup C}(\mathcal U)$.
The third claim follows from adapting the expansion to a normalized frame.
\end{proof}

\begin{remark}
\label{ParticularDecomposition}
Lemma \ref{Lemma-AlmostHomogeneous} and Proposition \ref{Proposition-DefineG} imply that if $\mathcal I(\mathcal P)u =f$, then $u = \mathcal G(f) + w$, where $w\in \mathcal A_\phg^{C+r}(M)$ and $\mathcal I(\mathcal P)w\in \mathcal A_{\delta}(M)$ for all $\delta\in \mathbb R$.
\end{remark}

\begin{remark}
\label{LogsComeFromMultiplicity}
As is evident from the proofs of Lemma \ref{Lemma-AlmostHomogeneous} and Proposition \ref{Proposition-DefineG}, the presence of logarithms in expansions of solutions to $\mathcal I(\mathcal P) u = f$ is a consequence of the algebraic structure of $\mathcal P$, and the exponents appearing in the expansion of $f$.
In particular, logarithms appear either if two characteristic exponents differ by an integer, or in the resonant case, if the expansion of $f$ includes a characteristic exponent.

\end{remark}

\subsection{Boundary regularity}
In this subsection we prove the following boundary regularity theorem.

\begin{theorem}
\label{SolutionsArePolyhomogeneous}
Suppose that $g\in \WAH^{k,\alpha;2}\cap\mathcal A_\phg(M)$, $\mathcal P$ satisfies Assumption \ref{Assume-P}, and that $f$ is polyhomogeneous. 
Suppose $\alpha \in (0,1)$ and $|\delta - \frac{n}{2}|<R$, where $R$ is the indicial radius of $\mathcal P$.
If $u\in C^{2,\alpha}_\delta (M)$ is a solution to $\mathcal P u = f$, then $u$ is polyhomogeneous.
\end{theorem}

We divide the proof of Theorem \ref{SolutionsArePolyhomogeneous} into two steps, showing first that $u$ is conormal and subsequently that it is polyhomogeneous.
Conormality is established by showing that, for $V\in \mathscr V_b$, $\mathcal L_Vu$ is in the same weighted H\"older space as $u$.
As $TM$ has weight $-1$, commuting $\mathcal L_V$ into the equation $\mathcal P u = f$ leads to a loss of weight; this loss can be recovered using Proposition \ref{ImprovedRegularity} if the indicial radius $R$ is greater than $1/2$; see e.g.~ \cite{AnderssonChruscielFriedrich,LeeMelrose-MongeAmpere}.
Here we follow an alternate approach, obtaining bounds on $\mathcal L_V$ by estimating difference quotients via Proposition \ref{StrongRegularity}; cf.~\cite{AnderssonChrusciel-Dissertationes}.

For $V\in \mathscr V_b$, denote by $\psi_V(\epsilon):\bar M\to \bar M$ the diffeomorphism obtained by flowing along integral curves of $V$ for time $\epsilon$.
Since $V$ is tangent to $\partial M$, and since $\bar M$ is compact, for each $V\in \mathscr V_b$ there exists some $\epsilon_*>0$ such that $\psi_V(t)$ is defined when $|\epsilon|\leq \epsilon_*$.
Define the difference operator, acting on a tensor field $u$, by $\Delta_V^\epsilon u = \psi_V(\epsilon)^*u - u$; thus
\begin{equation*}
\mathcal L_Vu 
= \frac{d}{d\epsilon}\left[\psi_V(\epsilon)^*u\right]_{\epsilon=0}
= \lim_{\epsilon \to 0}\left[\frac{\Delta_V^\epsilon u}{\epsilon}\right].
\end{equation*}
We record some elementary facts regarding difference operators; while stated for $V\in \mathscr V_b$, they hold for any vector field $V$, provided $\Delta_V^\epsilon$ is well-defined.
\begin{lemma}
\label{DifferenceOperatorProperties}
For each $V\in \mathscr V_b$ there exists $\epsilon_\ast>0$ such that we have the following.
\begin{enumerate}
\item 
\label{part:basic-difference-bound}
For each $k\geq 1$ there exists a constant $C$ such that for all  $ u\in C^{k}_\delta (M)$ we have 
$$
\|\Delta_V^\epsilon u\|_{C^{k-1}_\delta(M)} \leq \epsilon C \|\mathcal L_Vu\|_{C^{k-1}_\delta(M)}
$$
for all $\epsilon\in(0,\epsilon_\ast]$.

\item 
\label{part:compact-difference-bound}
For any $k\geq 1$ and for any compact set $K\subset M$ there exists a constant $C$ such that if $u\in C^{k}_\delta(M)$, then we have 
$$
\|\Delta_V^\epsilon u\|_{C^{k-1}_\delta(K)}\leq \epsilon C \|u\|_{C^{k}_\delta(M)}
$$
for all $\epsilon\in(0,\epsilon_\ast]$.

\end{enumerate}
\end{lemma}

\begin{proof}
For any tensor field $w$, we may integrate $\mathcal L_V w$ along the flow associated to $V$, obtaining
\begin{equation}
\label{DifferenceIntegral}
\Delta_V^\epsilon w = \int_0^\epsilon \psi_V(\sigma)^*(\mathcal L_V w)\,d\sigma.
\end{equation}
This implies that $\|\Delta_V^\epsilon w\|_{C^0(M)} \leq \epsilon C \|\mathcal L_V w\|_{C^0(M)}$.
The first claim then follows from differentiating \eqref{DifferenceIntegral} in background coordinates and observing that $|\Delta_V^\epsilon\rho| = \mathcal O(\epsilon)$, while the second claim is a consequence of $\left.\rho\right|_K$ being uniformly bounded away from zero.
\end{proof}

The following commutator estimates rely essentially on $V$ being in $\mathscr V_b$.

\begin{lemma}
\label{CommutatorEstimates}
Suppose that $(M,g)$ and $\mathcal P$ satisfy the hypotheses of Theorem \ref{SolutionsArePolyhomogeneous}, and that $V\in \mathscr V_b$.
Furthermore, let $k\geq 2$. 

\begin{enumerate}
\item Let $\epsilon_*>0$ be as in Lemma \ref{DifferenceOperatorProperties}.
Then for any $u\in C^{k}_\delta(M)$ we have
\begin{equation*}
\|\left[ \mathcal P,\Delta_V^\epsilon\right]u\|_{C^{k-2}_\delta(M)}
\leq \epsilon C \|u\|_{C^{k}_\delta(M)}
\end{equation*}
for all $\epsilon\in(0,\epsilon_\ast]$.

\item The commutator $[\mathcal P, \mathcal L_V]$ is a uniformly degenerate operator and thus for any $u\in C^{k}_\delta(M)$ we have $[\mathcal P, \mathcal L_V] u \in C^{k-2}_\delta(M)$ and
\begin{equation*}
\|[\mathcal P, \mathcal L_V] u\|_{C^{k-2}_\delta(M)}
\leq
C \| u\|_{C^k_\delta(M)}.
\end{equation*}
Furthermore, if $k\geq 3$ and $W u \in C^{k-1}_\delta(M)$ for all $W\in \mathscr V_b$, then 
\begin{equation*}
\|\mathcal L_W [\mathcal P, \mathcal L_V] u\|_{C^{k-3}_\delta(M)}
\leq
C \| u\|_{C^k_\delta(M)}.
\end{equation*}

\end{enumerate}

\end{lemma}

\begin{proof}
In background coordinates $(\Theta^\mu)$ we have 
\begin{gather*}
\left|\delta_\mu^\alpha - \partial_\mu \psi_V(\epsilon)^\alpha\right| = \mathcal O(\epsilon),
\qquad
\left|\partial_\mu\partial_\nu\psi_V(\epsilon)^\alpha\right| = \mathcal O(\epsilon),
\end{gather*}
etc.;
see the proof of Theorem D.5 in \cite{Lee-SmoothManifolds2}.
Directly inspecting the background coordinate expression of
\begin{equation*}
\left[ \mathcal P,\Delta_V^\epsilon\right]u
= \mathcal P\left(\psi_V(\epsilon)^*u\right) - \psi_V(\epsilon)^*\left(\mathcal P u\right)
\end{equation*}
leads to the first estimate.

The second claim follows from direct inspection of the commutator term, together with fact that the coefficients of $\mathcal P$ are polyhomogeneous, and thus conormal.
\end{proof}

We now use difference operators to establish conormality of solutions to $\mathcal P u = f$.
\begin{proposition}
\label{SolutionIsConormal}
Suppose that $g\in \WAH^{k,\alpha;2}\cap\mathcal A_\phg(M)$ and that $\mathcal P$ satisfies Assumption \ref{Assume-P}.
Suppose furthermore that $\alpha\in (0,1)$ and that $|\delta - \frac{n}{2}|<R$, where $R$ is the indicial radius of $\mathcal P$.
Finally, suppose $u\in C^{2,\alpha}_\delta(M)$ satisfies $\mathcal P u = f$, with $f$ polyhomogeneous.
Then $u\in \mathcal A_\delta(M)$; i.e.~$u$ is conormal.
\end{proposition}

\begin{proof}
We first note that $f = \mathcal P u \in C^{0,\alpha}_\delta(M)\cap \mathcal A_\phg(M)$; thus by Lemma \ref{InclusionLemma} we have $f\in C^{k,\alpha}_\delta(M)$ for all $k$. 
Lemma \ref{ImprovedRegularity} implies that  $u\in C^{k,\alpha}_\delta(M)$ for all $k$ as well.

Fixing $V\in \mathscr V_b$, we see that $\Delta_V^\epsilon u$ satisfies
\begin{equation*}
\mathcal P(\Delta_V^\epsilon u) = \left[\mathcal P, \Delta_V^\epsilon\right]u + \Delta_V^\epsilon f.
\end{equation*}
Using Proposition \ref{StrongRegularity} we have, for any $k\geq 2$ and $\alpha\in (0,1)$, that
\begin{equation}
\label{EstimateDifferenceQuotient}
\begin{aligned}
\|\Delta_V^\epsilon u\|_{C^{k}_\delta(M)} 
&\leq \|\Delta_V^\epsilon u\|_{C^{k,\alpha}_\delta(M)}
\\
&\leq C\left( \| \left[\mathcal P, \Delta_V^\epsilon \right]u\|_{C^{k-2,\alpha}_\delta(M)}
\right.
\\
&\qquad \left.
	+ \|\Delta_V^\epsilon f \|_{C^{k-2,\alpha}_\delta(M)}
	+ \|\Delta_V^\epsilon u\|_{C^{k,\alpha}_\delta(K)}
\right)
\\
&\leq C\left( \| \left[\mathcal P, \Delta_V^\epsilon \right]u\|_{C^{k-1}_\delta(M)}
\right.
\\
&\qquad \left.
	+ \|\Delta_V^\epsilon f \|_{C^{k-1}_\delta(M)}
	+ \|\Delta_V^\epsilon u\|_{C^{k+1}_\delta(K)}
\right)
\end{aligned}
\end{equation}
for some compact set $K\subset M$.
Using Lemma \ref{DifferenceOperatorProperties}\eqref{part:basic-difference-bound} we have
$
\|\Delta_V^\epsilon f \|_{C^{k-1}_\delta(M)} = \mathcal O(\epsilon).
$
Furthermore, Lemma \ref{DifferenceOperatorProperties}\eqref{part:compact-difference-bound} implies that
$
\|\Delta_V^\epsilon u\|_{C^{k+1}_\delta(K)}
= \mathcal O(\epsilon),
$
while Lemma \ref{CommutatorEstimates} implies that
$
\| \left[\mathcal P, \Delta_V^\epsilon \right]u\|_{C^{k-1}_\delta(M)}
= \mathcal O(\epsilon).
$
Consequently, from \eqref{EstimateDifferenceQuotient} we have $\|\Delta_V^\epsilon u\|_{C^{k}_\delta(M)} = \mathcal O(\epsilon)$ and hence $\mathcal L_Vu\in C^{k}_\delta(M)$ for all $k$.

Proceeding by induction, we assume for some integer $l$ that for any $\{V_1,\dots,V_m\}\subset \mathscr V_b$ we have $w_m = \mathcal L_{V_1}\cdots \mathcal L_{V_m}u \in C^{k}_\delta(M)$ for all $k\geq0$.
Fixing $V\in \mathscr V_b$, we see that $\Delta_V^\epsilon w_m$ satisfies
\begin{equation*}
\mathcal P(\Delta_V^\epsilon w_m)
= \left[\mathcal P, \Delta_V^\epsilon\right]w_m
+\Delta_V^\epsilon(\mathcal L_{V_1}\dots \mathcal L_{V_l} f)
+\Delta_V^\epsilon(\left[\mathcal P,\mathcal L_{V_1}\dots \mathcal L_{V_l}\right]u).
\end{equation*}
Using Lemma \ref{CommutatorEstimates}, we see that 
\begin{gather*}
\left[\mathcal P,\mathcal L_{V_1}\dots \mathcal L_{V_m}\right]u \in C^{k}_\delta(M),
\\
\mathcal L_V\left[\mathcal P,\mathcal L_{V_1}\dots \mathcal L_{V_m}\right]u \in C^{k}_\delta(M)
\end{gather*}
for all $k\geq 0$.
We now invoke Proposition \ref{StrongRegularity}, obtaining estimates analogous to \eqref{EstimateDifferenceQuotient} for $\Delta_V^\epsilon w_m$.
Proceeding as above, we find $\mathcal L_V w_m \in C^k_\delta(M)$ for all $k \geq 0$.
Thus by induction on $m$ we obtain  $u\in \mathcal A_\delta(M)$.
\end{proof}

\begin{proof}[Proof of Theorem \ref{SolutionsArePolyhomogeneous}]
In view of Proposition \ref{SolutionIsConormal}, we have that the solution $u$ to \eqref{GenericEllipticEquation} is conormal; thus $u\in \mathcal A_{\delta}(M)$ for some $\delta\in \mathbb R$.

Using Lemma \ref{LocalApproximationLemma}, we write $\mathcal P = \mathcal I(\mathcal P) + \mathcal R$ and fix $\gamma$ as in that lemma.
We proceed inductively, constructing a sequence of approximate solutions $u_k$ such that $u_k \in \mathcal A_\phg^{S_k}(M) \cap \mathcal A_\delta(M)$ for some finite sets $S_k\subset \mathbb C$, and such that
$f_k:= f - \mathcal P u_k \in \mathcal A_{\delta+k\gamma}(M)$.
We further arrange that $r_k:= u - u_k \in \mathcal A_{\delta + k\gamma}(M)$ and that $r_{k+1} - r_k \in \mathcal A_{\delta+k\gamma}(M)$ for sufficiently large $k$.

When $k=0$ we set $u_0 =0$ and, as $f = \mathcal P u \in \mathcal A_\delta(M)$, we have nothing to prove.
For convenience, we set $S_0 = C+r$, the finite collection of characteristic exponents of $\mathcal P$, shifted by the weight $r$ of $u$ (see Remark \ref{DamnBundleWeights}).

Suppose now that $u = u_k + r_k$ satisfies the inductive hypothesis above.
The remainder $r_k$ satisfies 
\begin{equation}
\label{InductiveMain}
\mathcal P r_k = f_k.
\end{equation}
Using Remark \ref{PhgSplitting}, we can write $f_k = f_k^\text{fin} + f_k^\text{rem}$, where $f_k^\text{fin}\in \mathcal A_\phg^{T_k}(M) \cap \mathcal A_{\delta + k\gamma}(M)$ for some finite set
\begin{equation*}
T_k \subset \{ s \in \mathbb C \mid \delta+k\gamma \leq \Re(s) \leq \delta+(k+1)\gamma\}
\end{equation*}
and $f_k^\text{rem}\in \mathcal A_{\delta+(k+1)\gamma}(M)$.
We rewrite \eqref{InductiveMain} as
\begin{equation*}
\mathcal I(\mathcal P) r_k = f_k^\text{fin} + f_k^\text{rem} - \mathcal R r_k.
\end{equation*}
Invoking Remark \ref{ParticularDecomposition}, we have $r_k = r_{k+1} + v_{k} + w_{k}$, where
\begin{equation*}
\begin{gathered}
r_{k+1} = \mathcal G\Big(f_k^\text{rem} +\mathcal R r_k \Big)
\in \mathcal A_{\delta+(k+1)\gamma}(M),
\\
v_{k} = \mathcal G\Big(f_k^\text{fin}\Big)
\in \mathcal A_\phg^{T_k \cup (C+r)}(M) \cap \mathcal A_{\delta+k\gamma}(M),
\\
\text{ and }\quad
w_{k}\in \mathcal A_\phg^{C+r}(M).
\end{gathered}
\end{equation*}
We set $u_{k+1} = u_k + v_{k} + w_{k}$ so that $u = u_{k+1} + r_{k+1}$.
Let $S_{k+1} = S_k \cup T_k$ so that $u_{k+1}\in \mathcal A_\phg^{S_{k+1}}(M)$.
Since $r_k$, $r_{k+1}$, and  $v_{k}$ are in $\mathcal A_{\delta+k\gamma}(M)$, we have $w_{k}\in \mathcal A_{\delta+k\gamma}(M)$ and therefore $u_{k+1} - u_k$ is in the same space.
This ensures that neither the exponents nor the log terms accumulate.

Finally, note that
\begin{equation*}
f_{k+1} 
= f_k - \mathcal I(\mathcal P)v_k
-\mathcal I(\mathcal P)w_k - \mathcal R(v_k + w_k).
\end{equation*}
By construction (see Proposition \ref{Proposition-DefineG}), we have
\begin{equation*}
\mathcal I(\mathcal P) v_k - f_k \in \mathcal A_{\delta+(k+1)\gamma}(M).
\end{equation*}
The remaining terms in $f_{k+1}$ are easily seen to be in $\mathcal A_{\delta+(k+1)\gamma}(M)$, which completes the proof.
\end{proof}

\subsection{Boundary regularity for nonlinear equations}
\label{S-NL}
The methods above can also be used to study the boundary regularity of solutions to many nonlinear elliptic equations.
Here we illustrate this by showing that solutions to the Lichnerowicz equation \eqref{SimpleLichPhi} are polyhomogeneous when the metric and coefficient functions are polyhomogeneous; see e.g.~\cite{AnderssonChrusciel-Dissertationes, AnderssonChruscielFriedrich, ChruscielDelayLeeSkinner,LeeMelrose-MongeAmpere,Mazzeo-Edge}
for other results of this nature.

We suppose that $g\in \WAH^{2,\alpha;1}$ with $\rho^2 g \in C^2_\phg(\bar M)$ and that the functions $A,B$ appearing in \eqref{SimpleLichPhi} are in $\rho C^0_\phg(\bar M)$.
Let $\phi$ be the solution to \eqref{SimpleLichPhi} guaranteed by the first part of Proposition \ref{SolveLichPhi}.
Note that $u=\phi-1\in C^{k}_1(M)$ for all $k$.

Setting $\mathcal P = \Delta_g -(n+1)$, we see that $u $ satisfies an equation of the form 
\begin{equation}
\label{nl-basic}
\mathcal P u = f(u)
\end{equation}
for some function $f$.
Since $u$ vanishes near $\partial M$, there exists some $\rho_*>0$ such that on the collar neighborhood of the boundary $\mathcal C_{\rho_*}$ the function $f$ may be represented by a uniformly and absolutely convergent series   
\begin{equation}
\label{nl-expansion}
f(u) = \sum_{l=0}^\infty a_l u^l
\end{equation}
with coefficient functions satisfying
\begin{equation}
\label{nl-coefficients}
\begin{gathered}
a_0, a_1 \in \rho C^0_\phg(\bar M),
\quad \text{ and }\quad
a_l \in C^0_\phg(\bar M),\quad l\geq 2.
\end{gathered}
\end{equation}
The polyhomogeneity of $u$, and hence $\phi$, is a consequence of the following.
\begin{proposition}
\label{Lich-BoundaryRegularity}
Suppose that $g\in \WAH^{2,\alpha;1}$ with $\rho^2 g \in C^2_\phg(\bar M)$, that $u$ satisfies \eqref{nl-basic} and 
$u\in C^k_1(M)$ for all $k\geq 0$,
and that $f$ is a function satisfying \eqref{nl-expansion} and \eqref{nl-coefficients} in a collar neighborhood of the boundary.
Then $u$ is polyhomogeneous.
\end{proposition}

\begin{proof}
We divide the proof in to two parts, first showing that the solution is conormal and subsequently showing that it is polyhomogeneous. 

In order to show $u$ is conormal, we adapt the proof of Proposition \ref{SolutionIsConormal}.
For any $V\in \mathscr V_b$, we see that $\mathcal L_V a_0, \mathcal L_V a_1 \in \rho C^0_\phg(\bar M)$ and $\mathcal L_V a_l \in C^0_\phg (\bar M)$ if $l\geq 2$.
If $w\in C^k_1(M)$ we have $\mathcal L_V w \in C^{k-1}_0(M)$; since $f(u)\in C^k_1(M)$ by \eqref{nl-expansion} and \eqref{nl-coefficients} it follows that $\mathcal L_V f(u) \in C^{k-1}_1(M)$.
Consequently, fixing $V\in \mathscr V_b$, we find that \eqref{EstimateDifferenceQuotient} holds with $\delta =1$ and $f$ replaced by $f(u)$. The subsequent argument shows that $\mathcal L_Vu \in C^k_1(M)$ for all $k\geq 0$.

Proceeding inductively, we assume that for any $\{V_1,\dots, V_m\} \subset \mathscr V_b$ we have $w_m = \mathcal L_{V_1}\cdots \mathcal L_{V_m}u \in C^k_1(M)$ for all $k\geq 0$. Fix $V\in \mathscr V_b$.
The properties \eqref{nl-coefficients} imply that
\begin{equation*}
\mathcal L_V \mathcal L_{V_1}\cdots \mathcal L_{V_m} f(u) \in C^k_1(M)
\end{equation*}
for all $k\geq 0$, and thus we may proceed exactly as in the proof of Proposition \ref{SolutionIsConormal}, establishing that $u$ is conormal.

To see that $u$ is polyhomogeneous, we adapt the proof of Theorem \ref{SolutionsArePolyhomogeneous}, constructing inductively an approximating sequence $u_k\in \mathcal A_\phg^{S_k}(M)\cap \mathcal A_1(M)$ for finite $S_k\subset \mathbb C$, such that $f_k = f(u_k)- \mathcal P u_k \in \mathcal A_{1+k\gamma}(M)$ and such that $r_k = u-u_k \in \mathcal A_{1+k\gamma}(M)$ with $r_{k+1} - r_k \in \mathcal A_{1+k\gamma}(M)$ when $k$ is large.
We may assume that $\gamma \in (0,1]$.

Setting $u_0 =0$, the properties \eqref{nl-coefficients} imply that there is nothing to prove; as before, we set $S_0 = C$, the 
set of characteristic exponents of $\mathcal P$.

Working under the inductive hypothesis, we see that the remainder $r_k = u-u_k$ satisfies
\begin{equation*}
\mathcal P r_k = f_k + f(u) - f(u_k).
\end{equation*}
Since $u_k\in \mathcal A_\phg^{S_k}(M) \cap \mathcal A_1(M)$ we have $f_k\in \mathcal A_\phg(M)$.
Furthermore, by inductive assumption we have $f_k\in \mathcal A_{1+k\gamma}(M)$.
Thus we may write $f_k = f_k^\text{fin} + f_k^\text{rem}$, where $f_k^\text{fin}\in \mathcal A_\phg^{T_k}(M)\cap \mathcal A_{1+k\gamma}(M)$ for some finite set
\begin{equation*}
T_k \subset \{ s \mid 1+k\gamma \leq \Re(s)\leq 1+(k+1)\gamma\}
\end{equation*}
and $f_k^\text{rem}\in \mathcal A_{1+(k+1)\gamma}(M)$.

We now analyze the difference $f(u) - f(u_k)$.
By assumption, $r_k = u-u_k\in \mathcal A_{1+k\gamma}(M)$, and thus $a_1 (u-u_k)\in \mathcal A_{1+(k+1)\gamma}(M)$.
For $l\geq 2$ we have $u^l - u_k^l\in \mathcal A_{1+(k+1)\gamma}(M)$, and therefore $f(u) - f(u_k)\in \mathcal A_{1+(k+1)\gamma}(M)$.
Re\-writing the equation as
\begin{equation*}
\mathcal I(\mathcal P) r_k = f_k^\text{fin}
+ f(u) - f(u_k) - \mathcal R r_k
\end{equation*}
and applying Remark \ref{ParticularDecomposition} we
obtain $r_k = r_{k+1} + v_k + w_k$, where
\begin{gather*}
r_{k+1} = \mathcal G\Big(f_k^\text{rem} + f(u) - f(u_k)- \mathcal R r_k\Big)\in \mathcal A_{1+(k+1)\gamma}(M),
\\
v_k = \mathcal G \Big(f_k^\text{fin}\Big)\in \mathcal A_\phg^{T_k\cup C}(M)\cap \mathcal A_{1+k\gamma}(M),
\\
\text{ and }\quad
w_k \in \mathcal A_\phg^C(M).
\end{gather*}
Proceeding as in the proof of Theorem \ref{SolutionsArePolyhomogeneous} completes the argument.
\end{proof}

\section{Corrections to \cite{Lee-FredholmOperators}}
\label{CorrectFred}

As pointed out to us by David Maxwell, Lemma 3.4(a) in \cite{Lee-FredholmOperators} is incorrectly stated.
The following corrections need to be made to \cite{Lee-FredholmOperators}:

\begin{description}
\item[Page 16]
The formula displayed in Lemma 3.4(a) should be replaced by
\begin{equation*}
\left( \sum_{0\leq j\leq k}\big\|\rho^{-\delta}\nabla^j u\big\|_{0,p;U}^p \right)^{1/p}.
\end{equation*}

\item[Page 17]
Inequality (3.6) should be replaced by
\begin{equation*}
C^{-1}\sum_i \rho(p_i)^{-\delta p} \big\|\Phi_i^*u\big\|^p_{k,p;B_r}
\leq
\|u\|^p_{k,p,\delta}
\leq 
C\sum_i \rho(p_i)^{-\delta p} \big\|\Phi_i^*u\big\|^p_{k,p;B_r}.
\end{equation*}
The proof of Lemma 3.5 can then be readily corrected
by inserting the exponent $p$ in appropriate places.

\item[Page 25]
Two of the formulas near the top of the page need to be changed 
as follows:
\begin{itemize}
\item[(c)]
$u \mapsto u \otimes \rho^2 g$;
\item[(d)]
$u\mapsto u \otimes \rho^{-2}g^{-1}$.
\end{itemize}

\item[Page 29]
Each term in the first series of inequalities should be raised to the $p$th power.

\item[Page 30]
Each term in the first series of inequalities should be raised to the $p$th power;
the second inequality is then justified by 
using the elementary inequality $(a+b)^p\le 2^{p-1}(a^p+b^p)$.

\item[Page 48]
Lines 7--12 should be replaced by the following:

\noindent
Then Lemma 6.1 implies that $P_i$ is close to $\breve P$ in the following sense:
For each $\delta\in\mathbb R$, $0<\alpha<1$, $1<p<\infty$, and $k$ such that $m\le k\le l$ 
and $m<k+\alpha\le l+\beta$, there is a constant $C$ (independent of $r$ or $i$)
such that for all compactly supported $u\in C^{k,\alpha,\delta}(Y_1,\breve E)$,
\begin{equation}\tag{6.5}
\|P_i - \breve P\|_{k-m,\alpha,\delta} \le Cr\|u\|_{k,\alpha,\delta},
\end{equation}
and for all compactly supported $u\in H^{k,p,\delta}(Y_1,\breve E)$,
\begin{equation}\tag{6.6}
\|P_i - \breve P\|_{k-m,p,\delta} \le Cr\|u\|_{k,p,\delta}.
\end{equation}

\item[Page 49] On line 13 from the bottom, ``Proposition 5.8" should be ``Proposition 5.6"; 
and on line 11 from the bottom, ``(6.5) implies (6.8)" should 
be ``(6.6) implies (6.8)."

\end{description}

\bibliographystyle{plain}
\bibliography{WAH}

\end{document}